\documentclass[11pt,letterpaper]{amsart}

\usepackage[utf8]{inputenc}

\usepackage{graphicx} % Required for inserting images
\usepackage{amssymb,amsmath,amsthm}
\usepackage{mathtools}
\usepackage[dvipsnames]{xcolor}
\usepackage{yhmath}%widehat and tilde piu` wide
\usepackage{enumerate}
\usepackage{xspace}%\newcommand{\iacc}{{\`\i\xspace}}
\usepackage{amsthm,amsfonts,amsxtra,amscd}
\usepackage[all]{xy}
\usepackage{verbatim}
\usepackage{xr}
\usepackage{hyperref}
\usepackage{tikz-cd}
\usepackage{csquotes}
\usepackage{ragged2e}
\usepackage{blindtext}

\newcommand{\calVg}{\calV^\kappa_{\Sigma}(\gog)}
\newcommand{\oLie}{\overline{\Lie}}

\newcommand{\gogl}{\check{\gog}}

\newcommand{\vac}{|0\rangle}
\newcommand{\Zcrit}{Z_{\kappa_c}(\hat{\gog}_{\Sigma})}
\newcommand{\Vcan}{V^{\can}}
\newcommand{\oSigma}{{ \overline \Sigma}}

\newcommand{\uch}{u}
\newcommand{\ulie}{\mathbf{u}}

\newcommand{\fsonoz}{fsonoz\xspace}

\newcommand{\VkSg}{\calV^\kappa_{\Sigma}(\gog)}
\newcommand{\VkSgstar}{\calV^\kappa_{\overline \Sigma^*}(\gog)}
\newcommand{\calVstar}{\calV_{\overline \Sigma^*}}

\newcommand{\Vkg}{V^\kappa(\gog)}

\newcommand{\vbasic}{\calV_{\text{basic}}}
\newcommand{\vbasicpiu}{\calV_{\text{basic}+}}
\newcommand{\vcom}{\calV_{\text{com}}}
\newcommand{\vcompiu}{\calV_{\text{com}+}}

\newcommand{\am}[1]{{\color{blue}\textsf{[[AM: #1]]}}}
\newcommand{\lc}[1] { {\color{orange} \textsf{[[LC: #1]]} } }

\DeclareMathOperator{\cont}{cont}

\newcommand{\QCC}{QCC\xspace}

\newcommand{\QCCF}{QCCF\xspace}

\newcommand{\tensor}[1]{{\stackrel{#1}{\otimes}}}
\newcommand{\exttensor}[1]{{\stackrel{#1}{\boxtimes}}}
\newcommand{\calHom}{\operatorname{\calH\! \it{om}}}

\newcommand{\calHomcont}{\operatorname{\calH\! \it{om}}^{\cont}}

\newcommand{\Homcont}{\operatorname{Hom}^{\cont}}

\newcommand{\piuno}{\calR}

\newcommand{\pbarO}{{\calO_{\overline{ \Sigma}}}}
\newcommand{\pbarOpoli}{{\calO_{\overline{\Sigma}^*}}}

\newcommand{\pbarOmega}{{\Omega^1_{\overline{\Sigma}}}}    
\newcommand{\pbarOmegapoli}{{\Omega^1_{\overline{\Sigma}^*}}}
\newcommand{\pbarD}{\calD_{\overline{\Sigma}}}
\newcommand{\pbarDpoli}{{\calD_{\overline{\Sigma}^*}}}

\newcommand{\pbarOm}[1]        {\calO^{#1}_{\overline{\Sigma}}}
\newcommand{\pbarOpolim}[1]    {\calO^{#1}_{\overline{\Sigma}^*}}
\newcommand{\pbarDm}[1]        {\calD^{#1}_{\overline{\Sigma}}}

\newcommand{\unoU}{{\mathbf{1}}}

\newcommand{\Tan}{T}
\DeclareMathOperator{\Res} {Res}

\newcommand{\pbarOpoliq}{{\calO^2_{\overline{\Sigma}^*}}}

\newcommand{\otimesr}{\overrightarrow{\otimes}}
\newcommand{\otimesl}{\overleftarrow{\otimes}}
\newcommand{\otimesst}{\stackrel{*}{\otimes}}
\newcommand{\otimessh}{\stackrel{!}{\otimes}}

\newcommand{\surjmap}{\twoheadrightarrow}
\newcommand{\op}{\text{op}}

\newcommand{\gsig}[1]{\hat{\gog}_{\Sigma,#1}}
\newcommand{\ugsig}[1]{\calU_{#1}(\hat{\gog}_\Sigma)}
\newcommand{\Zenv}[2]{Z_{#1}(\hat{\gog}_{#2})}

\DeclareMathOperator{\ffset}{ffSet}

\DeclareMathOperator{\tmod}{-mod}
\DeclareMathOperator{\Aut}{Aut}
\newcommand{\Autzero}[1]{\Aut^+_{#1}}
\newcommand{\Autpiu}[1]{\Aut^0_{#1}}
\DeclareMathOperator{\Aff}{Aff}
\DeclareMathOperator{\Sch}{Sch}

\DeclareMathOperator{\Sp}{Sp}
\DeclareMathOperator{\Fun}{Fun}

\theoremstyle{plain}
\newtheorem{lemma}{Lemma}[subsection]
\newtheorem{theorem}[lemma]{Theorem}
\newtheorem{proposition}[lemma]{Proposition}
\newtheorem{corollary}[lemma]{Corollary}

\theoremstyle{definition}
\newtheorem{definition}[lemma]{Definition}
\newtheorem{remark}[lemma]{Remark}
\newtheorem{example}{Example}

\theoremstyle{remark}
\newtheorem*{remarkno}{Remark}
\theoremstyle{plain}
\newtheorem*{theoremno}{Theorem}
\theoremstyle{definition}

\theoremstyle{remark}
\newtheorem{ntz}[lemma]{Notation}

\newcommand{\mA}{\mathbb A}
 
\newcommand{\mC}{\mathbb C}

\newcommand{\mF}{\mathbb F} 
\newcommand{\mG}{\mathbb G}

\newcommand{\mN}{\mathbb N}

\newcommand{\mV}{\mathbb V}

\newcommand{\mZ}{\mathbb Z}

\newcommand{\calA}{\mathcal A}
\newcommand{\calB}{\mathcal B} 
 
\newcommand{\calD}{\mathcal D}
\newcommand{\calE}{\mathcal E} 
\newcommand{\calF}{\mathcal F} 
\newcommand{\calG}{\mathcal G}
\newcommand{\calH}{\mathcal H} 
\newcommand{\calI}{\mathcal I} 
\newcommand{\calJ}{\mathcal J}
\newcommand{\calK}{\mathcal K} 
\newcommand{\calL}{\mathcal L} 
\newcommand{\calM}{\mathcal M}
 
\newcommand{\calO}{\mathcal O}

\newcommand{\calR}{\mathcal R} 

\newcommand{\calT}{\mathcal T} 
\newcommand{\calU}{\mathcal U} 
\newcommand{\calV}{\mathcal V}
\newcommand{\calW}{\mathcal W} 
 
\newcommand{\calY}{\mathcal Y}

\newcommand{\goF}{\mathfrak F}

\newcommand{\gob}{\mathfrak b}

\newcommand{\gog}{\mathfrak g}
\newcommand{\goh}{\mathfrak h}

\newcommand{\gol}{\mathfrak l} 
\newcommand{\gom}{\mathfrak m}

\newcommand{\gos}{\mathfrak s}

\newcommand{\gra}{\alpha} 
\newcommand{\grb}{\beta}

\newcommand{\gre}{\varepsilon}

\newcommand{\grf}{\varphi}

\newcommand{\ra}       {\rightarrow}

\newcommand{\lra}      {\to}
\newcommand{\isocan}   {\simeq}
\renewcommand{\geq}    {\geqslant}%vuole amssymb
\renewcommand{\leq}    {\leqslant}%vuole amssymb
         \newcommand{\mand}     {\text{ and }}

\DeclareMathOperator{\Hom}  {Hom}
\DeclareMathOperator{\End}  {End}
\DeclareMathOperator{\Der}  {Der}
\DeclareMathOperator{\Sym}  {Sym}

\DeclareMathOperator{\Spec} {Spec}
\DeclareMathOperator{\Qco}  {QCoh}

\DeclareMathOperator        {\Lie}{Lie}
\DeclareMathOperator        {\Liepoli}{Lie_{\bar \Sigma^* }}

\newcommand{\limind}{\varinjlim}
\newcommand{\limpro}{\varprojlim}
\DeclareMathOperator{\Set}{Set}

\DeclareMathOperator{\id}{id}
\DeclareMathOperator{\univ}{univ}
\DeclareMathOperator{\triv}{triv}

\DeclareMathOperator{\can}{can}
\DeclareMathOperator{\ev}{ev}
\DeclareMathOperator{\loc}{-loc}

\DeclareMathOperator{\Ran}{\mathbf{Ran}}
\DeclareMathOperator{\fact}{\mathbf{fact}}
\DeclareMathOperator{\Op}{Op}
\DeclareMathOperator{\Conn}{Conn}
\DeclareMathOperator{\fil}{fil}

\DeclareMathOperator{\ad}{ad\,}

\usepackage{xcolor}
\usepackage{soul}

\usepackage{fullpage}
\usepackage[toc]{appendix}
\newcommand{\lowOS}{\calO_S}

\title{The factorizable Feigin-Frenkel center}
\author{Luca Casarin, Andrea Maffei}

\begin{document}

\maketitle
\begin{center}
To George Lusztig, with admiration
\end{center}

\bigskip

\begin{flushright}
	\begin{minipage}{0.45\textwidth} % Modifica 0.5 per cambiare la larghezza (es. 0.4 o 0.6)
		\begin{it}{La quinta luce, ch'è tra noi più bella,
			
			spira di tale amor, che tutto 'l mondo
			
			là giù ne gola di saper novella:
			
			\smallskip
			
			entro v'è l'alta mente u' sì profondo
			
			saver fu messo, che, se 'l vero è vero,
			
			a veder tanto non surse il secondo}\end{it}
		
		\vspace{2mm} 
		\begin{small}(Dante, La Divina commedia, Paradiso, canto X)\end{small}
	\end{minipage}
\end{flushright}

\medskip

\begin{center}
	\textbf{Abstract}
	
	\smallskip
	\justifying
	\noindent We prove a factorizable version of the Feigin-Frenkel theorem on the center of the completed enveloping algebra of the affine Kac-Moody algebra attached to a simple Lie algebra at the critical level. On any smooth curve $C$ we consider a sheaf of complete topological Lie algebras whose fiber at any point is the usual affine algebra at the critical level and consider its sheaf of completed enveloping algebras. We show that the center of this sheaf is a factorization algebra and establish that it is canonically isomorphic, in a factorizable manner, with the factorization algebra of functions on opers for the Langlands dual Lie algebra on the pointed disk. 
\end{center}

\tableofcontents

\section{Introduction}

In this paper we prove a factorizable version of the Feigin-Frenkel Theorem on the center of the completed enveloping algebra of  the affine Kac-Moody algebra at the critical level attached to a simple Lie algebra $\gog$.
So, let us start the introduction by recalling the Feigin-Frenkel Theorem and saying some more words on what it means to prove a factorizable version of it. For convenience, we state it in its coordinate free version, so we consider a formal disk $D$ without a specified coordinate and its pointed version $D^*$. These appear, for instance, as formal disks around a point $x \in C$ of a smooth curve, where one can consider $\calO_x := \varprojlim_n \calO_{C,x}/\gom^n_x$, $D_x = \Spec \calO_x$ and $D^*_x = \Spec \mathrm{Frac}(\calO_x)$. 

Let $\gog$ be a finite simple Lie algebra and let $\hat{\gog}_{D,\kappa}$  be the affine Kac-Moody algebra attached to $D$ and $\gog$ of level $\kappa$: if $\calK$ is the (complete topological) ring of functions on $D^*$ then $\hat{\gog}_{D,\kappa} = \gog\otimes\calK \oplus \mC\mathbf{1}$, with the usual commutation formulas. Recall also the definition of $\Op_\gog(C)$, the space of $\gog$-opers on a smooth curve $C$ (also $C = D,D^*$, are allowed), which, given a Borel $B \subset G$ of the adjoint group $G$ relative to $\gog$, classifies certain connections on $G$-torsors on $C$. For the precise definitions and related constructions we refer to \cite[Sec. 2, Sec. 4]{frenkel2007langlands} and to \cite{BDopers}.

	\begin{theoremno}[Feigin-Frenkel \cite{feigin1992affine}]\label{thm:knownfeiginfrenkel}
		Let $\gog$ be a simple Lie algebra over $\mC$, let $\hat{\gog}_{D,\kappa_c}$ be the affine algebra at the critical level $\kappa_c = -\frac{1}{2}\kappa_{\gog}$ (here $\kappa_{\gog}$ is the non-normalized Killing form) attached to $D$ and let $Z_{\kappa_c}(\hat{\gog}_D)$ be the center of its completed enveloping algebra. There is a canonical isomorphism 
		\[
			Z_{\kappa_c}(\hat{\gog}_D) = \mC[\Op_{\gogl}(D^*)],
		\]
		where $\gogl$ is the Langlands dual Lie algebra of $\gog$, while $\mC[\Op_{\gogl}(D^*)]$ is the algebra of functions on the space of $\gogl$-opers over the pointed disk. 
	\end{theoremno}
    
    This result was conjectured by Drinfeld and was proved by Feigin and Frenkel in \cite{feigin1992affine}. A different proof was given in \cite{frenkeladvances} (see also \cite{frenkel2007langlands} for an exposition of this proof). Finally a more geometric proof was given by Raskin \cite{raskinproof}.

    In the case of $\gos\gol(2)$ a generalization of this Theorem was given in \cite{fortuna2022local} in the case of two singular points. This result was generalized to the case of arbitrary $\gog$ and an arbitrary number of points in \cite{cas2023}. In this paper we give generalizations of the results in \cite{cas2023} to a global geometric setting and in a coordinate independent manner.

	The Feigin-Frenkel theorem may be considered as a point-wise statement of a more general picture: for any smooth complex curve $C$ there exists a canonically defined ind-affine scheme over $C$, $\Op_{\gogl}(D^*)_C \to C$ (see Definition \ref{def:opersCI}) whose fiber at any point $x \in C$ identifies with $\Op_{\gogl}(D^*_x)$ (here $D^*_x$ is the pointed formal disk attached to $x \in C$). Similarly, it is possible to construct a sheaf of complete topological Lie algebras $\hat{\gog}_{C,\kappa_c}$ on $C$ (see Section 2.3 for its definition; with respect to the notation found there, $S=C$,$X=C^2$, $\Sigma = \{\Delta : C \to C^2\}$, for the notation $S,X,\Sigma$ see Section 2.1)
	\begin{comment}(c.f. Definition \ref{def:constructiongogsigma} in the case when $S = C,X = C^2$ and $\Sigma$ consists of the diagonal section $\Delta : C \to C^2$)
	\end{comment} 
	whose fiber at any point $x \in C$ coincides 
	with the affine algebra $\hat{\gog}_{D_x,\kappa_{c}}$, and consider $Z_{\kappa_c}(\hat{\gog}_C)$, the center of its sheaf of completed enveloping algebras at the critical level. Part of our main Theorem \ref{thm:teofinalefinale}, which actually follows almost directly from the Feigin-Frenkel theorem (c.f. Proposition \ref{cor:feiginfrenkel1section}), states that there is a canonical isomorphism between $Z_{\kappa_c}(\hat{\gog}_C)$ and $\Fun\left(\Op_{\check{\gog}}(D^*)_C\right)$, the sheaf (on $C$) of functions on $\Op_{\check{\gog}}(D^*)_C$.

	There is actually even more to this story. Indeed, the space $\Op_{\gogl}(D^*)_C$ has a natural \emph{factorization structure} (see \cite[Ch. 3.4]{BDchirali}), which we now explain briefly.

	\subsection*{Factorization}

	The factorization picture concerns spaces or sheaves over finite powers of a smooth curve $C$ and some relationship between them over appropriate sub-varieties. Before explaining this we first take the opportunity to introduce some notation which we will keep for the rest of the work. 

	\begin{ntz}[Notations for finite sets]\label{ssez:notazionifiniteset}
	To avoid possible set theoretical issues we fix the following setting. We will denote by $\ffset$ the category of finite non-empty subsets of $\mN_{\geq 1}$ with surjective maps as morphisms. We will refer to elements $I \in \ffset$ simply as finite sets. The main properties of $\ffset$ that we are interested in are that it is itself a set and that any subset of an object in $\ffset$ is still in $\ffset$. Given a surjective map of finite sets $\pi : J \twoheadrightarrow I$ and a subset $I' \subset I$ we will write $J_{I'}$ for $\pi^{-1}(I')$, when $I' = \{i\}$ we will write simply $J_i$ \index{$J_i$}. Given two surjections $q : K \twoheadrightarrow J, p : J \twoheadrightarrow I$ and an element $i \in I$ we write $q_i: K_i \twoheadrightarrow J_i$ for the associated projection on the fibers.   	
	\end{ntz}

	\smallskip

	\begin{ntz}[Notations for diagonals]\label{sssez:notazionidiagonali}
	We introduce some notations to describe diagonals of products of a scheme $X$ over a base scheme $S$ (so products are to be understood relative to $S$), we also assume that the structural morphism $X \to S$ is separated. We denote by $\Delta:X\lra X^2$ the diagonal map. 
	More generally, given a finite set $I$, we define the small diagonal $\Delta(I):X\lra X^I$, and given $\pi :J\surjmap I$, a surjective map of finite sets, we define the diagonals $\Delta(\pi) 
 = \Delta_{J/I}:X^I\lra X^J$ \index{$\Delta(\pi),\Delta_{J/I}$} in the obvious way, so that, for instance, given surjections $K \twoheadrightarrow J \twoheadrightarrow I$ we have $ \Delta_{K/I} = \Delta_{K/J}\Delta_{J/I}$. If $a,b\in J$ we denote by $\pi_{a,b}$ or by $\pi^J_{a=b}$ a surjection from $J$ to a set with one element less than $J$ such that $\pi_{a=b}(a)=\pi_{a=b}(b)$, \index{$\pi_{a=b}$} and by $\Delta_{a=b}$ \index{$\Delta_{a=b}$} the corresponding diagonal. In many constructions the fact that $\pi_{a=b}$ is not fixed in a unique way will not be a problem, we will specify the map only when necessary.

	By a slight abuse of notation we write $\Delta_{J/I}$ also for the closed subscheme of $X^J$ determined by $\Delta_{J/I}$. Given a finite set $I$ we define the big diagonal $\nabla(I)$ of $X^I$ as the scheme theoretic union of all diagonals $\Delta_{a=b}$ and more generally, given $\pi:J\surjmap I$, a surjective map of finite sets, we set
	$$
	\nabla(J/I)=\bigcup_{\pi(a)\neq \pi(b)}\Delta_{a=b}\subset X^J
	$$
	\index{$\nabla(J/I)$}so that $\nabla(I)=\nabla(I/I)$. We denote by $j_{J/I} : U_{J/I} = X^J \setminus \nabla(J/I) \hookrightarrow X^J$ the associated open immersion. Finally, given two disjoint subsets $A$ and $B$ of $J$ we define 
	$$
	\nabla_{A=B}=\bigcup_{a\in A, b\in B}\Delta_{a=b}
	$$
	and write $j_{A\neq B} : U_{A \neq B} = X^I \setminus \nabla_{A=B} \hookrightarrow X^I$ for the associated open immersion.

	\end{ntz}

\smallskip

Using the above notation, we recall the definition of a factorization algebra given in \cite[Sec. 3.4]{BDchirali}.
\begin{definition}\label{def:factorizationspace}
		A \emph{pseudo-factorization algebra} $\calA$ on a smooth curve $C$ is the datum of a collection of quasi-coherent sheaves $\calA_I \in \Qco(C^I)$ together, for any surjection of finite sets $\pi : J \twoheadrightarrow I$, with morphisms
		\begin{align*}
			\Ran^{\calA}_{J/I} &: \Delta_{J/I}^*\left( \calA_J \right) \rightarrow \calA_I \\
			\fact^{\calA}_{J/I} &: j_{J/I}^*\left( \boxtimes_{i \in I} \calA_{J_i} \right) \rightarrow j_{J/I}^*\left( \calA_J \right)
		\end{align*}
		where $\Delta_{J/I}^*$ and $j_{J/I}^*$ denote the pullback along $\Delta_{J/I} : C^I \to C^J$ and $j_{J/I} : C^J \setminus \nabla(J/I) \to C^J$ respectively. We require these morphisms to be compatible with one another when considering surjections $K \twoheadrightarrow J \twoheadrightarrow I$ as in \cite[3.4.4]{BDchirali}, that is:
		\begin{itemize}
			\item $\Ran^{\calA}_{K/I} = \Ran^{\calA}_{J/I}\circ\left(\Delta_{J/I}^*\Ran^{\calA}_{K/J}\right)$;
			\item $\fact^{\calA}_{K/J} = \left( j^*_{K/J} \fact^{\calA}_{K/I} \right) \circ \left(j_{K/I}^* \boxtimes_{i \in I} \fact^{\calA}_{K_i/J_i} \right)$ ;
			\item $ \left( j_{J/I}^*\Ran^{\calA}_{K/J} \right) \circ \left( \Delta_{K/J}^*\fact^{\calA}_{K/I} \right) = \fact^{\calA}_{J/I} \circ \left(j_{J/I}^*\boxtimes_{i \in I} \Ran^{\calA}_{K_i/J_i}\right) $. 
		\end{itemize}
		We say that the collection $\calA_I$ is a \emph{factorization algebra} if the morphisms $\Ran$ and $\fact$ are isomorphisms.
	\end{definition}

	\begin{remarkno}
		In \cite[3.4]{BDchirali} there are further assumptions regarding torsion along the diagonals for the sheaves $\calA_I$. These guarantee that the forgetful functor from factorization algebras to quasi-coherent sheaves on $C$ is faithful and therefore allow one to talk about \emph{factorization structures} on $\calA = \calA_C \in \Qco(C)$. We won't stress this point of view in this work since the factorization algebras we will consider are (topologically) free. We will use the term \textquote{factorization structure} to denote the datum of the maps $\Ran$ and $\fact$ when we are already given a family of quasi coherent sheaves $\calA_I$ on $C^I$.
	\end{remarkno}

	It is easy to adapt this definition to other cases as well. For instance, one may consider spaces over $C^I$ in place of quasi-coherent sheaves and replace the product $\boxtimes$ with the product of spaces, thus obtaining the notion of a \emph{factorization space}. Our case of interest is the one of completed quasi-coherent sheaves (see \cite[Section \ref{app:top}]{casmaffei1}) with the completed external tensor product $\boxtimes^!$ (see Section \ref{ssec:topologytensorprodsheaves}), we call them \emph{complete topological (pseudo) factorization algebras}.

    \smallskip

	The space $\Op_{\gogl}(D^*)_C$ is actually part of a factorization ind-affine scheme $\Op_{\gogl}(D^*)_{C^I}$, hence the collection of sheaves of functions $\Fun(\Op_{\gogl}(D^*)_{C^I})$ is naturally a  complete topological factorization algebra.

	On the other hand one can construct a factorizable version of the affine algebra $\hat{\gog}_{C,\kappa_c}$ and show that the associated collection $\calU_{\kappa_c}(\hat{\gog}_{C^I})$ of sheaves of completed enveloping algebras is naturally a complete topological factorization algebra. Taking the center of these associative algebras we get a collection of complete topological commutative algebras $Z_{\kappa_c}(\hat{\gog}_{C^I})$ which inherits from $\calU_{\kappa_c}(\hat{\gog}_{C^I})$ the structure of a complete topological pseudo-factorization algebra. The main Theorem of this work is the following.

	\begin{theorem}\label{thm:mainteointro}
		Let $\gog$ be a simple finite dimensional Lie algebra over $\mC$ and $C$ a smooth (non necessarily proper) curve over $\mC$. Consider the collections $Z_{\kappa_c}(\hat{\gog}_{C^I}), \Fun(\Op_{\gogl}(D^*)_{C^I})$ as above. Then for any finite set $I$ there is a canonical isomorphism
		\[
			Z_{\kappa_c}(\hat{\gog}_{C^I}) = \Fun(\Op_{\gogl}(D^*)_{C^I})
		\]
		which is compatible with the pseudo-factorization structures. In particular the collection $Z_{\kappa_c}(\hat{\gog}_{C^I})$ is naturally a complete topological factorization algebra.
	\end{theorem}

The approach we take in dealing with this problem follows closely the strategy of \cite{cas2023} and heavily uses the foundational background developed in \cite{casmaffei1}. In particular, we exploit the theory of fields/distributions of \emph{loc. cit.} to construct central elements, which for us builds the bridge between the theory of usual vertex algebras and that of factorization algebras. An alternative proof, based on categorical arguments and on chiral algebras rather than fields/distributions, was pointed out to us by Sam Raskin; however, we could not find a precise reference for this approach.

\subsection*{Description of the Work}

\textbf{Disclaimer:} the present paper is a direct continuation of \cite{casmaffei1}, so we will use freely use the terminology and results exposed there. That includes the rest of this introduction, we proceed to briefly review the content of each section making references to the objects in \cite{casmaffei1}.

\smallskip

We start in Section \ref{sec:recollections} by reviewing some of the material of \cite{casmaffei1}, so we present our geometric setting and the basic objects we will work with for the rest of the paper.

We move forward in Section \ref{sec:envelopingalgchiralalg} by introducing $\Lie_{\overline{\Sigma}^*}(\calV)$ and $\calU_{\overline{\Sigma}^*}(\calV)$, the Lie algebra and the completed enveloping algebra attached to a chiral algebra $\calV$. We show that if $\calV \subset \mF^1_{\Sigma,\calU}$ is a chiral algebra of mutually local fields we have a natural morphism of Lie algebras $\Lie_{\overline{\Sigma}^*}(\calV) \to \calU$, which extends to a continuous morphism of associative algebras $\Phi : \calU_{\overline{\Sigma}^*}(\calV) \to \calU$. In the case where $\calU = \ugsig{\kappa}$ and $\calV = \calV^\kappa_{\Sigma}(\gog)$ we show that $\Phi$ is an isomorphism and deduce some factorization properties of the construction $\calU_{\overline{\Sigma}^*}$ in this case. These factorization properties induce a factorization structure on $\calU_{\overline{\Sigma}^*}(\zeta^{\kappa}_{\Sigma}(\gog))$ and $\Phi$ induces a continuous morphism $\Phi : \calU_{\overline{\Sigma}^*}(\zeta^{\kappa}_{\Sigma}(\gog)) \to Z_{\kappa}(\hat{\gog}_{\Sigma}) = Z(\ugsig{\kappa})$ which by construction is compatible with the (pseudo) factorization structure on both spaces.

\smallskip

Before getting to our main theorems, in section \ref{sec:opersigma} we adapt some classical, known constructions (for which we refer to \cite{casarin2025bundle}) to our geometric setting. in order to give the definition of $\gog$-opers on $\overline{\Sigma}$ and $\overline{\Sigma}^*$. Here we also review the construction of the spaces $\Op_{\gog}(D)_C,\Op_{\gog}(D^*)_C$ and their factorization enhancement $\Op_{\gog}(D)_{C^I}, \Op_{\gog}(D^*)_{C^I}$.

\medskip

Finally, in section \ref{sec:identificationopers2} we prove Theorem \ref{thm:mainteointro}. We actually focus on proving an analogous, slightly stronger statement in our $X,S,\Sigma$ setting (c.f. \ref{thm:teofinale1}) and then show how this specializes to Theorem \ref{thm:mainteointro} when $S = C^J, X = C \times C^J$ and $\Sigma = \Sigma^{\univ}_{C,J} = \{\sigma^{\univ}_{J,j} \}_{j \in J}$ is composed by the canonical sections $\sigma^{\univ}_{J,j}((x_{j'})_{j'\in J}) = (x_j, (x_{j'})_{j'\in J})$ in Section \ref{ssec:thefactorizablefeiginfrenkelcenter}.

The proof of Theorem \ref{thm:teofinale1} is composed by two different parts: showing that the map $\Phi^\zeta_{\Sigma} : \calU_{\overline{\Sigma}^*}\left( \zeta^{\kappa_c}_\Sigma(\gog)\right) \to Z_{\kappa_c}(\hat{\gog}_\Sigma)$ is an isomorphism and establishing a canonical isomorphism $ \gamma_{\Sigma} : \calU_{\overline{\Sigma}^*}\left(\zeta^{\kappa}_{\Sigma}(\gog)\right) = \Fun\left( \Op_{\gogl}(\overline{\Sigma}^*) \right)$. The main strategy of the proof is to use the factorization properties of our constructions to reduce to the case of a single section, which we essentially get for free from the Feigin-Frenkel Theorem. To show the statement about $\gamma$ we define $\gamma$ locally, after the choice of a coordinate, emulating the Feigin-Frenkel isomorphism. We prove that our construction is coordinate independent and that that $\Phi$ and $\gamma$ are isomorphism in the same way: by induction on the number of sections and using factorization properties. 

\smallskip

In Appendix \ref{sec:feiginfrenkelclassical} we review the classical Feigin-Frenkel theorem and make some remarks on its proof. We relate their constructions, as presented in \cite{frenkel2007langlands}, to ours. In this appendix we heavily use the terminology introduced in Section \ref{sec:envelopingalgchiralalg} and some of the terminology of Section \ref{sec:opersigma}; it should be read before getting to Section \ref{sec:identificationopers2}.

%\newpage

\subsection*{Acknowledgements}

Both authors want to thank Alberto De Sole for his support and the encouragement during the preparation of this work. The first author also wants to thank the University of Pisa for the hospitality during his visits in which part of this work was carried out. 

\textbf{Funding:} The first author was funded by national
PRIN Grants 2022S8SSW and 2022HMBTTL and by INFN - CSN4 (Commissione Scientifica Nazionale 4 - Fisica Teorica), MMNLP project.

\medskip

It is our great pleasure to dedicate this paper to George Lusztig and to take this opportunity to thank him for the beautiful mathematics he has created.

%%%%% QUI SEZIONI 2 e 3 SPARITE

\section{Recollections}\label{sec:recollections}
In this paper we will deal with sheaves equipped with a topology on a scheme $S$. We will use the formalism developed in \cite{casmaffei1} to treat chiral algebras and related constructions in this context. 
In this section we introduce some notations and we recall briefly some of the results of \cite{casmaffei1}. In what follows and in the rest of the paper we will freely use the notion of topological sheaf developed in \cite[Section \ref{app:top}]{casmaffei1}, without recalling it here. In \emph{loc. cit.} various topological tensor products are introduced: $\otimes^*,\otimesr,\otimes^!,\boxtimes^{\text{cg}}$; let us mention that there are various technical difficulties in treating these objects and that we deal with all of them in \emph{loc. cit.}. One can read the present paper without paying too much attention to these details, assuming that these construction are all well behaved.

\subsection{The geometric setting}\label{ssec:richiamigeometricsetting}

We will work over a fixed quasi-separated based scheme $S$, which we will assume to be topologically noetherian. For the final application we will also assume $S$ to be integral. We consider a fixed smooth family of curves $p : X \to S$ and $\Sigma = \{ \sigma_i : S \to X \}_{i \in I}$, a finite collection of sections of $p$. Attached to this datum we consider various sheaves of $\calO_S$-modules. We start by considering the locally free of rank $1$ ideal $\calI_\Sigma = \prod \calI_{\sigma_i} \subset \calO_X$, where $\calI_{\sigma_i}$ is the ideal defining the subscheme $\sigma_i(S) \subset X$. We consider the complete topological sheaves on $S$ (see \cite[Section \ref{app:top}]{casmaffei1})

\[
	\pbarO = \varprojlim_n p_*\big( \calO_X / \calI^n_\Sigma\big),\qquad \pbarOpoli = \varprojlim_n p_*\big( \calO_X(\infty\Sigma)/\calI^n_\Sigma \big)
\]
The topology on $\pbarO$ and $\pbarOpoli$ has a fundamental system of open neighborhoods given by $\pbarO(-n) = \ker( \pbarO \to p_*(\calO_X/ \calI_{\Sigma}^n))$.

For a given point $s \in S$ consider the surjection $\pi : J \twoheadrightarrow I$ determined by $\pi(j_1) = \pi(j_2)$ iff $\sigma_{j_1}(s) = \sigma_{j_2}(s)$. By \cite[Section \ref{ssec:descrizionelocale1}]{casmaffei1}, we may find an open neighborhood of $s$ such that for $i_1 \neq i_2 \in I$ the images of sections $\sigma_{j_1}$, for $j_1 \in J_{i_1} = \pi^{-1}(i_1)$ are pairwise disjoint from the images of sections $\sigma_{j_2}$ for $j_2 \in J_{i_2} = \pi^{-1}(i_2)$. It follows that in this case $\pbarOpoli \simeq \prod_{i \in I} \calO_{\overline{\Sigma}^*_{J_i}}$. 

In addition, we can find a function $t = (t_i)_{i \in I} \in \pbarO = \prod_{i \in I} \calO_{\overline{\Sigma}_{J_i}}$ and functions $a_j \in \calO_S$ such that

\begin{align*}
	\calO_{\overline{\Sigma}_{J_i}} &\simeq \varprojlim \frac{\calO_S[t_i]}{\left( \varphi_i^n\right)} \quad \text{where} \quad \varphi_i = \prod_{j\in J_i} (t_i- a_j), \\
	\calO_{\overline{\Sigma}^*_{J_i}} &\simeq \pbarO\left[\varphi_i^{-1}\right].
\end{align*}
We will call such a function $t$ a \emph{coordinate}, an open subset of $S$ which admits a coordinate is called \emph{well covered}. Notice that in the above setting it follows that, for two given indices $j_1, j_2 \in J$, the locus $\{\sigma_{j_1} = \sigma_{j_2}\} \subset S$ is either empty or principal (i.e. determined by $a_{j_1} - a_{j_2}=0$).

As in \cite[Section \ref{ssec:operatoridifferenziali}]{casmaffei1} we will also consider continuous modules of differentials $\Omega^1_{\overline{\Sigma}}$, $\Omega^1_{\overline{\Sigma}^*}$ which are locally isomorphic to $\pbarO \mathrm{d}t$ and $\pbarOpoli \mathrm{d}t$ after the choice of a coordinate $t$; as well as continuous sheaves of differential operators $\calD_{\overline{\Sigma}},\calD_{\overline{\Sigma}^*}$, which analogously are locally isomorphic to $\pbarO[\partial_t],\pbarOpoli[\partial_t]$. Finally, let us recall the residue map of \cite[Section \ref{sec:residuo}]{casmaffei1} $\Res_\Sigma : \Omega^1_{\overline{\Sigma}^*}/\Omega^1_{\overline{\Sigma}} \to \calO_S$ which is obtained as the sum of the usual residue map along all sections. 

All these constructions satisfy several \emph{factorization properties}.

\subsection{Factorization}\label{ssec:richiamifactorization}

Let us introduce here the analogue in our $X,S,\Sigma$ geometric setting of the factorization picture of Notation \ref{sssez:notazionidiagonali} and Definition \ref{def:factorizationspace}. Let $p : X \to S$ be a smooth family of curves as before and $\Sigma = \{ \sigma_i \}_{j \in J}$ a finite set of sections of $p$. For a surjection $\pi : J \twoheadrightarrow I$ we define two subschemes of $S$.

\begin{definition}\label{def:openclosedfactorization}
	Let $\pi : J \twoheadrightarrow I$ be a surjection. We define
	\begin{itemize}
		\item $i_{J/I} : V_\pi = V_{J/I} \to S$ the closed immersion corresponding to the closed subscheme $$V_{J/I} = \{ x \in S : \sigma_{j_1}(x) = \sigma_{j_2}(x) \text{ if } \pi(j_1) = \pi(j_2) \};$$
		\item $j_{J/I} : U_\pi = U_{J/I} \to S$ the open immersion corresponding to the open subscheme $$U_{J/I} = \{ x \in S : \sigma_{j_1}(x) \neq \sigma_{j_2}(x) \text{ if } \pi(j_1) \neq \pi(j_2) \}.$$
	\end{itemize}
\end{definition}

\begin{remark}
	Let $J_i= \pi^{-1}(i) \subset I$, then the following hold by construction:
	\begin{itemize}
		\item For any $j_1,j_2 \in J_i$ the sections $\sigma_{j_1},\sigma_{j_2} : V_{J/I} \to X \times_S V_{J/I}$ coincide. It follows that if we restrict our set of sections $\Sigma$ to $V_{J/I}$ we may identify it with a finite set of sections indexed by $I$. We will denote this new collection of sections by $\Sigma_I$;
		\item On $U_{J/I}$ the closed subschemes $\sigma_{j_1}(U_{J/I}),\sigma_{j_2}(U_{J/I}) \subset X\times_S U_{J/I}$ are disjoint if $\pi(j_1) \neq \pi(j_2)$. We will write $\Sigma_{J_i}$ for the sub-collection of sections indexed by $J_i$ in this case.
	\end{itemize}
\end{remark}

\noindent So when we restrict our attention to $V_{J/I}$ we move to the case where the finite collection of sections is indexed by $I$, while when we restrict our attention to $U_{J/I}$ we get a partition of $J$ where the sections belonging to each partition are disjoint. We are interested in studying how our constructions behave under restriction to $U_{J/I}$ and $V_{J/I}$.

\begin{definition}\label{def:factorizationcompletesheafsigma}
	Let $X,S$ be as above and let $\calA = \{ \calA_{\Sigma}\}$ be a collection of complete topological (often QCC) sheaves of $\calO_S$-modules indexed by families of sections $\Sigma : S \to X^I$ (we let $\Sigma$ and $S$ vary). A \emph{pseudo-factorization algebra structure} on $\calA$\index{pseudo-factorization algebra} is the datum, for any surjection of finite sets $\pi : J \twoheadrightarrow I$, of continuous morphisms of $\calO$-modules
	\begin{align*}
		\Ran^{\calA}_{J/I} &: \hat{i}_{J/I}^*\left( \calA_{\Sigma_J} \right) \rightarrow \calA_{\Sigma_I}, \\
		\fact^{\calA}_{J/I} &: \hat{j}_{J/I}^*\left( \bigotimes^!_{i \in I} \calA_{\Sigma_{J_i}} \right) \rightarrow \hat{j}_{J/I}^*\left( \calA_{\Sigma_J} \right),
	\end{align*}
	where $\hat{i}_{J/I}^*$ and $\hat{j}_{J/I}^*$ denote the completed pullback along $i_{J/I} : V_{J/I} \to S$ and $j_{J/I} : U_{J/I} \to S$. We require these morphisms to be compatible with one another when considering surjections $K \twoheadrightarrow J \twoheadrightarrow I$ in analogy with \cite[3.4.4]{BDchirali}, that is:
	\begin{itemize}
		\item $\Ran^{\calA}_{K/I} = \Ran^{\calA}_{J/I}\circ\left(\hat{i}_{J/I}^*\Ran^{\calA}_{K/J}\right)$;
		\item $\fact^{\calA}_{K/J} = \left( \hat{j}^*_{K/J} \fact^{\calA}_{K/I} \right) \circ \left(\hat{j}_{K/I}^* \bigotimes^!_{i \in I} \fact^{\calA}_{K_i/J_i} \right)$ ;
		\item $ \left( \hat{j}_{J/I}^*\Ran^{\calA}_{K/J} \right) \circ \left( \hat{i}_{K/J}^*\fact^{\calA}_{K/I} \right) = \fact^{\calA}_{J/I} \circ\left(\hat{j}_{J/I}^*\bigotimes^!_{i \in I} \Ran^{\calA}_{K_i/J_i}\right) $. 
	\end{itemize}
	We say that the collection $\calA_{\Sigma}$ is a \emph{completed topological factorization algebra} \index{completed topological factorization algebra}if the morphisms $\Ran$ and $\fact$ are isomorphisms.
\end{definition}

\begin{remark}
	Often the sheaves $\calA_{\Sigma}$ are endowed with more algebraic structures (e.g. they are associative or Lie algebras) and we will consider morphisms $\Ran,\fact$ which preserve the appropriate kind of structure. Let us say that the term \emph{algebra} in \emph{factorization algebra} has nothing to do with this extra structure.

	Moreover, note that the $\otimes^!$ product appearing in the definition may be replaced by any other symmetric monoidal operation on complete topological sheaves, for instance by $\times$ or $\oplus$. We will refer to such structures on a family of sheaves $\calA_\Sigma$ as a (pseudo) factorization algebra structure with respect to $\times,\oplus$.
\end{remark}

\begin{example}
	The collection of sheaves $\pbarO$ is a completed topological factorization algebra with respect to $\times$. In order to prove this, notice that $\hat{i}^*_{J/I}(\pbarO),\hat{j}^*_{J/I}(\pbarO)$ are the sheaves obtained by the same construction of $\pbarO$ for the family of curves $X \times_S V_{J/I} \to V_{J/I}$ and $X \times_S U_{J/I} \to U_{J/I}$ respectively. To show the claim it is then enough to show that
	\begin{enumerate}
		\item Given a surjection $J \twoheadrightarrow I$ such that $\sigma_{j_1} = \sigma_{j_2} : S \to X$ anytime $\pi(j_1) = \pi(j_2)$ then $\calO_{\overline{\Sigma}_J} = \calO_{\overline{\Sigma}_I}$. This is obvious since the sections are the same but with higher multiplicity, and such information disappears when taking the completion;
		\item Given a surjection $J \twoheadrightarrow I$ such that $\sigma_{j_1}(S) \cap \sigma_{j_2}(S) = \emptyset$ anytime $\pi(j_1) \neq \pi(j_2)$ then $\calO_{\overline{\Sigma}_J} = \times_{i \in I }\calO_{\overline{\Sigma}_{J_i}}$. This is easy to check since any quotient $p_*(\calO_X/\calI_{\Sigma_J}^n)$ splits as $\times_{i \in I} p_*(\calO_X/\calI_{\Sigma_{J_i}}^n)$.
	\end{enumerate}
\end{example}

Let us conclude with a simple Remark on how what we just discussed is a generalization of the factorization picture of \ref{def:factorizationspace}.

\begin{remark}\label{rmk:factsigmaugualefactcurva}
	Let $C$ be a smooth curve over $\mC$ and let $I$ be a finite set. Let $S = C^J$ and $X = C \times C^J$, via the canonical projection $p : X \to S$ this is a smooth family of curves. Consider the set of canonical sections $\Sigma^{\univ}_{C,J} = \{ \sigma^{\univ}_{C,J,j'} \}_{j' \in J}$, where
	\[
		\sigma^{\univ}_{C,J,j}(x_{j'})_{j' \in J} = \left(x_{j}, (x_{j'})_{j' \in J}\right).
	\]
	Then it is immediate to check that the closed subscheme $i_{J/I} : V_{J/I} \to S = C^J$ of Definition \ref{def:openclosedfactorization} matches up with the diagonal embedding $\Delta_{J/I} : C^I \to C^J$ appearing in Definition \ref{def:factorizationspace}. In the same manner the open immersion $j_{J/I} : U_{J/I} \to S = C^J$ of Definition \ref{def:openclosedfactorization} is the same as the open immersion $j_{J/I} : U_{J/I} \to C^J$ of Definition \ref{def:factorizationspace}, so that the notations agree. 
	In particular any completed factorization algebra in our $X,S,\Sigma$ (Definition \ref{def:factorizationcompletesheafsigma}) setting produces a completed topological factorization algebra on any smooth curve in the sense of Definition \ref{def:factorizationspace}.
\end{remark}

\subsection{The affine Lie algebra \texorpdfstring{$\gsig{\kappa}$}{in our geometric setting} and the enveloping algebra \texorpdfstring{$\ugsig{\kappa}$}{of the affine algebra}}\label{ssec:richiamiaffinealgebra}

In \cite[Section \ref{sec:thecaseoftheaffinealgebra}]{casmaffei1} out of Lie algebra $\gog$ over $\mC$ and a symmetric invariant bilinear form on it $\kappa$, we constructed a sheaf of Lie algebras which mimics the usual affine algebra attached to $(\gog,\kappa)$:
\begin{align*}
	\gsig{\kappa} &= \gog \otimes \pbarOpoli \oplus \calO_S \mathbf{1}; \\
	[\mathbf{1},\gsig{\kappa}] &= 0, \qquad [X \otimes f, Y \otimes g] = [X,Y]\otimes fg + \kappa(X,Y)\mathbf{1}\Res_\Sigma(gdf). 
\end{align*}
This sheaf, in the case $S = \Spec \mC$ and $X = C$ is a smooth curve, recovers the usual affine algebra of level $\kappa$. The factorization properties of $\pbarOpoli$ and $\Res$ induce a factorization structure on $\gsig{\kappa}$. 

From $\gsig{\kappa}$ we construct its sheaf of completed enveloping algebras $\ugsig{\kappa}$, obtained mimicking the usual construction for the affine algebra: first, we consider the usual sheaf of enveloping algebras and factor out $\mathbf{1} - 1$, where $\mathbf{1}$ is the central element of $\gsig{\kappa}$ while $1$ is the unit of the enveloping algebra of $\gsig{\kappa}$; then we complete this sheaf along the topology generated by the left ideals generated by $\gog \otimes \pbarO(-n)$. The collection of sheaves $\ugsig{\kappa}$ naturally forms a completed topological factorization algebra, this structure comes from the factorization structure of $\gsig{\kappa}$. For a precise construction of these structures we refer to \cite[Section \ref{sec:thecaseoftheaffinealgebra}]{casmaffei1}.

The goal of this paper is to study the center of this sheaf of algebras, $Z_\kappa(\hat{\gog}_{\Sigma}) = Z\left( \ugsig{\kappa}\right)$ at the critical level $\kappa_c = -\frac{1}{2}\kappa_{\gog}$ (here $\kappa_{\gog}$ is the non-normalized Killing form on $\gog$). For the precise definition of $Z_\kappa(\hat{\gog}_\Sigma)$ and its factorization properties we invite the reader to take a look at Section \ref{ssec:factorizationcenter}.

\subsection{Spaces of fields}\label{ssec:richiamifields}

Much of \cite{casmaffei1} is dedicated to study the \emph{spaces of field-distributions} $\mF^1_{\Sigma,\calU}$, which are built up from the sheaf $\pbarOpoli$ and a \QCC associative algebra with topology generated by left ideals $\calU$. The definition is simple: $\mF^1_{\Sigma,\calU} = \calHomcont_{\calO_S}\left( \pbarOpoli, \calU \right)$; the sheaves $\mF^1_{\Sigma,\calU}$ are well behaved and it can be shown that in the affine well covered case $S = \Spec A$ we have $\mF^1_{\Sigma,\calU}(S) = \Homcont_A(\pbarOpoli(S),\calU(S))$. Notice that the sheaf of distributions is naturally a \emph{right} $\pbarDpoli$-module by acting on $\pbarOpoli$ and that the residue morphism induces a canonical map $\mathbf{1} : \pbarOmegapoli \to \mF^1_{\Sigma,\calU}$, $\omega \mapsto ( f \mapsto \Res_\Sigma(f\omega)\cdot 1_{\calU})$, where $1_{\calU}$ is the unit of the associative algebra $\calU$.

Much of the work of \emph{loc. cit.} and of \cite{cas2023} consists in showing that the space of fields-distributions, together with its multivariable version $\mF^n_{\Sigma,\calU} = \calHomcont_{\calO_S}( \pbarOpolim{n},\calU)$ can be used as an analogue of the usual space of fields of the theory of vertex algebras. We recover indeed the usual space of fields in the case where $S = \mC$, $X = C$ is any smooth curve and $\Sigma =\{ x \in X\}$ is any point; in this case we have $\pbarOpoli = \mC((z))$ for any choice of an étale coordinate $z$ around $x \in X$. 

In what follows we review some of these constructions and considerations. In order to do so, let us first recall some notation: let $\calI_\Delta \subset \pbarOpolim{2}$ to be the kernel of the multiplication map $\pbarOpolim{2} = \pbarOpoli \otimes^* \pbarOpoli \to \pbarOpoli$, this is a locally free sheaf of $\pbarOpolim{2}$-modules of rank $1$, which in the case $S$ is well covered with a coordinate $t$, is freely generated by $t\otimes 1 - 1\otimes t$. Starting from this sheaf we can consider $\pbarOpolim{2}(\infty\Delta)$ which can be defined as the colimit of the sheaves $\pbarOpolim{2}(n\Delta) = \calHomcont_{\pbarOpolim{2}}\left( \calI^n_\Delta,\pbarOpolim{2} \right)$, this is naturally a sheaf of algebras which, in the case $S$ is well covered with a coordinate $t$, identifies with the localization $\pbarOpolim{2}[(t \otimes 1 - 1\otimes t)^{-1}]$. A crucial fact is that when considering ordered tensor product (see \cite[Section \ref{app:top}]{casmaffei1}) $\pbarOpoli \otimesr \pbarOpoli$ we get a natural continuous morphism of commutative algebras $\pbarOpolim{2}(\infty\Delta) \to \pbarOpoli\otimesr\pbarOpoli$.

Having these notions at hand one can define the product of two fields $X,Y \in \mF^1_{\Sigma,\calU}$ as the following $2$-field: \[XY : \pbarOpolim{2} \to \calU, \quad f\otimes g \mapsto X(f)Y(g),\] we call the induced morphism $m_r : \mF^1_{\Sigma,\calU} \otimes^* \mF^1_{\Sigma,\calU} \to \mF^2_{\Sigma,\calU}$; the fact that the topology of $\calU$ is generated by left ideals, so that $\calU$ is an $\otimesr$ associative algebra, implies that $m_r$ upgrades to a morphism $m_r : \mF^1_{\Sigma,\calU} \otimesr \mF^1_{\Sigma,\calU} \to \mF^2_{\Sigma,\calU}$. Analogously, we define $m_l : \mF^1_{\Sigma,\calU} \otimesl \mF^1_{\Sigma,\calU} \to \mF^2_{\Sigma,\calU}$ as $m_l(X\otimes Y)(f\otimes g) = Y(g)X(f)$. One moves forward defining what a local $2$-field $Z \in \mF^{2}_{\Sigma,\calU}$ is: that is simply a $2$-field $Z : \pbarOpolim{2} \to \calU$ which vanishes on some power of $\calI_\Delta$. The space of local $2$-fields $\mF^{2\loc}_{\Sigma,\calU}$ identifies with the $\calD$-module pushforward $\Delta_!\mF^1_{\Sigma,\calU}$. Two fields $X,Y$ are then said to be \emph{mutually local} if their bracket $[X,Y] = m_r(X\otimes Y) - m_l(X\otimes Y)$ is a local $2$-field.

The fact that there are morphisms of algebras $\pbarOpolim{2}(\infty\Delta) \to \pbarOpoli \otimesr \pbarOpoli$ implies that there are maps $\mF^1_{\Sigma,\calU} \otimes^* \mF^1_{\Sigma,\calU}(\infty\Delta) = (\mF^1_{\Sigma,\calU} \otimes^* \mF^1_{\Sigma,\calU})\otimes_{\pbarOpolim{2}} \pbarOpolim{2}(\infty\Delta) \to \mF^1_{\Sigma,\calU}\otimesr \mF^1_{\Sigma,\calU}$. In particular we get a morphism
\[
	\mu = m_r - m_l : \mF^1_{\Sigma,\calU} \otimesst \mF^1_{\Sigma,\calU} (\infty\Delta) \to \mF^2_{\Sigma,\calU}
\]
which we call the \emph{chiral bracket}. One of the main results of \cite{casmaffei1} is that the chiral bracket is part of a Lie structure on the collection $\mF^n_{\Sigma,\calU}$, in an appropriate sense. This, after we have carefully set up some operads, directly follows from associativity of $\calU$. We refer to \cite[Section \ref{sec:chiralprod}]{casmaffei1} for the precise formulation of these assertions. 

The crucial implication of the latter fact is that when we are given a subspace $\calV \subset \mF^1_{\Sigma,\calU}$ of mutually local fields, which is a sub $\pbarD$-module of $\mF^1_{\Sigma,\calU}$ and which is closed under the chiral product (that is $\mu \left(\calV \otimes^* \calV (\infty\Delta) \right) \subset \Delta_!\calV$), then the morphism $\mu$ is a Lie bracket in the chiral pseudo-tensor structure of Beilinson and Drinfeld. This leads to the fact that, under the further assumption that $\calV$ contains the image of $\mathbf{1} : \pbarOmega \to \mF^1$, the sheaf $\calV$ is a \emph{chiral algebra} over $\overline{\Sigma}$. The precise definition of a chiral algebra in our topological setting is a bit more subtle since we need to be careful with the topologies. This is a technical detail on which we will expand in the next section.

Another crucial result on the theory of field/distributions is the analogue of \emph{Dong's Lemma} in the theory of vertex algebras (see \cite[Lemma \ref{lem:Dong2}]{casmaffei1}). Very roughly, this states the following two facts:
\begin{itemize}
	\item For any mutually local fields $X,Y \in \mF^1_{\Sigma,\calU}$ and any $f \in \pbarOpolim{2}(\infty\Delta)$ the $2$-field $\mu(X \otimes Y \cdot f)$ is local;
	\item Let $\piuno : \mF^{2\loc}_{\Sigma,\calU} = \Delta_!\mF^1_{\Sigma,\calU} \to \mF^1_{\Sigma,\calU}$ be the residue map of $\calD$-module pushforward and let $X,Y,Z$ be mutually local fields. Then for every $f \in \pbarOpolim{2}(\infty\Delta)$ the $1$-fields $X$ and $\piuno(\mu(Y\otimes Z \cdot f))$ are mutually local. 
\end{itemize} 

This allows us to produce a subspace of mutually local fields which is closed under the chiral product, and hence a chiral algebra, starting from a subspace of mutually local fields. This procedure should be thought of as the analogue of the construction of the vertex algebra generated by some mutually local fields.

\subsubsection{Factorization of fields}\label{ssec:richiamifactfields}

When considering the space of fields $\mF^1_{\Sigma,\calU}$ we may let the complete $\otimesr$-algebra $\calU$ vary with $\Sigma$ as well: an important case for us is considering $\calU = \calU_\kappa(\hat{\gog}_\Sigma)$. It turns out that, in the case where we consider a collection  $\calU_\Sigma$ of complete $\otimesr$-algebras equipped with a factorization structure, the space of fields acquires natural factorization morphisms as well, as explained in \cite[Section \ref{ssec:fattorizzazionecampigenerale}]{casmaffei1}. Let us recall how they are constructed. 

Consider a finite set $J$ with a surjection $J \to I$ and consider the case where $\Sigma$ is indexed by $J$. Consider $X : \pbarOpoli \to \calU_\Sigma$ to be a field in $\mF^1_{\Sigma,\calU_\Sigma}$, by taking the pullbacks along $i_{J/I},j_{J/I}$ (see Definition \ref{def:openclosedfactorization}) and using the factorization properties of $\pbarOpoli$ and $\calU_{\Sigma}$ we get fields
\begin{align*}
	\hat{i}^*_{J/I}X &: \calO_{\oSigma^*_I} \to \hat{i}^*_{J/I}\calU_{\Sigma_k} \xrightarrow{\Ran^{\calU}_{J/I}} \calU_{\Sigma_I},\\
	\hat{j}^*_{J/I}X &: \prod_{i \in I}\hat{j}^*_{J/I}\calO_{\oSigma^*_{J_i}} \to \hat{j}^*_{J/I}\calU_{\Sigma_J} \xrightarrow{(\fact^{\calU}_{J/I})^{-1}} \bigotimes^!_{i \in I} \hat{j}^*_{J/I}\calU_{\Sigma_{J_i}}. 
\end{align*}
 
\noindent The construction $\hat{i}^*_{J/I}$ allows us to construct a morphism
$\Ran^{\mF^1}_{J/I}$ from $\hat{i}^*_{J/I}\mF^1_{\Sigma_J,\calU_{\Sigma_J}}$ to $\mF^1_{\Sigma_I,\calU_{\Sigma_I}}$, which is an isomorphism if $\calU$ is locally topologically free.
In addition, notice that $\hat{j}^*_{J/I}$ coincides with taking the restriction to $U_{J/I}$, so that
\[
	\hat{j}^*_{J/I}\mF^1_{\Sigma_J,\calU_{\Sigma_J}} = \calHom^{\cont}\left(\prod_{i \in I} \hat{j}^*_{J/I}\calO_{\oSigma^*_{J_i}},\bigotimes_{i \in I} \hat{j}^*_{J/I}\calU_{\Sigma_{J_i}}\right)
\]
In particular, given fields $X_i : \hat{j}^*_{J/I}\calO_{\oSigma^*_{J_i}} \to \hat{j}^*_{J/I}\calU_{\Sigma_{J_i}}$ we can construct a field
\[
	\left(\prod_i X_i\right) (f_i)_{i \in I} = \sum_{i \in I} \iota_i(X_i(f_i))  
\]
where $f_i \in \hat{j}^*_{J/I}\calO_{\oSigma^*_{J_i}}$ and $\iota_i$ is the natural map $\hat{j}^*_{J/I}\calU_{\Sigma_{J_i}}\to \otimes^!_{i \in I} \hat{j}^*_{J/I}\calU_{\Sigma_{J_i}}$. This construction yields a morphism
\[
	\fact^{\mF^1}_{J/I} : \prod_{i \in I} \hat{j}^*_{J/I}\mF^1_{\Sigma_{J_i},\calU_{\Sigma_{J_i}}} \to \hat{j}^*_{J/I}\mF^1_{\Sigma_J,\calU_{\Sigma_J}}.
\]
The morphisms $\Ran^{\mF^1},\fact^{\mF^1}$ satisfy properties analogous to those of Definition \ref{def:factorizationcompletesheafsigma}.

Let us mention here that the fields we will be dealing with will have the additional property that $\hat{j}^*_{J/I}X$ lies in the image of $\fact^{\mF^1}_{J/I}$ so the field $X$ factorizes to a collection of fields $X_i : \hat{j}^*_{J/I}\calO_{\oSigma^*_{J_i}} \to \hat{j}^*_{J/I}\calU_{\Sigma_{J_i}}$.

\subsection{Chiral algebras}\label{ssec:richiamichiralalgebras}

In order to deal with some topological issues the definition we give in \cite{casmaffei1} of a chiral algebra differs slightly from the one that Beilinson and Drinfeld give in \cite{BDchirali}. In particular our chiral algebras come equipped with a filtration. Let us restate here the definition of a chiral algebra in our topological setting. In order to state this definition we need to recall some constructions:
\begin{itemize}
	\item Let $\calV$ be a topological sheaf which is a right $\pbarD$-module, then there is a canonical isomorphism of topological right $\pbarDm{2}$-modules \[D(\calV) : \Delta_!\calV \to \frac{\pbarOmega\exttensor{*}\calV(\infty\Delta)}{\pbarOmega\exttensor{*}\calV}.\]
	We invite the reader to look at \cite[Proposition \ref{prop:defD}]{casmaffei1} for the proof.
	\item For a topological sheaf  $\calV$ which is a right $\pbarDm{k}$-module an increasing filtration of $\calV$ is a sequence of $\pbarOm{k}$ submodules $\calV(n) \subset \calV$ such that for every $n,l$ there exists some $m$ such that $\calV(n)\cdot (\pbarDm{k})^{\leq l} \subset \calV(m)$ (see \cite[Section \ref{ssec:operatoridifferenziali}]{casmaffei1} for the definition of $\pbarDm{k}$ and its filtration) and such that $\calV = \varinjlim_n \calV(n)$ as sheaves (we do not require them to be equal as topological sheaves).
	\item Let $\{\calV_{i}\}_{i = 1,\dots, h}$ be a collection of topological sheaves which are $\pbarDm{k_i}$-modules with increasing filtrations $\calV_i(n) \subset \calV_i$, then one defines \[\calV_1 \exttensor{\text{fil}}\calV_2 \exttensor{\text{fil}}\dots\exttensor{\text{fil}}\calV_h := \varinjlim_n \calV_1(n) \exttensor{*}\calV_2(n) \exttensor{*}\dots\exttensor{*}\calV_h(n).\]
	\item The chiral (topological) Beilinson-Drinfeld pseudo tensor structure on filtered $\pbarD$-modules is defined as
	\[
		P^{chBD}_I(\{M_i\}_{i \in I},N) = \Hom^{\text{fil}}_{\pbarDm{I}}\left( (\exttensor{\text{fil}} M_i)(\infty\nabla(I)), \Delta(I)_!N \right).
	\]
	For the notation on $\infty\nabla(I)$ we refer to \cite[Section \ref{ssec:polimolti}]{casmaffei1}, while for the filtrations on the modules $(\exttensor{\text{fil}} M_i)(\infty\nabla(I))$ and $\Delta(I)_!N$ we refer to \cite[Section \ref{ssec:filteredbeilinsondrinfeld}]{casmaffei1}. There it is also shown that that the filtered assumption on the modules is needed in order to define composition in this pseudo-tensor structure.
\end{itemize} 
With these constructions at hand we can define chiral algebras in our topological setting as follows.

\begin{definition}
	A \emph{chiral algebra} on $\overline{\Sigma}$ is the datum of a topological, filtered $\pbarD$-module $\calV$ together with morphisms
	\begin{align*}
		\mu &: \calV \exttensor{\text{fil}}\calV(\infty\Delta) \to \Delta_!\calV, \\
		\mathbf{1} &: \pbarOmega \to \calV,
	\end{align*}
	such that $\mu$ is a Lie bracket for the Beilinson-Drinfeld chiral pseudo-tensor structure, such that the composition $\mu \circ (\mathbf{1}\exttensor{}\id(\infty\Delta))$ vanishes on $\pbarOmega\exttensor{\text{fil}}\calV$, and the induced map on the quotient coincides with the inverse of the morphism $D(\calV)$. 
	The notion of a chiral algebra on $\overline{\Sigma}^*$ is defined analogously.
\end{definition}

\smallskip

\subsubsection{Local description of chiral algebras: vertex algebras over \texorpdfstring{$(\pbarO,t)$}{(Sigma,t)}}\label{ssec:richiamivertexalgoverost}

In the case where $S = \Spec A$ is affine and well covered we may relate the notion of chiral algebras to the notion of a ordinary vertex algebra, in the incarnation of what we call \emph{vertex algebra over} $(\pbarO,t)$ in \cite[Section \ref{ssez:Otalgebre}]{casmaffei1}. 

\noindent A vertex algebra over $(\pbarO(S),t)$ is the datum of a right $\pbarD(S)$-topological module $\mV$ together with an $A$-linear vertex algebra structure $(\mV,T,\vac,Y: \mV \to \End_A(\mV)[[z^{\pm 1}]])$ such that $T = -\partial_t$ and for any $v,w \in \mV$ and $f \in \pbarO(S)$ the following sesquilinearity conditions hold:
\begin{align}\label{eq:sequivertexalgebraost}
	v_{(n)}(wf) &= (v_{(n)}w)f, \\
	(vf)_{(n)}w &= \sum_{k\geq 0} \frac{1}{k!}v_{(n+k)}w\cdot(\partial_t^kf).
\end{align}

\smallskip

When looking at global sections of a chiral algebra in the affine and well covered case, we see that we may write
\[
	\left(\calV\exttensor{\text{fil}}\calV (\infty\Delta)\right)(S) = \calV(S)\exttensor{\text{fil}}\calV(S)[(t\otimes 1-1\otimes t)^{-1}] \qquad (\Delta_!\calV)(S) = \bigoplus_{k\geq 0} \calV(S)\partial^k_{t\otimes 1}.
\]
It follows that the chiral product of two elements $v,w \in\calV(S)$ may be written as
\[
	\mu\left(v\boxtimes w (t\otimes 1-1\otimes t)^n\right) = \sum_{k \geq 0} \frac{1}{k!} v_{(n+k)}w\cdot \partial_{t\otimes 1}^k,
\]
for any $n \in \mZ$ and some $\calO(S)$-linear operations $v_{(m)}w$. It follows by \cite[Lemma \ref{lem:chiralivertice} and Corollary \ref{cor:chiraliOt}]{casmaffei1} that these operations satisfy the axioms of a vertex algebra over $(\pbarO(S),t)$ and that the assignment $\calV \mapsto \calV(S)$ determines an equivalence between filtered \QCC chiral algebras and filtered complete vertex algebras over $(\pbarO(S),t)$.

\smallskip

The above considerations allow us to construct chiral algebras in the local case starting from vertex algebras. Indeed, given any ordinary vertex algebra $V$, the module $V \otimesl \pbarO(S)$ is naturally a vertex algebra over $(\pbarO,t)$, defining its structure of right $\pbarD(S)$-module by $(v\otimes f)\partial_t = -Tv \otimes f - v\otimes\partial_tf$ and $(X\otimes f)_{(n)}(Z \otimes g)$ in order to make the above formulas work.

\subsection{The chiral algebra \texorpdfstring{$\VkSg$}{attached to the affine Lie algebra}}\label{ssec:richiamiVkSg}

By the discussion of \cite[Section \ref{ssez:generatechiral}]{casmaffei1}, given a sub $\calO_S$-module $\calG \subset \mF^1_{\Sigma,\calU}$ consisting of mutually local fields we can consider its  \emph{generated chiral algebra}, which may be constructed by iterating the chiral product and adding the image of the unit $\mathbf{1} : \pbarOmega \to \mF^1$. There are two interesting versions of this generated chiral algebra, which in \emph{loc. cit.} go by the name of $\calV_{\text{basic}+}(\calG)$,$\calV_{\text{com}+}(\calG)$. The chiral algebra $\calV_{\text{basic}+}(\calG)$ is obtained as the colimit of the image of iterations of the chiral product, while $\calV_{\text{com}+}(\calG)$ differ by considering a suitable completion. We are interested in applying this construction to the particular case where $\calU = \ugsig{\kappa}$ and $\calG = \gog \otimes \calO_S$ is the subspace generated by the fields $X : f \mapsto X \otimes f$, for $X \in \gog$. The above difference between $\calV_{\text{basic}+}(\calG)$,$\calV_{\text{com}+}(\calG)$ expained above won't play any significant role in this paper, since, as shown in \emph{loc. cit.}, the two constructions $\calV_{\text{basic}+}(\calG),\calV_{\text{com}+}(\calG)$ agree in this case, so that we get a canonically defined chiral algebra $\VkSg \subset \mF^1_{\Sigma,\ugsig{\kappa}}$.

\medskip

\subsubsection{The map \texorpdfstring{$\calY_{\Sigma,t}$}{Y t}}\label{ssec:recollectionsysigma}

It turns out that, in the case $S = \Spec A$ is affine, well covered, and with a coordinate $t \in \pbarO$, is possible to describe explicitly $\VkSg$. The module of sections $\VkSg(S)$ becomes, in this case and as explained before, a vertex algebra over $(\pbarO(S),t)$. For any $X \in \gog$ there are canonically defined fields $X \in \VkSg(S)$ which are defined as $X(f) = X \otimes f$. When looking at their $n$-products relative to the vertex algebra over $(\pbarO(S),t)$ structure on $\VkSg(S)$ we check in \cite[Proposition \ref{prp:mappaY}]{casmaffei1} (using the results of \cite{cas2023}) that the attached operators $X_{(n)}$ (let us emphasise that these highly depend on the choice of the coordinate $t$) satisfy the usual commutation relations of the affine algebra $\hat{\gog}_{\kappa}$. By the universal property of $V^\kappa(\gog)$ there exists an induced morphism of vertex algebras $V^\kappa(\gog) \to \VkSg(S)$, we denote its $\pbarO$-linear extension by
\[
	\calY_{\Sigma,t} : V^\kappa(\gog) \otimes \pbarO \to \VkSg.
\]
In \cite[Theorem \ref{teo:descrizionelocaleVkSg}]{casmaffei1} it is shown that $\calY_{\Sigma,t}$ is an isomorphism. In addition we show in Corollary \ref{coro:subchiralalg} of \emph{loc. cit.} that the map $\calY_{\Sigma,t}$ restricts to an isomorphism $\zeta^{\kappa}(\gog) \otimes \pbarO \to \zeta^\kappa_\Sigma(\gog)$, where $\zeta^\kappa(\gog)$ is the center of the vertex algebra $V^\kappa(\gog)$ while $\zeta^\kappa_\Sigma(\gog)$ is the center of the chiral algebra $\VkSg$.

The map $\calY_{\Sigma,t}$ will be one of the main players of this paper, so let us expand a little bit on how it is constructed. Notice that by $\pbarO$-linearity we just need to describe its behavior on $V^\kappa(\gog)$ and that we may see elements in $\VkSg \subset \mF^1_{\Sigma,\ugsig{\kappa}}$ as fields, so that given $v \in V^\kappa(\gog)$ we describe each $\calY_{\Sigma,t}(v)$ as a field.

For any $X \in \gog, f \in \pbarOpoli$ we have by construction $\calY_{\Sigma,t}(X_{(-1)}\vac)(f) = X \otimes f \in \ugsig{\kappa}$. The extension of $\calY_{\Sigma,t}$ to $V^\kappa(\gog)$ is defined inductively, using the chiral product on $\VkSg$, indeed, given any two elements $v,w \in V^\kappa(\gog)$ and any $n \in \mZ$, we have
\[
	\calY_{\Sigma,t}(v_{(n)}w)(f) = \mu\left( v \otimes w (1\otimes f)(t\otimes 1-1\otimes t)^n \right).
\]

In particular, in the case where $S = \Spec \mC$ the choice of a coordinate $z \in \pbarO$ induces isomorphisms $\pbarO \simeq \mC[[z]]$, $\pbarOpoli \simeq \mC((z))$. Via these isomorphisms we may identify
\[
	\mF^1_{\Sigma,\ugsig{\kappa}} \simeq \Homcont_{\mC}\left( \mC((z)),U_{\kappa}(\hat{\gog})\right),
\] 
where $U_\kappa(\hat{\gog})$ is the completed enveloping algebra of $\hat{\gog}_\kappa = \gog\otimes\mC((z)) \oplus \mathbf{1}\mC$. 

The map $\calY_z : V^\kappa(\gog) \to \mF^1_{\Sigma,\ugsig{\kappa}}$ coincides, via the above identification, with an enhancement, of the usual state/field correspondence. Indeed, as explained in \cite{cas2023}, post-composing $\calY_z$ with $\Homcont(\mC((z)),U_\kappa(\hat{\gog})) \to \Homcont(\mC((z)),\End V^\kappa(\gog))$ we recover the usual state/field correspondence $Y$ of the structure of a vertex algebra. This follows from the fact that our definition of $n$-products agrees with the usual one after the identification $\Homcont(\mC((z)),\End(V^\kappa(\gog))) = \{ \text{fields on } V^\kappa(\gog) \}$.

\subsubsection{Factorization properties of \texorpdfstring{$\VkSg$}{the chiral algebra}}\label{ssec:richiamifactchiral}

A crucial point in the discussion of factorization properties is Theorem \ref{thm:factorizationofcaly} of \cite{casmaffei1} which shows that the map $\calY_{\Sigma,t}$ intertwines between the factorization structure of $\pbarO$ and the factorization morphisms of $\mF^1_{\Sigma,\gog}$ (see Section \ref{ssec:richiamifactfields}). Let us mention that this result can be also found in \cite{cas2023}[Proposition 7.2.1 and Corollary 7.2.3] and essentially follows from the factorization properties of the fields attached to elements of $\gog$ together with some inductive argument. 

It follows that the fields attached to elements of $V^\kappa(\gog)$ naturally factorize and that the chiral algebra $\VkSg$ has a natural factorization structure (\cite[Proposition \ref{prop:factpropertieschiralalg}]{casmaffei1}), so that we have isomorphisms of chiral algebras
\begin{align*}
	\Ran^{V^\kappa}_{J/I} &: \hat{i}_{J/I}^*\calV^\kappa_{\Sigma_J}(\gog), \to \calV^\kappa_{\Sigma_I}(\gog) \\ \fact^{V^\kappa}_{J/I} &: \hat{j}_{J/I}^*\left( \prod_{i\in I} \calV^\kappa_{\Sigma_{J_i}}(\gog) \right) \to  \hat{j}_{J/I}^*\calV^\kappa_{\Sigma_J}(\gog) 
\end{align*}
which satisfy the condition of Definition \ref{def:factorizationcompletesheafsigma}. These maps are read under the isomorphism $\calY_{\Sigma,t} : V^\kappa(\gog)\otimes\pbarO \to \VkSg$ as the morphisms induced by the factorization structure of $\pbarO$.

\subsubsection{The group \texorpdfstring{$\Autpiu{} O$}{PDFstring} and \texorpdfstring{$\calY_t$}{Y} in the single section case and the action of coordinate changes}\label{ssec:autoeytilde}

In the case where $\Sigma$ consists of a single section, and $S = \Spec A$ is affine and well covered, we restrict our attention to coordinates $t \in \pbarO$ which induce an isomorphism $\rho_t : A[[z]] \xrightarrow{\simeq} \pbarO(S), (z \mapsto t)$. The set of such coordinates forms an $\Autpiu{} O (A) = \{ \tau \in \Aut^{\cont}_A(A[[z]]) : \tau(z) \in (z) \}$-torsor (we refer to section \ref{ssec:recollectionsauto} for the precise definition of the group scheme $\Autpiu{} O$). In particular, for any element $\tau \in \Autpiu{} O(A)$ there is a well defined element $\tau(t) \in \pbarO$ which is a coordinate.

In \cite[Section \ref{ssec:autoactions}]{casmaffei1} we studied how the map $\calY_{\Sigma,t}$ behaves when changing the coordinate. In particular the map \[\tilde{\calY}_{\Sigma,t} : V^\kappa(\gog) \otimes A[[z]] \xrightarrow{\id \otimes \rho_t} V^\kappa(\gog)\otimes\pbarO(S) \to \VkSg(S)\]
is taken in analysis and the composition $\calY_{\tau} = \tilde{\calY}_{\Sigma,t}^{-1} \circ \tilde{\calY}_{\Sigma,\tau(t)}$ is studied. It is shown that the assignment $\tau \mapsto \calY_{\tau}$ determines a group homomorphism $\Autpiu{} O(A) \to \Aut_A\left( V^\kappa(\gog) \otimes A[[z]] \right)$ which is independent from the choice of a coordinate.

After the choice of a specified coordinate $z$ we have $$\mF^1 = \Hom_A\left( A((z)),\calU_{\kappa}(\hat{\gog}_{A((z))})\right)$$ and an immersion $\calY : V^\kappa \otimes A[[z]] \to \Hom_A\left( A((z)),\calU_{\kappa}(\hat{\gog}_{A})\right)$, where $\calU_\kappa(\hat{\gog}_{A})$ is the complete enveloping algebra of $\hat{\gog}_{A,\kappa} =\gog \otimes A((z)) \oplus A\mathbf{1}$. If we consider on the rightmost space the action of $\Autpiu{} O (A)$ by conjugation it is shown in \cite[Lemma \ref{lem:groupactions}]{casmaffei1} that the space on the left is left invariant and that the restriction of this action for some $\tau \in \Autpiu{} O(A)$ coincides with $\calY_\tau$.

\section{Lie and enveloping algebras of chiral algebras}\label{sec:envelopingalgchiralalg}

In this section we are going to construct a sheaf of associative complete topological $\otimesr$-algebras on $S$ starting from a chiral algebra $\calV$ on $\pbarO$. We are interested in the case $\calV =\VkSg$ but we will provide a construction in general. This discussion may be interpreted as a  `sheafification' along the base scheme $S$ of the construction of Section 4.2 of \cite{cas2023}, it parallels \cite[Sect. 3.6]{BDchirali} in which we allow poles on our $\Sigma$ instead of allowing poles around a $\mC$-point $x \in X$.

\subsection{The Lie algebra $\Lie_{\overline{\Sigma}^*}(\calV)$}\label{ssec:LieV}
Let $\big(\calV, \mu: \calV \exttensor{\text{fil}} \calV \to \Delta_! \calV, u:\pbarOmega\lra \calV\big)$ be a chiral algebra and let $\calV(n)$ be its filtration. We first construct a sheaf of Lie algebras attached to $\calV$. The construction will be an $\calO_S$-linear version of \cite[2.5.2]{BDchirali} and it will require some steps.

Consider the sheaf 
	\[
	\calVstar \stackrel{\text{def}}{=} \calV \otimes_{\pbarO} \pbarOpoli
	\]
	This sheaf is naturally a chiral algebra over $\pbarOpoli$ extending the chiral bracket $\pbarOpoli$-linearly. The filtration on this chiral algebra is given by $\calVstar(n)=\calV(n)\otimes _ \pbarO \pbarO(n)$. 
	The modules $\calVstar(n)$ and $\calV(n)$ are locally isomorphic, on the former, we put the resulting topology and we consider on $\calVstar$ the colimit topology
	(see \cite[Section \ref{ssec:localization}]{casmaffei1} for the details on this topological version of the localization). In the case of the chiral algebra $\calV=\VkSg$  we denote these objects by $\VkSgstar$ and $\VkSgstar(n)$. 
	
	In what follows we are going to assume that $\calV=\limind \calV(n)$ in the category of topological sheaves. We are going to make some constructions that can be given also without this hypothesis, but let us notice that these would produce the same structures for $\calV$ and $\limind \calV(n)$, so there is no loss of generality in assuming $\calV=\limind \calV(n)$.

\subsubsection{The Lie algebra $h^0(\calVstar)$.}

    Recall that the De Rham cohomology of the right $\pbarDpoli$ module $\calVstar$ is defined as 
	\[
	h^0(\calV_{\overline{\Sigma}^*}) = \frac{\calVstar}{\calVstar \cdot T_{\overline{\Sigma}^*}}
	\] 
	and notice that 
	$$
	h^0(\Delta_!\calVstar)=\frac{\calVstar}{\calVstar \cdot T^2_{\overline{\Sigma}^*}}\isocan
	\Delta_*(h^0(\calVstar)).
	$$
It follows that $h^0(\mu)$ induces a structure of $\calO_S$-Lie algebra $[\cdot,\cdot]:h^0(\calVstar)\otimes h^0(\calVstar)\lra h^0(\calVstar)$. This is an $\calO_S$-linear version of \cite[2.5.2]{BDchirali}.

The sheaf $h^0(\calVstar)$ has a natural filtration constructed as follows
$$
	h^0(\calV_{\overline{\Sigma}^*})(n) = \frac{\calVstar(n)}{\big(\calVstar \cdot T_{\overline{\Sigma}^*}\big)\cap \big(\calVstar(n)\big)}.
$$
Using the fact that $S$ is noetherian and the assumption $\calV=\limind \calV(n)$ it is possible to verify that $h^0(\calVstar)=\limind h^0(\calVstar)(n)$ as a sheaf. We put on $h^0(\calVstar)$ the colimit topology. With this topology we have that the restriction of the Lie bracket $[\cdot,\cdot]$ to $h^0(\calVstar)(n)\times h^0(\calVstar)(n)$ is continuous and, hence, 
the Lie bracket $[\cdot,\cdot]$ on $h^0(\calVstar)\times h^0(\calVstar)$ is continuous in each variable. 

\begin{comment}

The topology on $\calV$ induces a natural topology on $h^0(\calV_{\overline{\Sigma}^*})$ which is naturally a sheaf of $\calO_S$-Lie algebras via the continuous bracket $[\_,\_] = h^0(\mu)$. This is an $\calO_S$-linear version of \cite{BDchirali}[2.5.2].
\end{comment}

\begin{remark}
	The chiral algebra unit morphism $\uch : \Omega^1_{\overline{\Sigma}} \to \calV$ induces a morphism $h^0(\uch) : h^0(\pbarOmegapoli) \to h^0(\calV_{\overline{\Sigma}^*})$. It follows from the chiral unit axioms that the image of $u$ is central in $h^0(\calV_{\overline{\Sigma}^*})$. Indeed, $\mu\circ(\uch,I)=D(\calVstar)$ the canonical morphism from $\pbarOmegapoli\boxtimes \calVstar(\infty \Delta)$ to $\Delta_!\calVstar$ which vanishes on $\pbarOmegapoli\boxtimes \calVstar$; hence for all sections $\omega$ of $\pbarOmegapoli$ and $v$ of $\calVstar$ we have
	$[u(\omega),v]= D(\calVstar)(v)=0$.
\end{remark}

\subsubsection{Construction of \texorpdfstring{$\Lie_{\oSigma^*}(\calV)$}{the Lie algebra of Fourier coefficients}}
The following definition is an adaptation to our geometric setting of the Lie algebra of Fourier coefficients of a vertex algebra (denoted by $U(V)$ in \cite{frenkel2007langlands}) and the Lie algebra $\Lie_{K_n}(V)$ of \cite{cas2023}.

\begin{definition}\label{def:liesigmastar}
	Let $\calV$ be a chiral algebra over $\pbarO$. We define
	\[
	\Lie_{\overline{\Sigma}^*}(\calV) \stackrel{\text{def}}{=} h^0(\calV_{\overline{\Sigma}^*}) \coprod_{h^0(\Omega^1_{\overline{\Sigma}^*})} \calO_S
	\]
	where\index{$\Lie_{\overline{\Sigma}^*}(\calV)$} the coproduct is taken over the maps $h^0(\uch) : h^0(\Omega^1_{\overline{\Sigma}^*}) \to h^0(\calV_{\overline{\Sigma}^*})$ and $\Res_{\Sigma} : h^0(\Omega^1_{\overline{\Sigma}^*}) \to \calO_S$. This is a sheaf of $\calO_S$-Lie algebras. We denote by $\ulie : \calO_S \to \Lie_{\overline{\Sigma}^*}(\calV)$ the induced natural map, which has central image.
	
	This construction is functorial in the chiral algebra $\calV$, so that given a morphism of chiral algebras $\calV \to \calV'$ we get a morphism of $\calO_S$-Lie algebras $\Lie_{\overline{\Sigma}^*}(\calV) \to \Lie_{\overline{\Sigma}^*}(\calV')$.
\end{definition}

This Lie algebra comes with a natural filtration: we define
$\Lie_{\oSigma^*}(\calV)(n)$ as the image of $h^0(\calV_{\oSigma^*})(n) \oplus \calO_S$ in $\Lie_{\oSigma^*}(\calV)$. We have that $\Lie_{\oSigma^*}(\calV)$ is the colimit of $\Lie_{\oSigma^*}(\calV)(n)$ as a sheaf. 

Each $\Lie_{\oSigma^*}(\calV)(n)$ is the quotient of $h^0(\calV_{\oSigma^*})(n) \oplus \calO_S$, we put on 
$\Lie_{\oSigma^*}(\calV)(n)$ the quotient topology and on 
$\Lie_{\oSigma^*}(\calV)$ we put the colimit topology.
For this topology the bracket is continuous in each variable and, moreover, for all $n$ there exists $m$ such that the bracket of  $\Lie_{\oSigma^*}(\calV)(n)$ with itself is contained in $\Lie_{\oSigma^*}(\calV)(m)$ and the map
$\Lie_{\oSigma^*}(\calV)(n)\times \Lie_{\oSigma^*}(\calV)(n)\lra \Lie_{\oSigma^*}(\calV)(m)$ is continuous. 

We define $\oLie_{\oSigma^*}(\calV)(n)$ as the completion of $\Lie_{\oSigma^*}(\calV)(n)$ and $\oLie_{\oSigma^*}(\calV)$ as the colimit of the system $\oLie_{\oSigma^*}(\calV)(n)$. The the bracket is well defined as a morphism
$\oLie_{\oSigma^*}(\calV)\,\exttensor{\fil}\,\oLie _{\oSigma^*}(\calV) \to \oLie _{\oSigma^*}(\calV)$.

\subsubsection{$\Lie_{\oSigma^*}(\calV)$ in the case of fields}\label{rmk:lievkfields}
The following remarks deal with the construction $\Lie_{\oSigma^*}$ in the special case of chiral algebras realized as a space of fields.

Let $\calV$ be a chiral algebra such that $\calV(n) \subset \mF^1_{\Sigma,\calU}$ consists of mutually local fields. By that we mean that the chiral product and the chiral unit $\Omega^1_{\overline{\Sigma}} \to \calV$ are induced by those of $\mF^1_{\Sigma,\calU}$. This case applies for instance when $\calV$ is generated by a subsheaf $\calG \subset \mF^1_{\Sigma,\calU}$ of mutually local fields so that we may take $\calV$ to be any of the chiral algebras $\vbasicpiu,\vcompiu$ defined in \cite[Section \ref{ssez:generatechiral}]{casmaffei1}. In this case we may identify the Lie product of the classes of $X,Y \in \calV_{\overline{\Sigma}^*}$ inside $\Lie_{\overline{\Sigma}^*}(\calV)$ with the class of the field
		\[
		\piuno(\mu(X\boxtimes Y)).
		\]
		Here we view $\mu(X\boxtimes Y) \in \Delta_! \calV \subset \mF^{2\loc}_{\Sigma,\calU}$, this coincides the bracket of the fields $X$ and $Y$, while $\piuno : \mF^2_{\Sigma,\calU} \to \mF^1_{\Sigma,\calU}$ is the restriction along the immersion $\pbarOpoli \to \pbarOpoliq$ which sends $f \mapsto 1 \otimes f$. This follows from \cite[Lemma \ref{lem:campilocali}]{casmaffei1}.

	\begin{lemma}\label{rmk:actionlievku} Assume $\calV$ satisfies the hypotheses above, and let $\calU$ be an associative complete topological $\otimesr$-algebra over $\calO_S$.  We may then consider the inclusion $\calV_{\overline{\Sigma}^*} \subset \mF^1_{\Sigma,\calU}$ and the associated morphism
		\[
		\ev_1 : \calV_{\overline{\Sigma}^*} \to \calU, \qquad X \mapsto X(1).
		\]
		This induces a continuous morphism of Lie algebras
		\(
		\beta: \Lie_{\overline{\Sigma}^*}\calV \to \calU. \qedhere
		\)
	\end{lemma}
	\begin{proof}
		The map defined above clearly vanishes on $\calV \cdot \Tan_{\overline{\Sigma}^*}$, in addition by assumption that the chiral unit is the same one of $\mF^1_{\Sigma,\calU}$ the corresponding morphism $h^0(\Omega^1_{\overline{\Sigma}}) \to \calU$ naturally factors as the composition $$h^0(\Omega^1_{\overline{\Sigma}}) \xrightarrow{\Res_\Sigma} \calO_S \xrightarrow{1_{\calU}} \calU.$$ It follows that $X \mapsto X(1)$ gives a well defined morphism $\Lie_{\overline{\Sigma}^*}\calV \to \calU$. Given two fields $X, Y \in \calV_{\overline{\Sigma}^*} \subset \mF^1_{\Sigma,\calU}$ we may compute the Lie bracket of their classes $\overline{X},\overline{Y} \in \Lie_{\overline{\Sigma}^*}(\calV)$ as in the discussion at the beginning of Section \ref{rmk:lievkfields} so that
		\[
		[\beta(\overline{X}),\beta(\overline{Y})] = [X(1),Y(1)] = [X,Y](1\otimes 1) = \beta(\overline{\piuno([X,Y])}) = \beta([\overline{X},\overline{Y}])
		\]
		Finally the morphism is continuous, becouse the evaluation $X\mapsto X(1)$ from $\mF^1_{\Sigma,\calU}$ is continuous, hence the restriction of $\grb$ to $\calV(n)$ is continuous for all $n$. 
	\end{proof}
	In the particular the case of the chiral algebra $\VkSgstar$ this construction is related with the construction of the affine Lie algebra as follows.
	\begin{lemma}\label{rmk:gktolievk}
		There exists a canonical continuous morphism of topological $\calO_S$-Lie algebras
		\[
		\alpha : \hat{\gog}_{\Sigma,\kappa} \to \Lie_{\overline{\Sigma}^*}(\VkSg)
		\]
		defined on $\hat{\gog}_{\Sigma,\kappa} = \gog\otimes\pbarOpoli \oplus \calO_S\mathbf{1}$ by sending
		\begin{align*}
		\gog\otimes\pbarOpoli &\to \VkSg\otimes_{\pbarO}\pbarOpoli \to \Lie_{\overline{\Sigma}^*}(\VkSg) \\
		\calO_S\mathbf{1} &\xrightarrow{u}  \Lie_{\overline{\Sigma}^*}(\VkSg),
		\end{align*}
		where the first map comes from the natural inclusion $\gog \subset \VkSg$.
	\end{lemma}
	\begin{proof}
		The map is already defined, linear and continuous, so the only condition to check is that it is a morphism of Lie algebras. Let $X,Y \in \gog$ and $\hat{X},\hat{Y}$ the associated fields, $f,g \in \pbarOpoli$. We want to compute the bracket $[\hat{X}f,\hat{Y}g] \in \Lie_{\overline{\Sigma}^*}(\VkSg)$. Computing the bracket of fields, for any $a,b \in \pbarOpoli$, we have
		\[
		[\hat{X}f,\hat{Y}g](a\otimes b) = [X\otimes fa,Y \otimes gb] = [X,Y]\otimes (fgab) + \kappa(X,Y)\Res_{\Sigma}\left( gb\cdot d(fa)\right)
		\]
		as elements of $\ugsig{\kappa}$. By the discussion of Section \ref{rmk:lievkfields} we can identify $[\hat{X}f,\hat{Y}g] \in \Lie_{\overline{\Sigma}^*}(\VkSg)$ with the class of the field
		\[
		b \mapsto [X,Y]\otimes fgb + \kappa(X,Y)\Res_{\Sigma}(b\cdot gdf).
		\]
		This is exactly the field 
		\(
		\widehat{[X,Y]} fg + \kappa(X,Y)\Res_{\Sigma}(\_\, gdf),
		\)
		which matches up with the image under $\alpha$ of 
		\[
		[X,Y] fg + \kappa(X,Y)\Res_{\Sigma}(gdf)\mathbf{1}.\qedhere
		\]
	\end{proof}

\subsection{Enveloping algebras}\label{ssec:envelopingalgchiralalg}

The goal of this section is to define a complete topological sheaf of $\otimesr$-algebras $\calU_{\overline{\Sigma}^*}(\calV)$ attached to a chiral algebra $\calV$ over $\pbarO$. This will be a modification of the enveloping algebra of $\Lie_{\overline{\Sigma}^*}(\calV)$ in the same spirit of the constructions $\widetilde{U}(V)$ of \cite{frenkel2007langlands} and $\mathbb{U}_{K_n}(V)$ of \cite{cas2023} (for $V$ a vertex algebra).

\begin{definition}\label{def:envelopingchiral1}
	Let $\calU(\Lie_{\overline{\Sigma}^*}(\calV))$ be the sheaf of enveloping algebras of $\Lie_{\overline{\Sigma}^*}(\calV)$. We define
	\[
	\calU^0_{\overline{\Sigma}^*}(\calV) \stackrel{\text{def}}{=} \frac{\calU(\Lie_{\overline{\Sigma}^*}(\calV))}{(\ulie - \unoU)}
	\]
	where $\ulie : \calO_S \to \Lie_{\overline{\Sigma}^*}(\calV) \to \calU(\Lie_{\overline{\Sigma}^*}(\calV))$ is the map defined in Definition \ref{def:liesigmastar} and $\unoU : \calO_S \to \calU(\Lie_{\overline{\Sigma}^*}(\calV))$ is the unit of the $\calO_S$-algebra structure.
	
	The topology on $\Lie_{\overline{\Sigma}^*}(\calV)$ induces a natural topology on $\calU^0_{\overline{\Sigma}^*}(\calV)$, where a \fsonoz is given by the left ideals generated by a \fsonoz of $\Lie_{\overline{\Sigma}^*}(\calV)$. Since the bracket on $\Lie_{\oSigma^*}(\calV)$ is continuous in each variable, the product on $\calU^0_{\oSigma^*}(\calV)$ is $\tensor{\ra}$-continuous with respect to this topology. 
	We define $\calU^1_{\overline{\Sigma}^*}(\calV)$
	to be the completion of $\calU^0_{\overline{\Sigma}^*}(\calV)$ along this topology. %This may be not a \QCC sheaf.
\end{definition}

The algebra $\calU^1_{\oSigma^*}(\calV)$ has the following universal property. Let $\calU$ be a  complete topological associative $\tensor{\ra}$-algebra and let $\grf:\Lie_{\oSigma^*}(\calV)\lra \calU$ be a continuous morphism of Lie algebras for which the composition $\varphi\circ \mathbf{u} : \calO_S \to \calU$ coincides with the unit morphism of $\calU$. Then $\grf$ extends uniquely to a continuous morphism of associative algebras from $\calU^1_{\oSigma^*}(\calV)$ to $\calU$.

\begin{remark}\label{rmk:UdaOLie}
We can give the same construction using the Lie algebra 
$\oLie_{\oSigma^*}$. It has an analogue universal properties and it follows that the  associative algebra constructed in this way is canonically isomorphic to the one constructed using $\Lie_{\oSigma^*}$
\end{remark}

To construct the desired enveloping algebra $\calU_{\overline{\Sigma}^*}(\calV)$ we must factor out some relations coming from the chiral algebra structure on $\calV$.
Let us note that there is a canonical morphism of $\pbarOpoli$-modules
\begin{align*}
	\Psi &: \calV_{\overline{\Sigma}^*} \to \calHomcont_{\calO_S}\left( \pbarOpoli, \calU^1_{\overline{\Sigma}^*}(\calV) \right) = \mF^1_{\Sigma,\calU^1_{\overline{\Sigma}^*}(\calV)}, \\ \Psi(v)(f) &= \overline{v \otimes f} \in \Lie_{\overline{\Sigma}^*}(\calV) \to \calU^1_{\overline{\Sigma}^*}(\calV).
\end{align*}

Informally, let's say that we want to impose minimal relations to $\calU^1_{\overline{\Sigma}^*}(\calV)$ so that $\Psi$ becomes a morphism of chiral algebras.

\begin{lemma}\label{lem:psilocalfields}
	The image of $\Psi$ consists of mutually local fields.
\end{lemma}

\begin{proof}
	Fix sections $X,Y \in \calV_{\overline{\Sigma}^*}$, we want to show that $\Psi(X),\Psi(Y)$ are mutually local. Consider the induced map $\mu_{X,Y} : \pbarOpoliq \to \Delta_!\calV_{\overline{\Sigma}^*}$ given by $f\otimes g \mapsto \mu( Xf\boxtimes Yg)$. Since the domain is supported on the diagonal there exists some power of the diagonal ideal $\calJ_\Delta \subset \pbarOpoliq$ such that $\mu_{X,Y}(\calJ^n_\Delta) = 0$. We claim that with the same $n$ we have $[\Psi(X),\Psi(Y)](\calJ_\Delta^n) = 0$. This follows from the fact that unravelling the definitions we have
	\[
	[\Psi(X),\Psi(Y)] = h^0 \circ \piuno \circ \mu_{X,Y}
	\]
	indeed, for any two functions $f,g \in \pbarOpoli$ we have
	\begin{align*}
	[\Psi(X),\Psi(Y)](f\otimes g) &= [\Psi(X)(f),\Psi(Y)(g)] \in \Lie_{\overline{\Sigma}^*}(\calV) \to \calU^1_{\overline{\Sigma}^*}(\calV)\\
	&= [\overline{X\otimes f},\overline{Y\otimes g}] = \overline{\piuno\left( \mu(Xf \boxtimes Yg) \right)} = h^0\left(\piuno(\mu_{X,Y}(f\otimes g))\right).
	\end{align*}
	The third equality follows from the discussion at the beginning of Section \ref{rmk:lievkfields}.
\end{proof}

We would like to consider now $\calW \subset \mF^1_{\Sigma,\calU^1_{\overline{\Sigma}^*}(\calV)}$ to be the chiral algebra over $\pbarOpoli$ generated by $\Psi(\calV_{\overline{\Sigma}^*})$, which by the previous Lemma consists of mutually local fields. In order to properly define $\calW$ and apply the construction of \cite[Section \ref{ssez:generatechiral}]{casmaffei1} we need to impose further assumptions on $\calV$, which we will assume from now on. We assume that $\calV$ is any of the chiral algebras $\calV_{\text{basic}+},\calV_{\text{com}+}$ generated by a subsheaf $\calG \subset \mF^1_{\Sigma,\calU}$ which satisfies the assumptions of loc. cit.. It follows from the proof of Lemma \ref{lem:psilocalfields} that $\Psi(\calG)$ satisfies the same assumptions of $\calG$. We may then define $\calW$ to be the $\pbarOpoli$-linear extension of the chiral algebra $\calW_{\text{com}}(\Psi(\calG)) \subset \mF^1_{\Sigma,\calU^1_{\overline{\Sigma}^*}(\calV)}$. In this way we get that the morphism $\Psi : \calV_{\overline{\Sigma}^*} \to \mF^1_{\Sigma,\calU^1_{\overline{\Sigma}^*}(\calV)}$ factors through $\calW$. 

\smallskip

Consider the following (non commutative!) diagram
\begin{equation}\label{eq:diagramcoeq}
\begin{tikzcd}{\calV_{\overline{\Sigma}^*}\exttensor{\fil}\calV_{\overline{\Sigma}^*}(\infty\Delta)} && {\calW\exttensor{\fil}\calW(\infty\Delta)} \\
\\
{\Delta_!\calV_{\overline{\Sigma}^*}} && {\Delta_!\calW} \\
&&& {\calU^1_{\overline{\Sigma}^*}(\calV)}
\arrow["\Psi\exttensor{\fil}\Psi", from=1-1, to=1-3]
\arrow["{\mu_{\calV_{\overline{\Sigma}^*}}}"', from=1-1, to=3-1]
\arrow["{\mu_{\mF}}", from=1-3, to=3-3]
\arrow["{\Delta_!\Psi}", from=3-1, to=3-3]
\arrow["{\text{ev}_{1}}", from=3-3, to=4-4]
\end{tikzcd}
\end{equation}
where $\mu_\mF$ is the product in the space of fields and $\text{ev}_1$ is obtained by evaluating at $1$ after embedding $$\Delta_!\calW \subset \calHomcont_S\left(\pbarOpoliq,\calU^1_{\overline{\Sigma}^*}(\calV)\right).$$

\begin{lemma}\label{rmk:coeqifmorphofchiral}
	Let $\calU'$ be another complete associative $\otimesr$-algebra, equipped with a continuous morphism of algebras $q : \calU^1_{\overline{\Sigma}^*}(\calV) \to \calU'$. Then the map
	\[
	\Psi_q : \calV_{\overline{\Sigma}^*} \to \calHomcont_S\left(\pbarOpoli,\calU^1_{\overline{\Sigma}^*}(\calV)\right) \xrightarrow{q_*} \calHomcont_S\left(\pbarOpoli,\calU'\right)
	\]
	is a morphism of chiral algebras between $\calV_{\overline{\Sigma}^*}$ and the $\pbarOpoli$-linearization of the chiral algebra generated by the image of $\Psi_q(\calG)$ if and only if $q$ coequalizes diagram \eqref{eq:diagramcoeq}. Here $q_*$ is just post composition with $q$. Recall that this makes sense since the image of $\Psi$ consists of mutually local fields and so the same goes for $\Psi_q$. 
\end{lemma}

\begin{proof}
	Denote by $\calW_q$ be the $\pbarOpoli$-linearization of the chiral algebra $(\calW_q)_{\text{com}}$ generated by the image of $\Psi_q(\calG)$ we have a commutative diagram
	\[\begin{tikzcd}
	{\calW\exttensor{\fil}\calW(\infty\Delta)} && {\calW_q\exttensor{\fil}\calW_q(\infty\Delta)} \\
	\\
	{\Delta_!\calW} && {\Delta_!\calW_q}
	\arrow["{q_*\exttensor{} q_*}", from=1-1, to=1-3]
	\arrow["{\mu_{\mF}}"', from=1-1, to=3-1]
	\arrow["{\mu_{\mF}}", from=1-3, to=3-3]
	\arrow["{\Delta_! q_*}"', from=3-1, to=3-3]
	\end{tikzcd}\]
	We extend diagram \eqref{eq:diagramcoeq} as follows:
	\[\begin{tikzcd}
	{\calV_{\overline{\Sigma}^*}\exttensor{\fil}\calV_{\overline{\Sigma}^*}(\infty\Delta)} && {\calW\exttensor{\fil}\calW(\infty\Delta)} && {\calW_q\exttensor{\fil}\calW_q(\infty\Delta)} \\
	\\
	{\Delta_!\calV_{\overline{\Sigma}^*}} && {\Delta_!\calW} && {\Delta_!\calW_q} \\
	\\
	&& {\calU^1_{\overline{\Sigma}^*}(\calV)} && {\calU'}
	\arrow["\Psi\exttensor{\fil}\Psi", from=1-1, to=1-3]
	\arrow["{\mu_{\calV_{\overline{\Sigma}^*}}}"', from=1-1, to=3-1]
	\arrow["{q_*\exttensor{\fil} q_*}", from=1-3, to=1-5]
	\arrow["{\mu_{\mF}}"', from=1-3, to=3-3]
	\arrow["{\mu_{\mF}}", from=1-5, to=3-5]
	\arrow["{\Delta_!\Psi}"', from=3-1, to=3-3]
	\arrow["{\Delta_! q_*}"', from=3-3, to=3-5]
	\arrow["{\text{ev}_1}"', from=3-3, to=5-3]
	\arrow["{\text{ev}_1}"', from=3-5, to=5-5]
	\arrow["q"', from=5-3, to=5-5]
	\end{tikzcd}\]
	Notice the rightmost squares are commutative. It immediately follows that if $\Psi_q =   q_*\circ\Psi$ is a morphism of chiral algebras then $q$ coequalizes diagram \eqref{eq:diagramcoeq}. 
	
	On the other hand assume that $q$ coequalizes diagram \eqref{eq:diagramcoeq}, consider an element $ A \in \calVstar\exttensor{\text{fil}}\calVstar (\infty\Delta)$ and let $$Z= \left(\Delta_!\Psi_q\circ \mu_{\calVstar} - \mu_{\mF}\circ (\Psi_q\exttensor{\text{fil}}\Psi_q)\right)(A) \in \Delta_!\calW_q \subset \mF^{2\loc}_{\Sigma,\calU'}.$$ 
	Our claim that $\Psi_q$ is a morphism of chiral algebras is equivalent to show that $Z=0$ for an arbitrary choice of $A$. Since the top squares are $\pbarOpoliq$-linear we have
	\[
	Z(f\otimes g) = \ev_1( f\otimes g \cdot Z) = \ev_1 \circ \left( \Delta_!\Psi_q\circ \mu_{\calVstar} - \mu_{\mF}\circ (\Psi_q\exttensor{\text{fil}} \Psi_q)\right) (f \otimes g \cdot A) = 0.
	\]
	It follows that $Z = 0$ and therefore the composition of the top squares is commutative.
\end{proof}

\begin{definition}\label{def:envelopingchiral2}
	We define $\calU_{\overline{\Sigma}^*}(\calV)$,\index{$\calU_{\overline{\Sigma}^*}(\calV)$} the complete enveloping algebra of $\calV$, to be the complete associative $\otimesr$-algebra which coequalizes the above diagram, so that we are equipped with a map $q : \calU^1_{\overline{\Sigma}^*} \to \calU_{\overline{\Sigma}^*}$ which is initial along morphisms of $\otimesr$ algebras $q' : \calU^1_{\overline{\Sigma}^*} \to \calU'$ such that
	\[
	q'\circ \text{ev}_1 \circ \left( \Delta_!\Psi \circ \mu_{\calV_{\overline{\Sigma}^*}} - \mu_{\mF}\circ\Psi\exttensor{\fil}\Psi \right) = 0.
	\]
	It may be constructed by dividing by the two sided ideal generated by the image of
	\begin{equation}\label{eq:defUV}
	\text{ev}_1 \circ \left( \Delta_!\Psi \circ \mu_{\calV_{\overline{\Sigma}^*}} - \mu_{\mF}\circ\Psi\exttensor{}\Psi \right) : \calV_{\overline{\Sigma}^*}\exttensor{\fil}\calV_{\overline{\Sigma}^*}(\infty\Delta) \to \calU^1_{\overline{\Sigma}^*}(\calV)
	\end{equation}
	and then taking the completion along the quotient topology. The assignment $\calV \mapsto \calU_{\overline{\Sigma}^*}(\calV)$ is functorial.
\end{definition}

\begin{remark}\label{rmk:UquozienteU0}
	If $\calG$ is a subsheaf of $\calF$ then the completion of $\calF/\calG$ is isomorphic to the completion of $\calF/\overline{\calG}$. Hence, in the construction of $\calU_{\oSigma^*}(\calV)$ we 
	can replace the images of $\calV_{\oSigma^*}\exttensor{\fil} \calV_{\oSigma^*}(\infty \Delta)$ along the map of formula \eqref{eq:defUV} by the union of the images of $\calV_{\oSigma^*}(n)\otimes \calV_{\oSigma^*}(n)(n\Delta)$. These images factors through $\calU^1_{\oSigma^*}(\calV)$, hence we can construct 
	$\calU_{\oSigma^*}(\calV)$ also as the completion of the quotient of 
	$\calU^1_{ \oSigma^*}(\calV)$.
\end{remark}

\begin{remark}By Lemma \ref{rmk:coeqifmorphofchiral}
	the composition
	\[
	\calV_{\overline{\Sigma}^*} \xrightarrow{\Psi} \calHomcont_S\left( \pbarOpoli, \calU^1_{\overline{\Sigma}^*}(\calV)\right) \xrightarrow{q\circ\_} \calHomcont_S\left( \pbarOpoli, \calU_{\overline{\Sigma}^*}(\calV)\right)
	\]
	has image consisting of mutually local fields and is compatible with the chiral product on both sides.
\end{remark}

\subsection{Description of \texorpdfstring{$\calU_{\oSigma^*}(\VkSg)$}{the enveloping algebra of the chiral algebra}}

Having the definition at hand we move on to describe the complete enveloping algebra $\calU_{\oSigma^*}(\VkSg)$. 

\begin{lemma}\label{lem:generationenvchiral}
	Assume that $\calG \subset \calV$ is a subsheaf which generates $\calV$ as a chiral algebra, by that we mean that $\calV$ may be obtained from $\calG$ by iterating the chiral product $\piuno\circ\mu$, taking derivations and finally taking the closure. Then, given any complete topological $\otimesr$ algebra $\calU'$, any two continuous morphisms of associative algebras $q,q' : \calU_{\overline{\Sigma}^*}(\calV) \to \calU'$,  which coincide on the image of $\calG \otimes \pbarOpoli \to \Lie_{\overline{\Sigma}^*}(\calV) \to \calU_{\overline{\Sigma}^*}(\calV)$ are actually equal.
\end{lemma}

\begin{proof}
	Notice that by construction of $\calU_{\overline{\Sigma}^*}$, it follows that if $q,q'$ coincide on $\Lie_{\overline{\Sigma}^*}(\calV)$ then they are equal. Therefore it suffices to prove that if $q,q'$ coincide on $\calG \to \calU_{\overline{\Sigma}^*}(\calV)$, they coincide on $\Lie_{\overline{\Sigma}^*}(\calV)$. In order to do so consider the filtration of $\calV_{\oSigma^*}$ inductively defined by $\tilde{\calV}_{\oSigma^*}(1) = \calG \otimes_{\calO_S}\pbarOpoli$ and $$\tilde{\calV}_{\oSigma^*}(n+1) = \piuno\mu\left( \tilde{\calV}_{\oSigma^*}(n)\boxtimes\tilde{\calV}_{\oSigma^*}(n)(\Delta) \right) + \tilde{\calV}_{\oSigma^*}(n)\Tan_{\overline{\Sigma}^*} + \tilde{\calV}_{\oSigma^*}(n)$$ We prove by induction on $n$ that the morphisms $q,q'$ coincide on $\tilde{\calV}_{\oSigma^*}(n) \to \calU_{\overline{\Sigma}^*}(\calV)$. This proves the Lemma since the union of the images of $\tilde{\calV}_{\oSigma^*}(n)$ is dense in $\calV_{\oSigma^*}$ by assumption and by that we get that $q,q'$ coincide when restricted to $\calV_{\oSigma^*} \to \calU_{\overline{\Sigma}^*}(\calV)$. The case $n = 1$ is tautological; so assume by induction that that $q,q'$ coincide on $h^0(\tilde{\calV}_{\oSigma^*}(n))$. The fact that $q,q'$ coincide on $h^0(\tilde{\calV}_{\oSigma^*}(n+1))$ follows by the fact that $\calU^1_{\overline{\Sigma}^*}(\calV) \to \calU_{\overline{\Sigma}^*}(\calV) \to \calU'$ (where $\calU_{\overline{\Sigma}^*}(\calV) \to \calU'$ may be $q$ or $q'$) coequalizes diagram \eqref{eq:diagramcoeq} and by the fact that $h^0(\tilde{\calV}_{\oSigma^*}(n)\Tan_{\overline{\Sigma}^*}) = 0$.
\end{proof}

%\subsubsection{rmk:alphabetaisoalgebre}

\begin{definition}\label{rmk:alphabetaisoalgebre}
	We construct canonical continuous morphisms of associative algebras
	\[
	U(\alpha): \ugsig{\kappa} \to \calU_{\overline{\Sigma}^*}(\VkSg) \qquad U(\beta): \calU_{\overline{\Sigma}^*}(\VkSg) \to \ugsig{\kappa}
	\]
	\begin{itemize}
		\item $U(\alpha)$ is constructed starting from the morphism of Lie algebras $\alpha$ of Lemma \ref{rmk:gktolievk} which extends to the completed enveloping algebras since it is continuous and the topologies on both algebras are both constructed taking left ideals along the topology of the Lie algebras;
		\item $U(\beta)$ is constructed analogously from the morphism of Lie algebras $$\beta : \Lie_{\overline{\Sigma}^*}(\VkSg) \to \ugsig{\kappa}$$ of Lemma \ref{rmk:actionlievku}. To upgrade $\beta$ to $\calU_{\overline{\Sigma}^*}(\VkSg)$ we proceed step by step and refer to Definitions \ref{def:envelopingchiral1},\ref{def:envelopingchiral2}.
		\begin{itemize}\begin{comment}
			\item By the universal property of the enveloping algebra $\beta$ extends to a morphism $U^{0'}(\beta): \calU(\Lie_{\overline{\Sigma}^*}(\VkSg)) \to \ugsig{\kappa}$;
			\item By construction $u : \calO_S \to \Lie_{\overline{\Sigma}^*}(\VkSg)$ is sent via $U^{0'}(\beta)$ to the unit of $\ugsig{\kappa}$ so that we get a morphism $U^0(\beta) : \calU^{0}_{\overline{\Sigma}^*}(\VkSg) \to \ugsig{\kappa}$;
			\item Since $\Lie_{\overline{\Sigma}^*}(\VkSg) \to \ugsig{\kappa}$ is continuous and both the topologies of $\calU^0_{\overline{\Sigma}^*}(\VkSg)$ are generated by left ideals we have that $U^0(\beta)$ is continuous so it induces $U^1(\beta) : \calU^1_{\overline{\Sigma}^*}(\VkSg) \to \ugsig{\kappa}$;
\am{questi tre punti li ho gia` sintetizzati in precendeza con la proprieta` universale di $U^1$, quindi lascerei solo quello che segue}\end{comment}
            \item By the universal property of $\calU^1_{\oSigma^*}$ the morphism $\beta$ extends to a morphism $U^{1}(\beta): \calU^1_{\overline{\Sigma}^*}(\VkSg) \to \ugsig{\kappa}$;
			\item Finally, we need to check that $U^1(\beta)$ coequalizes diagram \eqref{eq:diagramcoeq}. In order to do so, we use Lemma \ref{rmk:coeqifmorphofchiral} and notice that the composition
			\[
			\VkSg_{\overline{\Sigma}^*} \xrightarrow{\Psi} \calHomcont_S\left(\pbarOpoli,\calU^1_{\overline{\Sigma}^*}(\VkSg) \right)  \xrightarrow{U^1(\beta)_*} \calHomcont_S\left(\pbarOpoli, \ugsig{\kappa} \right)
			\]
			is the natural morphism $\VkSg_{\overline{\Sigma}^*} \to \mF^1_{\Sigma,\gog}$ induced by the inclusion $\VkSg \subset \mF^1_{\Sigma,\gog}$, so that it is a morphism of chiral algebras between $\VkSg_{\overline{\Sigma}^*}$ and its image. Indeed the composition above restricts to
			\[
			\VkSg_{\overline{\Sigma}^*} \xrightarrow{\Psi} \calHomcont_S\left(\pbarOpoli,\Lie_{\overline{\Sigma}^*}(\VkSg)\right)  \xrightarrow{\beta_*} \calHomcont_S\left(\pbarOpoli, \ugsig{\kappa} \right)
			\]
			which, if we denote by $\overline{X} \in \Lie_{\overline{\Sigma}^*}$ the class of an element $X\in \VkSg_{\overline{\Sigma}^*}$, reads as
			\[
			X \mapsto \left( f \mapsto \overline{fX}\right) \mapsto \left( f \mapsto \overline{fX}(1) = X(f) \right).
			\]
		\end{itemize}
	\end{itemize}
\end{definition}

\begin{proposition}\label{prop:descrenvelopingalg}
	$U(\alpha)$ and $U(\beta)$ are mutually inverse and establish a canonical isomorphism
	\[
	\Phi_{\Sigma} =U(\beta) : \calU_{\overline{\Sigma}^*}(\VkSg) \to \calU_\kappa(\hat{\gog}_\Sigma).
	\]\index{$\Phi_{\Sigma}$}
\end{proposition}

\begin{proof}
	It follows from the constructions that both composites are identities on subsheaves of generators, namely the subsheaves $\hat{\gog}_{\kappa,\Sigma} \to \ugsig{\kappa}$ and $\hat{\gog}_{\kappa,\Sigma} \to \calU_{\overline{\Sigma}^*}(\VkSg)$. The latter generates $\calU_{\overline{\Sigma}^*}(\VkSg)$ thanks to Lemma \ref{lem:generationenvchiral}.
\end{proof}

\subsection{Factorization Properties}
We want to construct factorization structures for $\Lie_{\overline{\Sigma}^*}$ and $\calU_{\overline{\Sigma}^*}$. For the precise definition of our factorization setting, we refer to Definitions \ref{def:openclosedfactorization} and \ref{def:factorizationcompletesheafsigma}. We analyze first the behaviour under pull backs and direct sums. 

\subsubsection{Direct sums}\label{ssec:LieUprodotti} In this section we analyse the behaviour of our constructions with respect to direct sum. Let $(\calV',\mu',u')$ and $(\calV'',\mu'',u'')$ be two chiral algebras. Define $\calV=\calV'\oplus \calV''$ with filtration given by the direct sum of the filtrations, $u=(u',u'')$ and the chiral product $\mu$ restricts to $\mu'$ and $\mu''$ respectively on $\calV'\exttensor{\mathrm{fil}}\calV'(\infty \Delta)$ and $\calV''\exttensor{\mathrm{fil}}\calV''(\infty \Delta)$ and is zero on $\calV'\exttensor{\mathrm{fil}}\calV''(\infty \Delta)$
and $\calV''\exttensor{\mathrm{fil}}\calV'(\infty \Delta)$ . In this setting we have
$\calV_{\oSigma^*}\simeq \calV_{\oSigma^*}'\oplus\calV_{\oSigma^*}''$, hence
 a natural isomorphism
$$
\Liepoli (\calV'\oplus \calV'')\simeq \frac{\Liepoli(\calV')\oplus \Liepoli(\calV'')}{(\mathbf{u}'=\mathbf{u}'')}.
$$
It follows immediately that 
$$
\calU^0_{\bar \Sigma ^*}(\calV'\oplus \calV'')\simeq \calU^0_{\bar \Sigma ^*}(\calV')
\otimes  \calU^0_{\bar \Sigma ^*}(\calV'').
$$
The topology on the left hand side is defined by the left ideals generated by the images of noz of $\Liepoli(\calV')$ and $\Liepoli(\calV'')$, Hence a fsnoz of the left hand side are of the form $U'\otimes \calU^0_{\bar \Sigma ^*}(\calV'') + \calU^0_{\bar \Sigma ^*}(\calV')\otimes U''$ where $U'$ and $U''$ are noz of $\calU^0_{\bar \Sigma ^*}(\calV')$ and $\calU^0_{\bar \Sigma ^*}(\calV'')$ respectively. Hence 
$$
\calU^1_{\bar \Sigma ^*}(\calV'\oplus \calV'')\simeq \calU^1_{\bar \Sigma ^*}(\calV')
\,\tensor {!} \, \calU^1_{\bar \Sigma ^*}(\calV'').
$$
Finally, by Remark \ref{rmk:UquozienteU0} we can realize $\calU_{\bar \Sigma ^*}(\calV)$
as the completion of the quotient of\ $\calU^1_{\bar \Sigma ^*}(\calV)$ by the ideal generated by elements in $\calV_{\bar\Sigma^*}(n)\boxtimes \calV_{\bar\Sigma^*}(n)(n\Delta)$. We denote this ideal by $\calR$ and by $\calR'$ and $\calR''$ the analogous ideals in $\calU^1_{\bar \Sigma^*}(\calV')$ and $\calU^1_{\bar \Sigma^*}(\calV'')$. By the definition of $\mu$ we have 
$$ \calR=\calR'\otimes \calU^1_{\oSigma^*}(\calV'') +\calU^1_{ \oSigma^*}(\calV')\otimes \calR'' . $$
Hence if we denote by $\calU'=\calU^0_{\oSigma ^*}(\calV')$ and $\calU''=\calU^0_{\oSigma ^*}(\calV'')$ then we have
\begin{align*}
	\calU_{\oSigma ^*}(\calV'\oplus \calV'')& \simeq
	\limpro \frac{\calU'\otimes \calU''}{U'\otimes \calU'' +\calU'\otimes U''+\calR'\otimes \calU''+\calU'\otimes \calR''}\\ & \simeq
	\limpro \frac{\calU'}{U'+\calR'}\otimes \frac{\calU''}{U''+\calR''}\\
	& \simeq \,\calU_{\oSigma ^*}(\calV')\, \tensor{!}\, \calU_{\oSigma ^*}( \calV'').
\end{align*}
where the limit is over $U'$ and $U''$ noz of $\calU'$ and $\calU''$.

\subsubsection{Pullbacks of chiral algebras}\label{ssec:LieUpullback}
Let $\grf:S'\lra S$ be a morphism of noetherian and quasi separated schemes and let $p':X'\lra S$ be the pullback of $p:X\lra S$ and $\Sigma'$ be the pullback of the sections $\Sigma$.  Let $\calV$ be a chiral algebra over $\oSigma$ such that $\calV=\limind\calV(n)$. We want to study the behaviour of the constructions of $\Lie_{\oSigma^*}$ and $\calU_{\oSigma^*}$ under  pullbacks. 

The following facts are easy to check and can be found in \cite{casmaffei1}.
\begin{enumerate}[\indent PR 1)]
	\item Let $\calF$ be a sheaf and let $\hat \calF$ be its completion. Let $\calG$ be a subsheaf of $\calF$ and $\calH$ be a subsheaf of $\hat \calF$ and assume that the image of $\calG$ in $\hat \calF$ is contained and is dense in $\calH$.  
	Then the completion of $\calF/\calG$ is isomorphic to the completion of $\calF/\overline \calG$ and to the completion of $\hat \calF/\calH$, in addition recall that $\hat\grf^*(\calF)\simeq \hat\grf^*(\hat \calF)$ (see \cite[Remark \ref{lemma31} and Lemma \ref{lem:pullbacklimits}]{casmaffei1});
	\item if $\calL$ and $\calM$ are topological sheaves on $S$, then $\hat\grf^*(\calL\tensor *\calM)\simeq 
	\hat\grf^*(\calL)\tensor *\hat\grf^*(\calM)$ as a topological sheaf and similarly for the $\tensor\ra$ product and the $\otimes^!$ product (see \cite[Proposition \ref{prop:pullbacktensorproduct}]{casmaffei1});
	\item if $\calF$ is a QCC sheaf, then $\hat\grf^*(\calF)$ is a QCC sheaf (see \cite[Remark \ref{rmk:qccpullback}]{casmaffei1});
	\item $\hat\grf^*$ commutes with direct sums (see \cite[Lemma \ref{lem:pullbacklimits}]{casmaffei1});
    \item $\hat \grf^*(\calF/\calG)$ is isomorphic to the completion of $(\hat \grf^*\calF)/(\hat\grf^*\calG)$ as a topological sheaf (here we abuse the notation a little bit since the map $\hat{\varphi}^*\calG \to \hat{\varphi}^*\calF$ does not need to be injective) (see \cite[Lemma \ref{lemma31}]{casmaffei1});
    \item For any inductive system $\calF_i$ the pullback $\hat \grf^*(\limind \calF_i)$ is isomorphic to the completion of $\limind \hat \grf^*\calF_i$ as a topological sheaf (see \cite[Lemma \ref{lem:pullbacklimits}]{casmaffei1});  
    \item $\hat\grf^*\Big( \Delta_{\oSigma,*}\calF\Big)\simeq \Delta_{\oSigma',*}\Big(\hat\grf^* \calF\Big)$ and if $\calF$ is a right $\calD_X$ module then $\hat\grf^*\Big( \Delta_{\oSigma,!}\calF\Big)\simeq \Delta_{\oSigma',!}\Big(\hat\grf^* \calF\Big)$ (see \cite[Remark \ref{rmk:pullbackpushdifferential}]{casmaffei1});
    \item if $\calF$ is a complete topological sheaf of $\pbarOpolim{2}$-modules then $\hat \grf^*\big(\calF(n\Delta_{\oSigma})\big) = \hat \grf^*\big(\calF\big)(n\Delta_{\oSigma'})$ as a topological sheaf (see \cite[Remark \ref{rmk:pullbackpolidiagonal}]{casmaffei1});
	\item $\hat\grf^* T_{\oSigma}\simeq T_{\oSigma'}$ and $\hat\grf^* \calD_{\oSigma}\simeq \calD_{\oSigma'}$ and $\hat\grf^*\pbarOmegapoli\simeq \Omega^1_{\overline{\Sigma'}^*}$ as topogical sheaves (see \cite[Remark \ref{rmk:pullbacktangentdifferential}]{casmaffei1});
\end{enumerate}

When dealing with pullback of a chiral algebra $\calV$, it is better to keep track of the filtered topological structure (i.e. the topology on $\calV(n)$), rather than the topological structure on the whole $\calV$. In order to deal with this we define the filtered pullback of $\calV$ with respect to $\grf$ as 
$$ \calV'=\limind \hat\grf^* \calV(n) $$
as a topological sheaf.

\begin{lemma}\label{lem:pullbackchiralalgebra}
	With the above notation, $\calV'$ comes with a natural filtration $\calV'(n) = \hat{\varphi}^*\calV(n)$. The chiral bracket $\mu$ of $\calV$ induces a morphism $$\mu' : \calV'\exttensor{\fil} \calV' (\infty\Delta') \to \Delta_{\oSigma',!}\calV',$$ in addition, the chiral unity morphism $u : \Omega^1_{\oSigma} \to \calV$ induces a morphism $u' : \Omega^1_{\oSigma'} \to \calV'$. With this data $\calV'$ is a chiral algebra.

	If in addition $\calV$ is constructed as a chiral algebra generated by mutually local fields in $\mF^1_{\Sigma,\calU}$ there is a natural morphism $\calV' \to \mF^1_{\Sigma',\calU'}$ which commutes with the chiral brackets and the unity morphisms.
\end{lemma} 

\begin{proof}
	Let $n,b$ be integers and suppose that $\mu$ induces a morphism $\mu_n : \calV(n)\otimesst \calV(n) (n\Delta_{\oSigma}) \to \Delta_{\oSigma,!}^{\leq m}\calV(m)$, then by combining PR2, PR7 and PR8, the pullback of $\mu$ identifies with a morphism
	$$\mu'_n : \calV'(n)\otimesst \calV'(n) (n\Delta_{\oSigma'}) \to \Delta_{\oSigma',!}^{\leq m}\calV'(m),$$
	analogously, PR9 induces a morphism $u' :\Omega^1_{\oSigma'} \to \calV'$. 
	The map $\mu'_n$ determines the chiral product $\mu:\calV\tensor{\fil} \calV' (\infty\Delta_{\oSigma'}) \to \Delta_{\oSigma',!}\calV'$.
	The fact that these satisfy the axioms of a chiral algebra follows by functoriality. The claim about $\mF^1$ follows by the fact that the its chiral bracket is defined using the multiplication structure on $\calU$ and the canonical morphisms $\pbarOpolim{2}(\infty\Delta) \to \pbarOpoli \otimesr \pbarOpoli$, which behave well under pullback. 
\end{proof}

\begin{remark}
	Assume  that, locally on $S$, there exists \QCC sub $\calO_S$-modules $\calW_n$ of $\calV$ such that $\calV(n)=\bigoplus_{i=0}^{n}W_i$. In this case, if we set $\calW'_n=\hat\grf^*(\calW_n)$  we have that $\calV'=\bigoplus_i \calW'_i$ and 
$\calV'(n)=\bigoplus_{i=0}^n \calW'_i$. Under this assumption the sheaves $\calV'(n)$ are closed subsheaves of $\calV'$ with the induced topology and $\calV'$ is their colimit. This assumption is satisfied for  $\calV=\calVg$. 
\end{remark}

\subsubsection{Pullback of \texorpdfstring{$\Lie_{\oSigma^*}$}{the Lie algebra of Fourier coefficients}}
%We assume that the hypotheses of \ref{sss:pullbackV} is satisfied. 

We want to compare $\calU_{\overline{\Sigma'}^*}(\calV')$ and the pullback of $\calU_{\oSigma^*}(\calV)$. To do that that we need to study first the behaviour of $\Lie_{\oSigma^*}$ under pullbacks. 

\begin{proposition}\label{prop:pullbackLie}
	Let $X,S,\Sigma$ be as usual and let $\calV$ be a chiral algebra over $\pbarO$ with a filtration of \QCC sheaves $\calV(n)$ such that $\calV = \varinjlim_n \calV(n)$ as topological sheaves. Let $\varphi: S' \to S$ be a morphism of quasi-separated schemes. With the above notation we have
$$
\hat \grf^* \Big( \oLie_{\oSigma^*}(\calV) \Big)\simeq \oLie_{\overline{\Sigma'}^*}(\calV')
$$
\end{proposition}

\begin{proof}
We prove first that $\hat{\grf}^*(h^0(\calV_{\overline{\Sigma'}^*}))$ is the completion of 
$h^0(\calV_{\oSigma^*}')$. Let us start by recalling that by definition, as topological sheaves, we have
$$
h^0(\calV_{\oSigma^*}) = \limind_{n} \frac{\calV_{\oSigma^*}(n)}{\calV_{\oSigma^*}(n)\cap \left(\calV_{\oSigma^*}\cdot T_{\oSigma^*}\right)}
$$
Since filtered colimit of sheaves are exact and $\calV_{\oSigma^*}=\limind\calV_{\oSigma^*}(n)$ we deduce
$$
h^0(\calV_{\oSigma^*}) = \limind_{n,m} \frac{\calV_{\oSigma^*}(n)}{\calV_{\oSigma^*}(n)\cap \left(\calV_{\oSigma^*}(m)\cdot T_{\oSigma^*}\right)}
$$
For all $m$ we choose $n_m\geq m$ such that $\calV_{\oSigma^*}(m)\cdot T_{\oSigma^*}\subset \calV_{\oSigma^*}(n_m)$.
The couples $(n_m,m)$ are cofinal in $\mN\times\mN$ hence we have 
$$ h^0(\calV_{\oSigma^*}) = \limind_{m} \frac{\calV_{\oSigma^*}(n_m)}{\calV_{\oSigma^*}(n_m)\cap \left(\calV_{\oSigma^*}(m)\cdot T_{\oSigma^*}\right)}
= \limind_{m} \frac{\calV_{\oSigma^*}(n_m)}{ \calV_{\oSigma^*}(m)\cdot T_{\oSigma^*}}. $$
Hence,  property PR5 of Sections \ref{ssec:LieUpullback} we have that $\hat \grf^*\big(h^0(\calV_{\oSigma^*})\big)$ is the completion of 
$$
\limind _m \frac{\hat \grf^* \calV_{\oSigma^*}(n_m)}{\hat\grf^*\Big(\calV_{\oSigma^*}(m)\cdot T_{\oSigma^*}\Big)}=\limind _m 
\frac{\calV_{\oSigma^*}'(n_m)}{\hat\grf^*\Big(\calV_{\oSigma^*}(m)\cdot T_{\oSigma^*}\Big)}.
$$
Finally by properties PR2 and PR8  above we have that the denominator contains $\calV_{\overline{\Sigma'}^*}'(m)\cdot T_{\overline{\Sigma'}^*}$ and is contained in the closure of the latter in $\calV_{\oSigma^*}'(n_m)$. Hence, by PR1, the completion of this colimit is isomorphic to the completion of $h^0(\calV_{\oSigma^*}')$.

Similarly (although we do not need to take any filtration) we see that $\hat\grf^*h^0(\pbarOmegapoli)$ is the completion of $h^0(\Omega^1_{\overline{\Sigma'}^*})$. Now the claim about $\Lie_{\oSigma^*}$  follows from $\hat\grf^*\calO_S=\calO_{S'}$,
$\hat \grf^*u=u'$ and $\hat\grf^*\Res_\Sigma=\Res_{\Sigma'}$ together with PR4-5.  
\end{proof}

\subsubsection{Pullback of \texorpdfstring{$\calU_{\oSigma^*}$}{the enveloping algebra}}
\begin{comment}
We can use the Lie algebra $\calL$ introduced in the previous sectione to construct an analogue of the algebra $\calU_{\oSigma'}(\calV')$.

We define $\calU^0$ as the enveloping algebra of $\calL$ quotiented by the ideal generated by $(\tilde {\mathbf u}-\boldsymbol {1})$. Then we define $\calU^1$ as the completion of $\calU^0$ along the topology whose a fsonoz is given by the left ideal generated by noz of $\calL$. 

\begin{lemma}\label{lemma33}\hfill
	
\begin{enumerate}[\indent a) ]
\item 	$\displaystyle{\calU^1 \simeq \hat\grf^*\calU^1_{\oSigma^*}(\calV)}$
\item  	$\displaystyle{\calU^1 \simeq \calU^1_{(\oSigma')^*}(\calV')}$
\end{enumerate}
\end{lemma}

\begin{proof}
	Point a) follows from Proposition \ref{prop:LLie1}.
	
	To prove point b) notice that in the construction of $\calU^1$ we can replace $\calL$ with its completion. Hence the claim follows from Proposition \ref{prop:LLie2} and Remark. 
\end{proof}
\end{comment}
We now study the behaviour of $\calU_{\oSigma^*}$ under pullbacks. 
\begin{lemma}
	Let $X,S,\Sigma$ be as usual and let $\calV$ be a chiral algebra over $\pbarO$ with a filtration of \QCC sheaves $\calV(n)$ such that $\calV = \varinjlim_n \calV(n)$ as topological sheaves. Let $\varphi: S' \to S$ be a morphism of quasi-separated schemes. With the above notation we have
$$
\calU^1_{\overline{\Sigma'}^*}(\calV')\simeq \hat \grf^*\left(\calU^1_{\oSigma^*}(\calV)\right).
$$
\end{lemma}
\begin{proof}
Let $\calL = \grf^*\Lie_{\oSigma^*}(\calV)$ (this is the usual pullback) and recall that its completion, by Proposition \ref{prop:pullbackLie}, is isomorphic to $\oLie_{\overline{\Sigma'}^*}(\calV')$. Denote by $\calU(\calL)$ and $\calU^0(\calL)$ the enveloping algebra of $\calL$, and its quotient by $\mathbf u-\mathbf{1}$ respectively (where $\mathbf{1}$ is the unit morphism of $\calU(\calL)$). Let $\calU^1(\calL)$ be the completion of $\calU^0(\calL)$.
It follows by Remark \ref{rmk:UdaOLie} that $\calU^1_{\overline{\Sigma'}^*}(\calV')$ is isomorphic to $\calU^1(\calL)$,.

It is easy to see (for example from the fact that $\grf^*$ is a left adjoint) that non completed pullback commutes with taking enveloping algebras, hence $\calU(\calL)$ is the pullback of the enveloping algebra of $\Lie_{\oSigma^*}(\calV)$. Since $\grf^*$ is right exact we have also 
that $$\calU^0(\calL)\simeq \grf^*(\calU^0  _{\oSigma^*}(\calV) ).$$ 
Finally a fsonoz in the right hand side are given by pullback of left ideals generated by noz of $\Lie_{\oSigma^*}(\calV)$. Since $\grf^*$ commutes with tensor product and preserves images of sheaves, these are exactly the left ideals generated by noz of $\calL$. Hence the completions for these topologies are isomorphic. 
\end{proof}

\begin{proposition}\label{prop:pullbackU}
Let $X,S,\Sigma$ be as usual and let $\calV$ be a chiral algebra over $\pbarO$ with a filtration of \QCC sheaves $\calV(n)$ such that $\calV = \varinjlim_n \calV(n)$ as topological sheaves. Let $\varphi: S' \to S$ be a morphism of quasi-separated schemes. With the above notation we have
$$\hat \grf^*(\calU_{\oSigma^*}(\calV))\simeq \calU_{\overline{\Sigma'}^*}(\calV').$$	
\end{proposition}

\begin{proof}
Recall that given a field $X:\pbarOpoli\lra \calU^1_{\oSigma^*}(\calV) $ its pullback is a map
$\hat \grf^* X:\hat \grf^*\pbarOpoli\lra \hat \grf^*\calU^1_{\oSigma^*}(\calV)$. Hence, using the previous Lemma is a field from $\calO_{\overline{\Sigma'}^*}$ to $\calU^1_{\overline{\Sigma'}^*}(\calV')$; this induces a map $$\varphi^*_{\mF^1} : \hat{\varphi}^*\mF^1_{\Sigma,\calU^1_{\oSigma^*}(\calV)} \to \mF^1_{\Sigma',\calU^1_{\overline{\Sigma'}^*}(\calV')}.$$ By construction, using the notation of Section \ref{ssec:envelopingalgchiralalg}, the following diagram is commutative:
$$
\xymatrix{
\hat\grf^*\big(\calV \ar[rr]^-{\hat\grf^*\Psi} \ar[d]^{\simeq}\big) && \hat\grf^*\big(\mF^1_{\Sigma,\calU_{\oSigma^*}^1(\calV)}\big) \ar[d]^{\grf^*_{\mF^1}}\ar[rr]^{\ev_1} &&
\hat\grf^*(\calU_{\oSigma^*}^1(\calV))\ar[d]^\simeq
\\
\calV' \ar[rr]^{\Psi'} && 
\mF^1_{\Sigma',\calU_{\overline{\Sigma'}^*}^1(\calV')}\ar[rr]^{\ev_1} &&
\calU_{\overline{\Sigma'}^*}^1(\calV')}
$$
and by Lemma \ref{lem:pullbackchiralalgebra} the morphism $\Psi'$ is compatible with the chiral brackets. Now recall that $\calU_{\oSigma^*}(\calV)$ is defined as the completion of the quotient of $\calU^1_{\oSigma^*}(\calV)$ by the relations given by formula \eqref{eq:defUV}. Let us formulate this fact differently. Let  $\psi_n
 : \calV_{{\oSigma}^*}(n)\exttensor{\fil}\calV_{\overline{\Sigma}^*}(n)(n\Delta) \to \calU^1_{\overline{\Sigma}^*}(\calV)
$ be the map given by formula \eqref{eq:defUV}, and let $\psi=\sum_n \psi_n$ be the coproduct of these maps. Define
$$\calR(\calV):  \calU^1_{\overline{\Sigma}^*}(\calV) \otimesr
\left(\bigoplus_{n \geq 0}\, \Big(\calV_{\oSigma^*}(n)\otimesst \calV_{\oSigma^*}(n)\Big)(n\Delta_X)\right)
\otimesr \calU^1_{\overline{\Sigma}^*}(\calV) \lra \calU^1_{\oSigma^*}(\calV)$$
as $\calR(\calV)(a\otimes ( \sum v_n )\otimes b)=a\cdot (\sum \psi_n(v_n ))\cdot b$. 
By Remark \ref{rmk:UquozienteU0} and PR6 the algebra $\calU^1_{\oSigma^*}(\calV)$ is the completion the cokernel of these map. 

By property PR2, the commutativity of the diagram above, the previous Lemma and the fact that the pull back of $\mu_\mF$ is equal to $\mu_{\mF}$ (see \cite[ Lemma \ref{lemma:pullbackoffields}]{casmaffei1}), we see that the pull back of $\calR(\calV)$ is equal to $\calR(\calV')$. Hence, our thesis, follows from property PR5 of Section \ref{ssec:LieUpullback}. 
\end{proof}

\subsection{Factorization properties}\label{ssec:factorizationU}
Recall Definitions \ref{def:openclosedfactorization}, \ref{def:factorizationcompletesheafsigma}. Assume  that the chiral algebra $\calV$ is \QCCF, satisfies the assumption of Proposition \ref{prop:pullbackU} and that it has a factorization structure, so that for every $\pi:J\surjmap I$ we have isomorphisms
\[
\Ran^{\calV}_{J/I} : \hat{i}_{J/I}^*\calV_{\Sigma_J} \to \calV_{\Sigma_I}, \qquad \fact^{\calV}_{J/I} :  \hat{j}_{J/I}^*\left(\prod_{i\in I} \calV_{\Sigma_{J_i}}\right) \to \hat{j}_{J/I}^* \calV_{\Sigma_J}.
\]
Then we can apply Proposition \ref{prop:pullbackU} to $\grf=i_{J/I}$ and we obtain an isomorphism
$$
\hat i^*_{J/I}\big(\calU_{\oSigma^*_J}(\calV_{\Sigma_J})\big)\simeq \calU_{\oSigma^*_I}(\hat i _{J/I}\calV_{\Sigma_J})\simeq \calU_{\oSigma^*_I}(\calV_{\Sigma_I}).
$$
Similarly applying Proposition \ref{prop:pullbackU} to $\grf=j_{J/I}$ and the results of Section \ref{ssec:LieUprodotti} we obtain an isomorphism
$$
\hat j^*_{J/I}\big(\calU_{\oSigma^*_J}(\calV_{\Sigma_J})\big)\simeq \calU_{\oSigma^*_I}(\hat j _{J/I}\calV_{\Sigma_J})\simeq \calU_{\oSigma_I^*}\big(\prod_{i\in I} \calV_{\Sigma_{J_i}}\big)\simeq \tensor{!}\; \calU_{\oSigma_{J_i}^*}(\calV_{\Sigma_{J_i}}).
$$
Hence the associative algebra $\calU_{\oSigma^*}(\calV)$ has a factorization structure. This discussion in particular applies to the case of $\calV=\calV_\Sigma^\kappa(\gog)$ (see \cite[Proposition \ref{prop:factpropertieschiralalg}]{casmaffei1} and Section \ref{ssec:richiamifactchiral}).

\begin{proposition}\label{prop:factpropertiescalu}
	Assume that $S$ is integral, noetherian and quasi-separated. With respect to notation \ref{ntz:factorization}, there are natural isomorphisms
		\begin{align*}
		\Ran^{\calU (V)}_{J/I} &: \hat{i}_{J/I}^*\left(\calU_{\overline{\Sigma}_{J}^*}(\calV^\kappa_{\Sigma_{J}}(\gog))\right) \to \calU_{\overline{\Sigma}_{I}^*}(\calV^\kappa_{\Sigma_{I}}(\gog)) \\
		\fact^{\calU (V)}_{J/I} &: \hat{j}_{J/I}^*\left( \bigotimes^!_{i\in I} \calU_{\overline{\Sigma}_{I_j}^*}(\calV^\kappa_{\Sigma_{J_i}}(\gog)) \right) \to \hat{j}_{J/I}^*\left(\calU_{\overline{\Sigma}_{J}^*}(\calV^\kappa_{\Sigma_{J}}(\gog))\right)
		\end{align*}
		which make $\calU_{\overline{\Sigma}^*}(\VkSg)$ into a complete topological factorization algebra. The map $\Phi_{\Sigma}$ of Proposition \ref{prop:descrenvelopingalg} preserves the factorization structures, where we equip $\calU_\kappa(\hat{\gog}_\Sigma)$ with the structure of  \cite[Proposition \ref{prop:factpropertiesgogcalugog}]{casmaffei1} (see also Section \ref{ssec:richiamiaffinealgebra}).
\end{proposition}

\begin{proof}
	The first claim is a particular case of the previous discussion.  Since the factorization structure of $\calU_{\oSigma^*}(\VkSg)$ is built from that of $\VkSg$, to show that $\Phi_{\Sigma}$ and the factorization structure on $\calU_\kappa(\hat{\gog}_\Sigma)$, it is enough to check that the map $\ev_1 : (\VkSg)_{\oSigma^*} \to \calU_{\kappa}(\hat{\gog}_\Sigma)$ is compatible with the factorization structures. Let us recall that this map is the composition of $(\VkSg)_{\oSigma^*} \to \mF^1_{\Sigma,\gog} \to \calU_{\kappa}(\hat{\gog}_\Sigma)$, where the second one is evaluation at $1 \in \pbarOpoli$. The factorization structure on $\VkSg$ is induced by that of $\mF^1_{\Sigma,\gog}$ (see Sections \ref{ssec:richiamifactfields},\ref{ssec:richiamifactchiral}) so that the fact that evaluation at $1 \in \pbarOpoli$ is compatible with the factorization structures is evident.
\end{proof}

\section{Opers on \texorpdfstring{$\overline{\Sigma}$ and $\overline{\Sigma}^*$}{the formal neighborhood}}\label{sec:opersigma}

We adapt the definition of opers to our $\overline{\Sigma}$ geometric setting. We refer to \cite{casarin2025bundle}, for the known material from which we will adapt the definitions. We fix as always a smooth family of curves on $X \to S$, a set of sections $\Sigma$ and we study opers on $\overline{\Sigma}$ and $\overline{\Sigma}^*$. We will first focus on the case of $\overline{\Sigma}$, which we treat as a formal scheme, and then extend our definitions to $\overline{\Sigma}^*$. This covers the special case $S= \Spec \mC$, $\calO_{\overline{\Sigma}}= \mC[[t]]$ as well. 

Let us say some words to justify the definitions we will give in the following pages by recalling some known constructions. Recall that to any smooth curve $C$ over $\mC$ there is a canonical $\Autpiu{} O$ bundle called $\Aut_C$ (see \cite[Section 6.5]{frenkel2004vertex} or \cite{casarin2025bundle}). It is known (c.f. \cite[Proposition 3.3.3]{casarin2025bundle}) that, for any $(\goF,\nabla)$ oper on $C$ the $B$-bundle $\goF$ is always isomorphic to a certain bundle $\goF_0$, which is a twist of the bundle $\Aut_C$ for a fixed morphism $\Autpiu{} O \to B$. Thus, in order to give appropriate definitions in our $\oSigma,\oSigma^*$ setting, we need to treat $\overline{\Sigma}$ as a more geometric object, so that we may construct an analogue of the bundle $\goF_0$ on it and speak about connections on $\goF_0$. In order to do so we will refer to \cite[Section \ref{ssec:canonicaltorsor}]{casmaffei1} where $\Aut_{\oSigma}$, an analogue of the torsor $\Aut_C$, is constructed. The reader can safely skip this section, as definitions are completely analogous to the case of a smooth curve over $\mC$, we give the constructions for completeness.

\subsection{Recollections on spaces and on the \texorpdfstring{torsor $\Aut_{\oSigma}$}{canonical torsor}}

We start by recalling some terminology developed in \cite[Sections \ref{ssec:recollectionsauto}, \ref{ssec:spacesjetscanonicalbundle}]{casmaffei1}.

\subsubsection{Spaces}

We refer to \cite[Section \ref{sssec:spaces}]{casmaffei1} for our geometric setting and consider the category of sheaves of sets on $\Aff_S$ (the category of affine schemes with a map to $S$) for the Zariski topology, which we denote by $\Sp_S$. We will be referring to the objects of this category simply as \emph{spaces}. The category of spaces is naturally equivalent to the category of sheaves for the Zariski topology on $\Sch_S$, so that it makes sense to evaluate $X \in \Sp_S$ on an arbitrary scheme $T \to S$.

To any space $X$ there there is attached a sheaf on $S$ of functions $\Fun(X)(U) = \Hom_{\Sp_S}(X|_U,\mA^1|_U)$ for any Zariski open subset $U \subset S$. The topological sheaf $\pbarO$ determines a space $\oSigma$, which, for any $\Spec R \to S$ satisfies
\[
	\oSigma(R) = \Homcont_R\left( \calO_{\oSigma_R},R\right),
\]
where $\calO_{\oSigma_R}$ is the completed pullback of $\pbarO$ along $\Spec R \to S$. We have that $\oSigma$ is an ind-affine scheme over $S$ and that $\Fun(\oSigma) = \pbarO$.

It makes sense to speak about vector bundles on spaces just in the same way as in the case of an ordinary scheme, but we require local triviality on $S$. So a vector bundle on a space $X$ is a space $V \to X$ such that there exists an open Zariski cover $\cup\, U_i = S$ and isomorphism $V|_{X|_{U_i}} \simeq X|_{U_i} \times_S \mA^r_{U_i}$ for which the transition maps are linear. In particular we have a notion of vector bundles on $\oSigma$ and it is easy to check that they correspond to locally free sheaves of $\pbarO$-modules on $S$.

Analogously to the case of vector bundles, there is a well defined notion of a $G$-torsor on a space $X$, for any group sheaf $G : \Aff^{\op}_S \to \mathrm{Grp}$. Again, we require that our torsors be locally trivial for the Zariski topology on $S$. In the case where $G$ is the pullback of an affine group scheme over $\mC$ and $V$ is a $G$-representation, given any $G$ torsor $\goF$ on $X$ it is possible to construct the twisted vector bundle $V_\goF$ by taking the quotient of $\goF \times V$ along the diagonal $G$ action. 

\subsubsection{The canonical torsor}

Recall the definition of the group schemes $\Aut O$, $\Autzero{} O$ defined over $\mC$ as 
\begin{align*}
	\Aut O (R) &= \Aut^{\cont}_R(R[[z]]); \\
	\Autpiu{} O (R) &= \{ \varphi \in \Aut O(R) \text{ such that } \varphi(z) \in (z) \}.
\end{align*}

Attached to any smooth curve $C$ over $\mC$, there is a canonical $\Autpiu{} O$-torsor called $\Aut_C$ (see for instance \cite[Section 6.5]{frenkel2004vertex}) which is a subscheme of the jet scheme $JC$. In \cite[Sections \ref{sssec:jetsonspaces}, \ref{ssec:canonicaltorsor}]{casmaffei1} it is shown that these constructions can be performed also in the case where we replace $C$ with $\oSigma$.

Indeed, it is possible to apply the jet construction to the case of spaces as well (see \cite[Section \ref{sssec:jetsonspaces}]{casmaffei1}), defining, for a space $X$, 
\[
	J_nX(R) = X(R[z]/z^{n+1}) \quad \text{ and } \quad JX = \varprojlim J_nX.
\]
In the case where $X = \oSigma$, it turns out that $J_2\oSigma$ is a vector bundle and its sheaf of sections identifies with $T\oSigma$. In addition, $J\oSigma$ identifies with the functor $R \mapsto \Homcont_R(\calO_{\oSigma_R},R[[z]])$ so that there is a natural left action of $\Autpiu{S} O$ on $J\oSigma$ for which the projection map $J\oSigma \to \oSigma$ is invariant. 

It is then possible to define a subfunctor $\Aut_{\oSigma} \subset J\Sigma$ which is stable for the above $\Autpiu{S} O$ action and which is an $\Autpiu{S} O$ torsor over $\oSigma$. Any local coordinate $t$ induces a trivialization 
\[
	\triv_t : \oSigma \times \Autpiu{S} O \to \Aut_{\oSigma}
\]
and one can show that for any two local coordinates $t,s$, the transition isomorphism $\triv_s^{-1}\triv_t$ is determined by the following element of $\Autpiu{S} O (\pbarO(S))$ (see \cite[Proposition \ref{prop:changecoordinate}]{casmaffei1}):
\[
	\triv^{\univ}_{st} = \sum_{k \geq 1} \frac{1}{k!}(\partial_t^ks)z^k.
\]

\subsection{\texorpdfstring{$G$}{G}-opers on \texorpdfstring{$\overline{\Sigma}$}{the formal neighborhood} and \texorpdfstring{$\oSigma ^*$}.}

In this section we introduce connections on $G$-torsors and opers over $\overline{\Sigma}$. 

\begin{definition}[Atiyah's bundle]
	Given a $G$ torsor $\goF$ on $\overline{\Sigma}$ the space $\calE_{\goF} = (T\goF)/G$ is naturally a vector bundle over $\overline{\Sigma}$ (recall that here $T\goF = J_2\goF$). In the (local) case where $\goF \simeq \overline{\Sigma} \times G$ is trivial we have
	\[
		\calE_{\goF} \simeq T\overline{\Sigma} \times \gog,
	\]
	where $\gog$ is the Lie algebra of $G$.
\end{definition}

\begin{definition}\label{def:connectionsigmabar}
	Let $\goF$ be a $G$ torsor on $\overline{\Sigma}$. Then there is a canonical exact sequence of locally free $\pbarO$-modules
	\[
		0 \to \gog_{\goF} \to \calE_{\goF} \to \Tan_{\overline{\Sigma}} \to 0.
	\]
	A \emph{connection} on $\goF$ is an $\pbarO$-linear section $\Tan_{\overline{\Sigma}} \to \calE_{\goF}$. The set of connections is naturally a sheaf for the Zariski topology on $S$, to be denoted by $\Conn(\goF)$. This is naturally a $\gog_{\goF}\otimes_{\pbarO}\Omega^1_{\overline{\Sigma}}$ torsor (as sheaves on $S$).
\end{definition}

\subsubsection{Opers on \texorpdfstring{$\overline{\Sigma}$}{ the formal neighbourhood}}

Consider two groups $B \subset G$ and let $\goF_B$ be a $B$ torsor on $\overline{\Sigma}$. Consider the induced $G$-torsor $\goF_G$ and the natural inclusion $\goF_B \subset \goF_G$. This induces a diagram
	\[\begin{tikzcd}
	0 & {\gob_{\goF_B}} & {\calE_{\goF_B}} & {\Tan_{\overline{\Sigma}}} & 0 \\
	0 & {\gog_{\goF_B}} & {\calE_{\goF_G}} & {\Tan_{\overline{\Sigma}}} & 0
	\arrow[from=1-1, to=1-2]
	\arrow[from=1-2, to=1-3]
	\arrow[hook, from=1-2, to=2-2]
	\arrow[from=1-3, to=1-4]
	\arrow[hook, from=1-3, to=2-3]
	\arrow[from=1-4, to=1-5]
	\arrow["{=}"{description}, from=1-4, to=2-4]
	\arrow[from=2-1, to=2-2]
	\arrow[from=2-2, to=2-3]
	\arrow[from=2-3, to=2-4]
	\arrow[from=2-4, to=2-5]
\end{tikzcd}\]
so that we have a natural inclusion $\Conn(\goF_B) \subset \Conn(\goF_G)$. Recall that under our assumptions $\goF_B$ is locally trivial for the Zariski topology on $S$. Then there is a canonical map $c : \Conn(\goF_G) \to (\gog/\gob)_{\goF_B}\otimes_{\pbarO}\Omega^1_{\overline{\Sigma}}$ which is constructed as in the usual case: given a connection $\nabla$, the element $c(\nabla) \in (\gog/\gob)_{\goF_B} \otimes \Omega^1_{\overline{\Sigma}}$ is constructed locally on $S$, on any open subset in which $\goF_B$ is trivial is defined by taking an arbitrary connection $\nabla_{\gob}$ on $\goF_B$ (which locally exists) and then computing $\nabla - \nabla_{\gob} \text{ mod } \gob_{\goF_B}\otimes\Omega^1_{\overline{\Sigma}}$, which does not depend on the choice of $\nabla_{\gob}$. The morphism $c$ is equivariant with respect to the action of $\gog_{\goF}\otimes\Omega^1_{\overline{\Sigma}}$ and therefore surjective as a morphism of sheaves on $S$.

\medskip

We will focus from now on the case where $\gog$ is a finite simple Lie algebra, $G$ is the attached algebraic group of adjoint type and $B$ is a Borel subgroup. We will consider only a specific $B$-torsor on $\overline{\Sigma}$, the analogue of the torsor $\goF_0$ of Proposition 3.3.3 of \cite{casarin2025bundle}. To recall its construction let us set up some notations. 

Fix a maximal torus $H \subset B$ and consider the induced root system $\Phi$ and root decomposition $\gog = \goh\oplus \oplus_{\alpha \in \Phi} \gog^\alpha$, where $\goh$ is the Lie algebra of $H$. The choice of $B$ induces a set of simple positive roots which we call $\Delta = \{ \alpha_i\}$. Consider $h_0 \in \goh$ to be the twice the sum of the fundamental coweights, so that $h_0$ acts on $\gog^\alpha$ by $2|\alpha|$ (here if $\alpha = \sum n_i\alpha_i$ we set $|\alpha| = \sum n_i$). Fix a principal nilpotent element $f_0 \in \oplus_{i} \gog^{-\alpha_i}$, from such data one can extract an $\mathfrak{sl}_2$-triple $\{f_0,h_0,e_0\} \subset \gog$, this determines a morphism $r_1:SL(2)\lra G$ which sends the standard $\mathfrak{sl}_2$-triple to $\{f_0,h_0,e_0\}$. 
Since $G$ is  of adjoint type, $\frac{1}{2}h_0$ may be exponentiated to a morphism $\check{\rho} : \mG_m \to H$, the principal cocharacter and $r_1$ factors through a morphism $r_2:PGL(2)\lra G$. We denote by $r$ the restriction of $r_2$ to the standard Borel $(B_2)_{\mathrm{ad}}$ of $PGL(2)$ which we can represent as follows:
%and moreover there exists a morphism $r_1:PGL(2)\lra G$ such that $r\circ \omega\spcheck=\check\rho$, where $\omega\spcheck$ is the standard fundamental coweight of $PGL(2)$ and such that $Lie(r)(e_0)=e$ and . We define $r$ to be the restriction of $r$ to the $B_2$
%We consider the exponentiation of $e_0$ as well, to be denoted by $e : \mG_a \to B \subset G$. Together these two morphisms induce a map $r : (B_2)_{\mathrm{ad}} \to B$, where $(B_2)_{\mathrm{ad}}$ is the adjoint version of the group of upper triangular matrices in $\mathrm{SL}_2$:
\[
	(B_2)_{\mathrm{ad}} \simeq \left\{ \begin{pmatrix} a & b \\ 0 & 1 \end{pmatrix} \right\}.
\]

Notice that there is a natural identification $\Autpiu{3} O \simeq (B_2)_{\mathrm{ad}}$, obtained as follows:
\[
	\left( t \mapsto at + bt^2 \right) \mapsto \begin{pmatrix}
		a & b/a \\ 0 & 1
	\end{pmatrix}
\]

Via these constructions, having at hand the torsor $\Aut_{\overline{\Sigma}}$ we may construct, using the morphism of group schemes $r_O : \Autpiu{} O \to \Autpiu{3} O \simeq (B_2)_{\text{ad}} \xrightarrow{r} B$, the $B$ torsor on $\overline\Sigma$ given by:
\begin{equation}\label{eq:deffzerosigmabar}
	\goF_0 \stackrel{\text{def}}{=} \Aut_{\overline{\Sigma}} \times_{r_O} B.
\end{equation}

We will give the definition of opers only relative to $\goF_0$. We do this in order to give easier definitions; this is justified by the fact that in the classical case of a smooth curve it is known that all opers have the same underlying $B$-torsor, which is the analogue of $\goF_0$ (see for instance \cite{casarin2025bundle}).

\begin{definition}\label{def:opersonsigmabar}
	A $\gog$-oper on $\overline{\Sigma}$ is a \index{$\gog$-oper on $\overline{\Sigma}$, $\Op_{\gog}(\overline{\Sigma})$}connection $\nabla$ on $(\goF_0)_G$ such that
	\begin{enumerate}
		\item $c(\nabla) \in (\text{gr}^{-1} \gog)_{\goF_0} \otimes \Omega^1_{\overline{\Sigma}} \subset (\gog/\gob)_{\goF_0}\otimes\Omega^1_{\overline{\Sigma}}$;
		\item For each negative simple root $\alpha$ the section of $c(\nabla)^\alpha \in \Gamma(\overline{\Sigma},(\gog/\gob)^{\alpha}_{\goF_0} \otimes \Omega^1_{\overline{\Sigma}})$ never vanishes. Where $(\gog/\gob)^{\alpha}$ is the root space relative to $\alpha$.
	\end{enumerate}
	There is an obvious notion of isomorphism of $\gog$-opers, we define $\Op_\gog(\overline{\Sigma})$ to be the groupoid of $\gog$-opers and isomorphisms between them. As in \cite[1.3]{beilinson2005opers} one can prove that a $\gog$-oper does not have any non-trivial automorphisms. It follows that $\Op_{\gog}(\overline{\Sigma})$ is equivalent to a set and that the assignment $U \mapsto \Op_{\gog}(\overline{\Sigma}_U)$ is a sheaf of sets on the Zariski topology on $S$.
\end{definition}

\begin{definition}
	All definitions above can be given in families, so that we have a functor $\Op_{\gog}(\overline{\Sigma}) \in \Sp_S$ which is naturally a sheaf for the Zariski topology on $S$. This, for any $R$ point $\Spec R \to S$ is defined by 
	\[\Op_\gog(\overline{\Sigma})(R) := \Op_{\gog}(\overline{\Sigma}_R).\]To avoid possible confusion let us emphasize that when we write $\Op_\gog(\overline{\Sigma})$ we mean the set of opers on $\overline{\Sigma}$ as defined in \ref{def:opersonsigmabar}, this set coincides with the value on $S$ of the functor just introduced.
\end{definition}

\subsubsection{Opers on \texorpdfstring{$\overline{\Sigma}^*$}{the pointed formal neighbourhood}}

We extend the definition of Opers on $\overline{\Sigma}^*$ essentially by linearity from the $\overline{\Sigma}$ case. We could not give this definition directly since we cannot find a suitable space in $\Sp_S$ attached to $\overline{\Sigma}^*$. Notice that given a $G$-torsor on $\overline{\Sigma}$ the exact sequence $0 \to \gog_{\goF} \to \calE_{\goF} \to \Tan_{\overline{\Sigma}} \to 0$ induces an exact sequence of locally free $\pbarOpoli$-modules
\[
	0 \to \gog_{\goF,\overline{\Sigma}^*} \to \calE_{\goF,\overline{\Sigma}^*} \to \Tan_{\overline{\Sigma}^*} \to 0.
\]
obtained by tensoring up with $\pbarOpoli$ along $\pbarO$.

Considering the above exact sequence we give the definition of opers on $\overline{\Sigma}^*$ exactly in the same way of $\overline{\Sigma}$, replacing every occurrence of $\gog_{\goF,\overline{\Sigma}},\calE_{\goF,\overline{\Sigma}}, T_{\overline{\Sigma}}$ with $\gog_{\goF,\overline{\Sigma}^*},\calE_{\goF,\overline{\Sigma}^*}, T_{\overline{\Sigma}^*}$ respectively.

\subsection{Canonical representative for Opers}\label{ssec:canonicalrepresentativesopers}

We recall here the description of $\Op_\gog(\overline{\Sigma})$ in the case where there exists a coordinate $t \in \pbarO$, this discussion holds for $\overline{\Sigma}$ replaced by $\overline{\Sigma}^*$.
Recall that we have a fixed principal $\mathfrak{sl}_2$ triple $\{f_0,h_0,e_0\} \in \gog$ and that we denote by $\check{\rho} : \mG_m \to G$ by exponentiation of $\frac{1}{2}h_0$.
Let $\Vcan = \gog^{e_o}$ then $\frac 12 h_0$ acts on $\Vcan$ with eigenvalues $d_i$, where $d_i+1$ are the exponents of $\gog$. Let $x_i$ be a basis of eigenvectors for $\Vcan$ such that $x_i$ has eigenvalue $d_i$.

The following Proposition goes back to \cite{drinfeld1984lie}.

\begin{proposition}\label{prop:art2descrlocopervcan}
	Let $\nabla \in \Op_\gog(\overline{\Sigma})$ be an oper. Since $\overline{\Sigma}$ admits a coordinate $\goF_0$ is trivial. There is a unique element $\omega_\nabla \in V\otimes\Omega^1_{\overline{\Sigma}} = \Vcan \otimes \pbarO dt$ such that
	\[
		(\goF_0,\nabla) \simeq (\overline{\Sigma}\times G, d + f_0 dt + \omega_\nabla)
	\]
	as opers. Since opers do not have automorphisms, the above isomorphism is unique. This is true in families so that given any map $\Spec R \to S$
	\[
		\Op_\gog(\overline{\Sigma})(R) = \Op_{\gog}(\overline{\Sigma}_R) \simeq \Vcan \otimes \Omega^1_{\overline{\Sigma}_R/R}(\Spec R) =  V^{\can}\otimes \calO_{\overline{\Sigma}_R}(\Spec R)dt.
	\]
	An analogous description holds for $\oSigma$ replaced with $\oSigma^*$. 
\end{proposition}

Given $\omega_\nabla \in V\otimes\Omega^1_{\overline{\Sigma}}$ and a coordinate $t \in \pbarO$ we will write $\omega_{\nabla}^t \in \Vcan \otimes \calO_{\overline{\Sigma}_R}$ for the element such that $\omega_{\nabla}^tdt = \omega_{\nabla}$ and $\omega_{\nabla}^{t,i} \in \pbarO \otimesl R$ to denote the $i$-th component of $\omega_{\nabla}$ with respect to the basis $x_i$.

\begin{remark}\label{rmk:art2changecoordinateoper}
	Under the assumption that $\overline{\Sigma}$ admits coordinates fix an oper $\nabla \in \Op_\gog(\overline{\Sigma})(R)$ and given two coordinates $t,s \in \overline{\Sigma}_R$ consider the two unique elements $\omega_\nabla^t,\omega_\nabla^s$ such that
	\[
		d + (f_0 + \omega_\nabla^t)dt \simeq \nabla \simeq d + (f_0 + \omega_\nabla^s)ds.
	\]
	Then a lengthy but simple computation shows that
	\begin{align*}
		\omega^{s,1}_\nabla &= (\partial_st)^{2}\omega^{t,1}_\nabla -\frac{1}{2}\{t,s\}, \\
		\omega^{s,j}_\nabla &= (\partial_st)^{d_j+1}\omega^{t,j}_\nabla \quad j > 1,
	\end{align*}
	where $\{t,s\} = \left( \frac{\partial_s^3t}{\partial_st}- \frac{3}{2}\left(\frac{\partial^2_st}{\partial_st}\right)^2\right)$ is the \textit{Schwarzian Derivative}.
\end{remark}

\begin{remark}[Factorization properties of the space of Opers]\label{rmk:factorizationopers}
    Recall the notation of Definition \ref{def:openclosedfactorization}: consider a family of sections $\Sigma:S \to X^J$, indexed by $J$ and a surjection $J \twoheadrightarrow I$ and the subschemes $i_{J/I} : V_{J/I} \to S, j_{J/I} : U_{J/I} \to S$. There are natural isomorphisms
    \[
        j^*_{J/I}\Op_{\gog}(\overline{\Sigma}_J^*) = \prod_{i \in I} \Op_{\gog}(\overline{\Sigma}_{J_i}^*), \qquad i_{J/I}^*\Op_{\gog}(\overline{\Sigma}_J^*) = \Op_{\gog}(\overline{\Sigma}_I^*),
    \]
    where $j^*_{J/I},i^*_{J/I}$ denote the pullback of spaces along the corresponding maps. The first one follows by the fact that on $U_{J/I}$ there is a natural identification $\overline{\Sigma}_J = \coprod_{i \in I} \overline{\Sigma}_{J_i}$. The second one comes from the fact that when restricted to $V_{J/I}$ our $J$ family of sections $\Sigma_J$ is the same as a collection of sections $\Sigma_I$.
	
	It follows that the collection $\Fun\left(\Op_{\gog}(\overline{\Sigma}_J^*)\right)$ is a complete topological factorization algebra, so that our data are equipped with canonical isomorphisms
    \begin{align*}
        \fact^{\Op}_{J/I} &: \hat{j}^*_{J/I} \bigotimes^!_{i \in I} \Fun\left( \Op_{\gog}(\overline{\Sigma}_{J_i}^*)\right) \to \hat{j}^*_{J/I}\Fun\left( \Op_{\gog}(\overline{\Sigma}_{J}^*)\right), \\
        \Ran^{\Op}_{J/I} &: \hat{i}^*_{J/I} \Fun\left( \Op_{\gog}(\overline{\Sigma}_{J}^*)\right) \to \Fun\left( \Op_{\gog}(\overline{\Sigma}_{I}^*)\right).
    \end{align*}
    Analogous properties hold for $\Fun\left( \Op_{\gog}(\overline{\Sigma}_{J})\right)$.
\end{remark}

\subsubsection{Description of \texorpdfstring{$\Fun(\Op_{\gog}(\oSigma^*))$}{the algebra of functions on Opers}}

In this section we focus on the case where $S$ is affine, well covered, with a fixed coordinate $t$ and use the description of $\Op_{\gog}(\oSigma^*)$ via canonical representatives to give a description of $\Fun\left(\Op_{\gog}(\oSigma^*) \right)$ as the completion of a polynomial algebra.

\smallskip

Let $V$ be a finite dimensional vector space, consider the functors
\begin{align*}
    J_{\oSigma} V &: \Aff^{\op}_S \to \Set \qquad R \mapsto V \otimes_{\mC} \calO_{\oSigma_R}(\Spec R), \\
    L_{\oSigma} V &: \Aff^{\op}_S \to \Set \qquad R \mapsto V \otimes_{\mC} \calO_{\oSigma^*_R}(\Spec R).
\end{align*}
Since $\pbarO (S)$ is topologically free, $J_{\oSigma} V$ is represented by an affine scheme, while $L_{\oSigma}V$ is an ind-affine scheme, there is a closed embedding $J_{\oSigma} V \hookrightarrow L_{\oSigma} V$ which induces a surjection $\Fun\left( L_{\oSigma} V \right) \twoheadrightarrow \Fun \left( J_{\oSigma} V \right)$.

\begin{remark}[Factorization properties of $J_{\oSigma},L_{\oSigma}$]\label{rmk:factpropertiesjsigmalsigma}
    It follows from the factorization properties of $\pbarO$ and $\pbarOpoli$ that for any finite dimensional vector space $V$, the spaces $J_{\oSigma} V,L_{\oSigma} V$ have natural factorization structures. The same goes for their sheaves of functions so that (with respect to the notation of Definition \ref{def:openclosedfactorization}), there are natural isomorphisms
    \begin{align*}
        \fact^{L_{\oSigma} V}_{J/I} &:\hat{j}^*_{J/I} \bigotimes_{i \in I}^! \Fun\left( L_{\oSigma_{J_i}} V \right) \to \hat{j}^*_{J/I}\Fun\left( L_{\oSigma_J} V \right), \\
        \Ran^{L_{\oSigma} V}_{J/I} &:\hat{i}^*_{J/I} \Fun\left( L_{\oSigma_J} V \right) \to \Fun\left( L_{\oSigma_I} V \right),
    \end{align*}
    which make the collection $\Fun\left( L_{\oSigma} V \right)$ into a complete topological factorization algebra.
\end{remark}

Given $g \in \pbarOpoli(S)dt$ and $h \in V^*$ we consider the function $h \otimes gdt : L_{\oSigma} V \to \mA^1$ defined by $v\otimes f \mapsto h(v)\Res_\Sigma(fgdt)$. This construction induces an injection
	\[
		V^*\otimes\pbarOpoli(S) \hookrightarrow \Fun(L_{\oSigma} V)
	\]
which, since $\Res_\Sigma : \pbarOpoli \times \pbarOpoli dt \to \calO_S$ is a perfect pairing (see \cite[Lemma \ref{lem:resnondegenere}]{casmaffei1}), induces an isomorphism of commutative complete topological sheaves of algebras 
\[ \Fun(L_{\oSigma} V) = \overline{\Sym}_{\lowOS}(V^*\otimes\pbarOpoli),
\]
where $\overline{\Sym}_{\lowOS}(V^*\otimes\pbarOpoli)$ is the completion of $\Sym_{\lowOS}(V^*\otimes\pbarOpoli)$ along the topology generated by the ideals $\left( V^*\otimes\pbarO(-n) \right)$. We apply this to the case of opers. Recall that by Proposition \ref{prop:art2descrlocopervcan} the choice of a coordinate $t \in \pbarO(S)$ induces an isomorphism $\chi_{\Sigma,t} : \Op_{\gog}(\overline{\Sigma}^*) \to L_{\oSigma}V^{\can}$. We will denote by 
\[
	\chi_{\Sigma,t}^* : \overline{\Sym}_{\lowOS}((V^{\can})^*\otimes\pbarOpoli) \to \Fun\left( \Op_{\gog}(\overline{\Sigma}^*) \right)
\]
the isomorphism on functions induced by $\chi_t$. The ideal generated by $(V^{\can})^*\otimes\pbarO$ (without poles) is sent exactly to the ideal defining $\Op_{\gog}(\overline{\Sigma}) \subset \Op_{\gog}(\oSigma^*)$, so that we get an analogous isomorphism
\[
	\chi_{\Sigma,t}^* : \Sym_{\lowOS}\left((V^{\can})^*\otimes\frac{\pbarOpoli}{\pbarO}\right) \to \Fun\left( \Op_{\gog}(\overline{\Sigma}) \right).
\]

\begin{remark}[Factorization properties for $\chi^*_{\Sigma,t}$]\label{rmk:factpropertieschi} Recall the notation of Definition \ref{def:openclosedfactorization} and fix a surjection of finite sets $J \twoheadrightarrow I$. Then the collection of isomorphisms $\chi^*_t$ make the following diagrams commute
    \[\begin{tikzcd}
    {\hat{j}_{J/I}^*\Fun\left(\Op_{\gog}(\overline{\Sigma}^*_J)\right)} &&& {\hat{j}_{J/I}^*\overline{\Sym}_{\lowOS}((V^{\can})^*\otimes\calO_{\overline{\Sigma}_J})} \\
	\\
	{\hat{j}_{J/I}^*\bigotimes^!_{i \in I} \Fun\left(\Op_{\gog}(\overline{\Sigma}_{J_i}^*)\right)} &&& {\hat{j}_{J/I}^*\bigotimes^!_{i \in I} \overline{\Sym}_{\lowOS}((V^{\can})^*\otimes\calO_{\overline{\Sigma}_{J_i}})}
	\arrow["{\hat{j}_{J/I}^*\chi^*_{\Sigma_J,t}}"', from=1-4, to=1-1]
	\arrow["{\fact^{\Op}_{J/I}}",from=3-1, to=1-1]
	\arrow["{\fact^{L_{\oSigma} V}_{J/I}}"', from=3-4, to=1-4]
	\arrow["{\hat{j}_{J/I}^*\prod_{i\in I} \chi^*_{\Sigma_{J_i},t_i}}"', from=3-4, to=3-1]
\end{tikzcd}\]

	\[\begin{tikzcd}
	{\hat{i}_{J/I}^*\Fun\left(\Op_{\gog}(\overline{\Sigma}^*_J)\right)} && {\hat{i}_{J/I}^*\overline{\Sym}_{\lowOS}((V^{\can})^*\otimes\calO_{\overline{\Sigma}_J})} \\
	\\
	{\Fun\left(\Op_{\gog}(\overline{\Sigma}_I^*)\right)} && {\overline{\Sym}_{\lowOS}((V^{\can})^*\otimes\calO_{\overline{\Sigma}_I})}
	\arrow["{\hat{i}_{J/I}^*\chi^*_{\Sigma_J,t}}"', from=1-3, to=1-1]
	\arrow["{\Ran^{\Op}_{J/I}}"', from=1-1, to=3-1]
	\arrow["{\Ran_{J/I}^{L_{\oSigma} V}}", from=1-3, to=3-3]
	\arrow["{\chi^*_{\Sigma_I,t_I}}"', from=3-3, to=3-1]
\end{tikzcd}\]
\end{remark}

\subsection{The factorization space \texorpdfstring{$\Op_{\gog}(D)_C$}{of opers on the disk}}

In this section, given $C$ a smooth curve over $\mC$, following the classical analogues of these constructions, we define the factorization spaces $\Op_{\gog}(D)_C$ and $\Op_{\gog}(D^*)_C$. Let $I$ be a finite set, we identify morphisms $\Spec R \to C^I$ with collections $\Sigma = \{ \sigma_i : \Spec R \to C_R \}$ of sections of $C_R \to \Spec R$. In particular we will use the notion of a $\gog$-oper on $\overline{\Sigma},\overline{\Sigma}^*$ given in Section \ref{sec:opersigma}. The point is that we need a notion of Opers on the formal neighbourhood of the scheme theoretic union of the sections of $\Sigma$ and on its pointed version, so that our formalization really comes in handy. It is not strictly necessary though: a possible way to avoid the usage of this language is to define $\Op_{\gog}(D)_C$ and $\Op_{\gog}(C)$ when $C$ is affine and equipped with a coordinate and then glue these constructions to get the correct spaces over $C$. 

\begin{definition}\label{def:opersCI}
	We define $\Op_{\gog}(D)_{C^I}$, $\Op_{\gog}(D^*)_{C^I}$ as a functor of $\mC$ algebras.
	\begin{align*}
		\Op_{\gog}(D)_{C^I}(R) &= \left\{ \left(\Sigma \in C^I(R), \nabla \in \Op_{\gog}\left(\overline{\Sigma}\right) \right) \right\}, \\
		\Op_{\gog}(D^*)_{C^I}(R) &= \left\{ \left(\Sigma \in C^I(R), \nabla \in \Op_{\gog}\left(\overline{\Sigma}^*\right)\right)\right\}.
	\end{align*}
	\index{$\Op_{\gog}(D)_{C^I}$}\index{$\Op_{\gog}(D^*)_{C^I}$}
\end{definition}

\begin{comment}
	QUESTO ANDRA RIMESSO
\begin{proof}
	We assume that $C$ is affine and equipped with a coordinate $t$. In this case we prove that the choice of $t$ induces an isomorphism $\triv_t^{\Op} : \Op_{\gog}(D)_C \to C \times \Op_{\gog}(D)$. This isomorphism is easy to construct: notice that in this case, given a section $\Sigma = \{x : \Spec R \to C_R\}$, the space $\overline{\Sigma}$ corresponds to the topological algebra $\calT_xR$ of Section \ref{sec:tr}, so that $\Op_\gog(D)_C(R) = \{ (x,\nabla) \}$ where $x \in C(R)$ and $\nabla$ is an Oper on $\calT_xR$. It follows that using the isomorphism $\rho^x_t : R[[z]] \to \calT_xR$ (c.f. Notation \ref{ntz:rhou}) the latter correspond to Opers on $R[[z]]$. To prove the claim of the Lemma it is enough to notice that, after the choice of another coordinate $s$, the isomorphism $(\triv_s^{\Op})^{-1}\triv_t^{\Op} : C \times \Op_\gog(D) \to C \times \Op_\gog(D)$ is given by
	\[
		(x,\nabla) \mapsto (x,\triv_{st}\cdot \nabla),
	\]
	where $\triv_{st} \in \Autpiu{} O$ is the element of Lemma \ref{lem:changecoordautx} acting on $\Op_{\gog}(D)(R)$.
\end{proof}
\end{comment}

\begin{remark}\label{rmk:operCopersigma}
	We notice that these constructions are special cases of our functors $$\Op_{\gog}\left(\overline{\Sigma^{\univ}_{C,I}}\right),\Op_{\gog}\left(\overline{\Sigma^{\univ}_{C,I}}^*\right)$$ for the particular choice of $\Sigma^{\univ}_{C,I}$ introduced in Remark \ref{rmk:factsigmaugualefactcurva}.
	Let us recall that here: let $S = C^I$ and $X = C \times C^I$ then $X \to S$ is a smooth family of curves equipped with a canonical $I$-family of sections $\Sigma^{\univ}_{C,I}$. Namely, for any index $i \in I$, the $i$-th section of $\Sigma^{\univ}_{C,I}$ is given by 
	\[
		(\sigma^{\univ}_{C,I,i}) : C^I \to C \times C^I \qquad (x_{i'})_{i'\in I} \mapsto (x_i,(x_{i'})_{i'\in I}).
	\]
	As functors on  $\Aff_{C^I}$ we have 
$$
		\Op_{\gog}(D)_{C^I} = \Op_{\gog}\left(\overline{\Sigma^{\univ}_{C,I}}\right), \qquad 
		\Op_{\gog}(D^*)_{C^I} = \Op_{\gog}\left(\overline{\Sigma^{\univ}_{C^I}}^*\right).
$$	In particular, by the factorization properties of \ref{rmk:factorizationopers} these naturally define a factorization spaces on $C$ with respect to Definition \ref{def:factorizationspace}. Given any surjection of finite sets $J \twoheadrightarrow I$ we have pullback diagrams (we write down only the $D^*$ case, $D$ is completely analogous)
	\[\begin{tikzcd}
	   {\Op_{\gog}(D^*)_{C^I}} & {\Op_{\gog}(D^*)_{C^J}} & {\prod_{i \in I} \Op_{\gog}(D^*)_{C^{J_i}}} & {\Op_{\gog}(D^*)_{C^J}} \\
	   \\
	   {C^I} & {C^J} & {U_{J/I}} & {C^J}
	   \arrow[from=1-1, to=1-2]
	   \arrow[from=1-1, to=3-1]
	   \arrow["\lrcorner"{anchor=center, pos=0.125}, draw=none, from=1-1, to=3-2]
	   \arrow[from=1-2, to=3-2]
	   \arrow[from=1-3, to=3-3]
	   \arrow[from=1-4, to=3-4]
	   \arrow["\lrcorner"{anchor=center, pos=0.125}, draw=none, from=1-3, to=3-4]
	   \arrow[from=1-3, to=1-4]
	   \arrow["{\Delta_{J/I}}"', hook, from=3-1, to=3-2]
	   \arrow[hook, from=3-3, to=3-4]
   \end{tikzcd}\]
\end{remark}

\section{The center of \texorpdfstring{$\VkSg$}{the chiral algebra}}

In this section we give the definition of $\zeta(\calV)$, the center of a chiral algebra $\calV$ over $\oSigma$ and study it in the local case where $S$ is affine and well covered. We apply this to $\calV = \VkSg$ and show that its center has a natural factorization structure.

\begin{definition}
    Let $\calV$ be a chiral algebra over $\oSigma$ and let $\mu : \calV \exttensor{\fil} \calV(\infty\Delta) \to \Delta_!\calV$ be its chiral bracket. Given an open subset $U\subset S$, a section $z \in \calV(U)$ is \emph{central} if the induced morphism
    \[
        \mu(\_ \boxtimes z) : \calV|_U \to (\Delta_! \calV)|_U
    \]
    vanishes. Notice that by skew-symmetry this is equivalent to require that $\mu(z\boxtimes \_)$ vanishes. The \emph{center} of $\calV$ is defined as the presheaf $\zeta(\calV)(U) = \{ z \in \calV(U): z \text{ is central}\}$. A chiral algebra is said to be commutative if $\zeta(\calV)  =\calV$ or, equivalently, if $\mu : \calV \boxtimes^{\fil} \calV \to \Delta_!\calV$ vanishes. The center $\zeta(\calV)$ is naturally a $\pbarO$-module and can be shown to be a sheaf, as the following remark shows.
\end{definition}

\begin{remark}
    Let us first check that $\mu(\_\boxtimes z)$ is indeed well defined. Let $\calV$ be a chiral algebra over $\oSigma$ and $\calV(n)$ be its attached filtration. Recall that we are assuming that $\calV = \varinjlim_n \calV(n)$ as sheaves and that since this is a directed colimit we have that for any open subset $\calV(U) = \varinjlim_n \calV(n)(U)$. In particular any local section $z \in \calV(U)$ belongs to $\calV(n)(U)$ for some $n$. It follows that for every $m\geq n$ there is a well defined map $\mu(\_ \boxtimes z) : \calV(m)|_U \to (\Delta_! \calV)|_U$ and taking its colimit gives a well defined map $\mu(\_ \boxtimes z) : \calV|_U \to (\Delta_!\calV)|_U$. To show that $\zeta(\calV)$ is indeed a sheaf notice that for any $m \geq n$ the assignment $U \mapsto \{ z \in \calV(n)(U) : \mu(\_ \boxtimes z) : \calV(m)|_U \to (\Delta_!\calV)|_U \text{ is zero} \}$ determines a subsheaf $\zeta_{n,m}(\calV)\subset \calV(n)$. Then $\zeta(\calV) = \varinjlim_n (\cap_{m\geq n} \zeta_{n,m}(\calV))$ is also a sheaf. 
\end{remark}

\subsection{Local and global description}

Recall that by \cite[Section \ref{ssez:Otalgebre}]{casmaffei1} (see also Section \ref{ssec:richiamivertexalgoverost}), in the case where $S$ is affine and well covered, there is an equivalence between filtered \QCC chiral algebras and filtered complete vertex algebras over $(\pbarO(S),t)$, which is established by $\calV \mapsto \calV(S)$. Recall also that after a coordinate $t \in \pbarO(S)$ is chosen, the chiral product on $\calV$ relates to the vertex algebra structure on $\mV = \calV(S)$ via the following formula:
\[
    \mu(v\boxtimes w (t\otimes 1 - 1 \otimes t)^n) = \sum_{k \geq 0} \frac{1}{k!} v_{(n+k)}w\cdot \partial_{t\otimes 1}^k.
\]
It immediately follows that $\zeta(\calV)(S)$ identifies with $\zeta(\mV)$, the set of the central elements of the vertex algebra $\mV = \calV(S)$ (recall that the center of a vertex algebra $V$ is given by the set of those $v \in V$ such that $w_{(n)}v = 0$ for any $w \in V$, $n \geq 0$). This also shows that in the case where $\calV$ is a filtered \QCC chiral algebra then $\zeta(\calV)$ is also a filtered \QCC(commutative) chiral algebra.
\smallskip

As explained in \cite[Section \ref{rmk:extvoaoslin}]{casmaffei1} (see also Section \ref{ssec:richiamivertexalgoverost}) to any vertex algebra $V$ we can attach a vertex algebra over $(\pbarO(S),t)$ as follows: $\mV = V \otimesl \pbarO(S)$, $(v\otimes f)\cdot \partial_t = -Tv \otimes f - v \otimes(\partial_tf)$ and 
\[
    (v\otimes f)_{(n)}(w \otimes g) = \sum_{k \geq 0} \frac{1}{k!} v_{(n+k)}w \cdot (g\partial^k_tf).
\]
We have the following description of $\zeta(\mV)$.

\begin{proposition}\label{prop:descrcenterchiral}
    Assume that $S = \Spec A$ is affine and well covered with a coordinate $t \in \pbarO(S)$. Let $V$ be a vertex algebra over $\mC$ and let $\mV = V \otimesl \pbarO(S)$ be its attached vertex algebra over $(\pbarO(S),t)$. Then
    \[
        \zeta(\mV) = \zeta(V) \otimesl \pbarO(S),
    \]
    where $\zeta(V)$ is the center of the vertex algebra $V$.
\end{proposition}

\begin{proof}
    It is easy to check that $\zeta(V) \otimes \pbarO(S) \subset \zeta(\mV)$; indeed for any $v\otimes f \in \mV$, $z \otimes g \in \zeta(V)\otimes \pbarO(S)$, and any $n\geq 0$, we have $(v \otimes f)_{(n)}(z \otimes g) = \sum_{k \geq 0} \frac{1}{k!} v_{(n+k)}z\otimes g\partial_t^kf = 0$.  We just need to show that every central element $z \in \zeta(\mV)$ lies in $\zeta(V) \otimes \pbarO(S)$. Let $f_i \in \pbarO(S)$ be a $\mC$ basis, consider a central element $z \in \zeta(\mV)$ and write it as a finite sum
    \[
        z = \sum_{i} v_i \otimes f_i. 
    \]
    We want to show that $v_i \in \zeta(V)$ for all $i$. The condition of $z$ being central implies that $v_{(n)}z = 0$ for all $n \geq 0$ and $v \otimes 1 \in V\otimes 1 \subset \mV$. Applying the above formulas we get
    \[
        (v\otimes 1)_{(n)}z = \sum_i v_{(n)}v_i \otimes f_i = 0.
    \] 
    Since $f_i$ is a basis of $\pbarO(S)$ it follows that $v_{(n)}v_i = 0$ for any $i$, $v \in V$ and $n\geq 0$ as desired.
\end{proof}

We apply the above to give a local description of $\zeta(\VkSg)$. Recall that by \cite[Section \ref{ssec:localdescriptionofvksg}]{casmaffei1} (see also Section \ref{ssec:richiamiVkSg}), in the case $S = \Spec A$ is integral, affine and well covered with a given coordinate $t \in \pbarO(S)$ we constructed an isomorphism
\[
    \calY_{\Sigma,t} : V^\kappa(\gog) \otimes \pbarO \to \VkSg.
\]
This implies that $\VkSg(S)$ is isomorphic to the vertex algebra over $(\pbarO(S),t)$ attached to the vertex algebra $V^\kappa(\gog)$, in particular $\VkSg$ is a filtered \QCC chiral algebra and we may apply Proposition \ref{prop:descrcenterchiral} to obtain the following Corollary. Before stating it, let us introduce here a bit of notation: we write $\zeta^\kappa(\gog)$ for the center of the vertex algebra $V^\kappa(\gog)$ and $\zeta^\kappa_{\Sigma}(\gog)$ for the center of the chiral algebra $\VkSg$. 

\begin{corollary}\label{coro:subchiralalg}
	Assume that $S$ is integral, affine and well covered. Then the map $\calY_{\Sigma,t}$ establishes a topological isomorphism
	\[
	\calY_{\Sigma,t}  : \zeta^\kappa(\gog)\otimes\pbarO(S) \to \zeta^\kappa_{\Sigma}(\gog)(S). 
	\]
\end{corollary}

\begin{remark}[Global Description]
    From this description it also follows a global description of $\zeta^\kappa_\Sigma(\gog)$ in the case where $\gog$ is simple and $\kappa = \kappa_c$ is the critical level just as in \cite[Section \ref{sec:identificationopers1}]{casmaffei1}. The proof goes on exactly as in \emph{loc. cit.}; we won't repeat it here since we will not need this result in the sequel and we just mention it for completeness. Before stating it let us recall some of the constructions of \emph{loc. cit.}. Recall that by \cite[Proposition \ref{prop:twisttanchiralcommutative}]{casmaffei1} the assignment $\calV \mapsto \calV \otimes_{\pbarO} T_{\oSigma}$ establishes an equivalence between filtered \QCC commutative chiral algebras and filtered \QCC $\pbarD$-commutative algebras. Recall also the construction of the sheaf $\Gamma_{\oSigma}(V)$ for a given $\Aut^0 O$-module $V$ of \cite[Section \ref{ssec:gammasigmav}]{casmaffei1} and that in the case $V$ is also an $\Aut O$-commutative algebra, then the sheaf $\Gamma_{\oSigma}(V)$ is naturally a \QCC filtered $\pbarD$-commutative algebra. We will consider this construction for the $\Aut O$ commutative algebra $\mC[\Op_{\gogl}(D)]$. 
    \smallskip
    
    Then, if we assume that $\gog$ is a simple Lie algebra, it follows by Theorem \ref{thm:globaldescrchiralalg} of \cite{casmaffei1} that there is a canonical isomorphism of commutative $\pbarD$-algebras
    \[
        \zeta^{\kappa_c}_{\Sigma}(\gog) \otimes_{\pbarO} T_{\oSigma} = \Gamma_{\oSigma}(\mC[\Op_{\gogl}(D)]).
    \]
\end{remark}

\subsection{Factorization}

We apply the local description of Corollary \ref{coro:subchiralalg} to establish some factorization properties of the center $\zeta^\kappa_{\Sigma}(\gog) = \zeta(\VkSg)$.

\begin{corollary}
    Assum that $S$ is integral, quasi-separated and quasi-compact. Then the factorization structure on $\VkSg$ of \cite[Proposition \ref{prop:factpropertieschiralalg}]{casmaffei1} induces isomorphisms of chiral algebras
    \begin{align*}
			\Ran^{\zeta}_{J/I} &: \hat{i}_{J/I}^*\zeta^{\kappa}_{\Sigma_J}(\gog) \to \zeta^{\kappa}_{\Sigma_I}(\gog), \\ \fact^{\zeta}_{J/I} &: \hat{j}_{J/I}^*\left( \prod_{i\in I} \zeta^{\kappa}_{\Sigma_{J_i}}(\gog) \right) \to \hat{j}_{J/I}^*\zeta^{\kappa}_{\Sigma_J}(\gog)
	\end{align*}
	which make $\zeta^\kappa_{\Sigma}(\gog)$ a complete topological factorization algebra with respect to $\prod$.
\end{corollary}

\begin{proof}
    The assertion may be checked locally, so that we may assume that $S$ is affine and well covered. By Theorem \ref{thm:factorizationofcaly} of \cite{casmaffei1}, when $S$ is affine and well covered the isomorphism $\calY_{\Sigma,t}$ intertwines with the factorization structure of $\pbarO$ and that of $\VkSg$ (see Section \ref{ssec:richiamifactchiral}), so that the isomorphisms $\Ran^{V^\kappa}_{J/I}, \fact^{V^\kappa}_{J/I}$, identify, via $\calY_{\Sigma,t}$, with
    \begin{align*}
		\id \otimes\Ran^{\pbarO}_{J/I} &: V^\kappa(\gog)\otimes\hat{i}_{J/I}^*\calO_{\oSigma_J}, \to V^\kappa(\gog)\otimes\calO_{\oSigma_I} \\ \id\otimes\fact^{\pbarO}_{J/I} &: V^\kappa(\gog)\otimes\hat{j}_{J/I}^*\left( \prod_{i\in I} \calO_{\oSigma_{J_i}} \right) \to  V^\kappa(\gog)\otimes\hat{j}_{J/I}^*\calO_{\oSigma_J}
	\end{align*}
    It is then evident that, by the local description of Corollary \ref{coro:subchiralalg} that these restrict to morphisms $\Ran^\zeta_{J/I},\fact^\zeta_{J/I}$ as in the statement of the corollary and that these restrictions are isomorphisms.
\end{proof}

By the discussion of Section \ref{ssec:factorizationU} the factorization structure on $\zeta^\kappa_\Sigma(\gog)$ induces a factorization structure on its completed enveloping algebra.

\begin{corollary}
    The isomorphisms $\Ran^{\calU(V)}_{J/I}, \fact^{\calU}_{J/I}$ of Proposition \ref{prop:factpropertiescalu} induce isomorphisms 
    \begin{align*}
    \Ran^{\calU(\zeta)}_{J/I} &: \hat{i}_{J/I}^*\left(\calU_{\overline{\Sigma}_{J}^*}(\zeta^{\kappa}_{\Sigma_{J}}(\gog))\right) \to \calU_{\overline{\Sigma}_{I}^*}(\zeta^{\kappa}_{\Sigma_{I}}(\gog)) \\
		\fact^{\calU (\zeta)}_{J/I} &: \hat{j}_{J/I}^*\left( \bigotimes^!_{i\in I} \calU_{\overline{\Sigma}_{J_i}^*}(\zeta^{\kappa}_{\Sigma_{J_i}}(\gog)) \right) \to \hat{j}_{J/I}^*\left(\calU_{\overline{\Sigma}_{J}^*}(\zeta^{\kappa}_{\Sigma_{J}}(\gog))\right)
		\end{align*}
		which make $\calU_{\overline{\Sigma}^*}\left( \zeta^{\kappa}_{\Sigma}(\gog)\right)$ into a complete topological factorization algebra.
\end{corollary}

\section{Identification with the algebra of function on opers}\label{sec:identificationopers2}

In what follows we consider $\gog$, a simple finite Lie algebra and $\gogl$, its Langlands dual Lie algebra. The goal of this section is to prove Theorem \ref{thm:teofinale1}, which establishes a canonical isomorphism between the center $\Zcrit = Z(\ugsig{\kappa_c})$ (we will only work at the critical level in this section) \index{$\Zcrit$} of the completed enveloping algebra $\ugsig{\kappa_c}$ at the critical level and $\Fun\left(\Op_{\gogl}(\overline{\Sigma}^*)\right)$, the complete topological algebra of functions on the space of $\gogl$-opers on $\overline{\Sigma}^*$. This isomorphism will be compatible with the factorization structures on both spaces and in the case when $\Sigma$ consists of a single section it will recover the usual Feigin-Frenkel isomorphism.

\subsection{The center and its factorization properties}
\label{ssec:factorizationcenter}
Let us start by introducing the center of a complete sheaf of algebras and prove some of its basic properties. We will deduce some factorization properties of $Z_{\kappa}(\hat{\gog}_\Sigma) = Z(\ugsig{\kappa})$ from those of $\ugsig{\kappa_c}$ as well. \index{$\ugsig{\kappa}$:$Z_{\kappa}(\hat{\gog}_\Sigma)$}

\begin{definition}\label{def:center}
    Let $\calA$ be a sheaf of associative algebra on $S$ for which its product $m$ is continuous in both variables. Given a local section $a \in \calA(U)$ we consider $[a,\_] : \calA_{|U} \to \calA_{|U}$ to be the continuous morphism determined by $m(a\otimes\_) - m(\_\otimes a)$. We define $Z(A)$, the \emph{center} of $\calA$, as the following subsheaf of $\calA$:
    \[
        Z(\calA)(U) = \left\{ z \in \calA(U) \text{ such that } [z,\_] = 0 \right\}
    \]
    where by $[z,\cdot]=0$ we mean that the map $\calA|_U\lra \calA|_U$ given by the bracket with $z$ is zero. 
\end{definition}

\begin{remark}
    If $\calA$ is a complete sheaf of associative algebras it is clear that also $Z(\calA)$ with the induced topology is a complete sheaf.  

    In general, $Z(\calA)(U)$ is smaller than $Z(\calA(U))$ and $U\mapsto Z(\calA(U))$ does not even need to define a presheaf, since restriction does not need to be well defined. However if we assume that $\calA$ is QCC, then over all affine subsets $U$ we have  $Z(\calA)(U)=Z(\calA(U))$. 
    
    If we assume in addition that the topology of $\calA$ is generated by left ideals then 
    the topology on $Z(\calA)$ consists of two sided ideals, so that $Z(\calA)$ is naturally a $\otimes^!$-topological algebra. 
    
    Finally we notice that even with very strong hypotheses on $\calA$ and $S$ (for example $\calA$ QCCF and $S$ integral and noetherian) the sheaf $Z(\calA)$ does not need to be a QCC sheaf. 
\end{remark}

\begin{remark}\label{rmk:centerandpullback}
    Assume $\calA$ is a QCC associative algebra on $S$ whose product is continuous in both variables. Let $f : S' \to S$ be a morphism of schemes and denote by $\calA'$ the pull back of $\calA$. Then the following hold
\begin{enumerate}
	\item The map $\hat{f}^*Z(\calA) \to \calA'$ naturally factors through $Z(\hat\calA')$;
	\item In the case where $f$ is an open immersion, the map $\hat{f}^*Z(\calA) \to Z(\calA')$ is an isomorphism.
\end{enumerate}
	Notice that this does not contradict the previous Remark on $Z(\calA)$ not being a \QCC sheaf, since we are dealing with the sheaf theoretic center, instead of the center of sections.
\end{remark}

\begin{proof}
    Point $(1)$ follows by the fact that (before taking the completion) the map $\calO_T\otimes_{f^{-1}\calO_S} f^{-1}Z(\calA) \to \calA'$ naturally factors through the center $Z(\calA')$ and by the fact that $Z(\calA')$ is complete. Point $(2)$ follows by the fact that if $f$ is an open immersion and we are dealing with complete sheaves the completed pullback along an open immersion $U \to S$ coincides with restricting sheaves to $U$ (and this preserve completeness).
\end{proof}

We state the following easy Remark. Recall that given QCCF, $\otimesr$ associative algebras $\calA$,$\calB$, their product $\calA \otimes^!\calB$ is naturally a $\otimesr$ associative algebra by \cite[Remark \ref{prop:otimesshcontrootimesr}]{casmaffei1}.

\begin{remark}\label{rmk:morphismcenterproduct}
    Let $\calA$ and $\calB$ be QCCF, $\otimesr$ associative algebras on $S$. Then there exists a natural continuous morphism
    \[
        Z(\calA)\otimessh Z(\calB) \to Z(\calA\otimessh\calB).
    \]
\end{remark}

\begin{proof}
    There is a natural continuous morphism $Z(\calA)\otimes^! Z(\calB) \to \calA\otimes^!\calB$ induced by functoriality of $\otimes^!$. The image of $Z(\calA)\otimes Z(\calB)$ is easily shown to be contained in $Z(\calA\otimes^!\calB)$, so that the result follows by the fact that $Z(\calA\otimes^!\calB)$ is complete.
\end{proof}

\begin{proposition}\label{prop:factorizationcenter}
    Recall Definitions \ref{def:openclosedfactorization} and \ref{def:factorizationcompletesheafsigma} and fix a surjection of finite sets $J \twoheadrightarrow I$. Then the factorization structure of $\ugsig{\kappa}$ from \cite[Proposition \ref{prop:factpropertiesgogcalugog}]{casmaffei1} induces morphisms
    \begin{align*}
        \Ran^Z_{J/I} &: \hat{i}^*_{J/I} Z_\kappa(\hat{\gog}_{\Sigma_J}) \to Z_\kappa(\hat{\gog}_{\Sigma_I}), \\ \fact^Z_{J/I} &: \hat{j}^*_{J/I} \bigotimes^!_{i \in I} Z_\kappa(\hat{\gog}_{\Sigma_{J_i}}) \to \hat{j}^*_{J/I}Z_\kappa(\hat{\gog}_{\Sigma_J})
	\end{align*}
    which make $Z_{\kappa}(\hat{\gog}_{\Sigma})$ into a complete topological pseudo factorization algebra.
\end{proposition}

\begin{proof}
    The construction of both maps uses Remark \ref{rmk:centerandpullback}[1]. This already allows us to construct $\Ran$, using the factorization structure of $\ugsig{\kappa}$. To construct $\fact$ we use also \ref{rmk:centerandpullback}[2], the fact that $\hat{j}^*_{J/I}$ commutes with $ \otimes^!$, together Remark \ref{rmk:morphismcenterproduct} to get
    \[\begin{tikzcd}
        \bigotimes^!_{i \in I} \hat{j}^*_{J/I}Z_\kappa(\hat{\gog}_{\Sigma_{J_i}}) = \bigotimes^!_{i \in I} Z\left(\hat{j}^*_{J/I}\calU_\kappa(\hat{\gog}_{\Sigma_J})\right) \arrow{d}{} \\ Z\left( \bigotimes^!_{i \in I} \hat{j}^*_{J/I} \calU_\kappa(\hat{\gog}_{\Sigma_J})\right) \arrow{r}{\fact^{\calU(\gog)}_{J/I}} & Z\left( \hat{j}^*_{J/I} \calU_\kappa(\hat{\gog}_{\Sigma_J})\right) = \hat{j}^*_{J/I} Z_\kappa(\hat{\gog}_{\Sigma_J}).\end{tikzcd}
    \]
\end{proof}

We will need the following technical Lemma as well.

\begin{lemma}\label{lem:isoprodcentro}
	Let $\calB_1$, $\calB_2$ be \QCC sheaves, which are $\otimesr$ associative algebras. Assume that $Z(\calB_1)$ and $Z(\calB_2)$ are also QCC sheaves. 
	Assume that $\calB_2$ is locally \QCCF and that $Z(\calB_1)$ is locally \QCCF. Then the natural map of Remark \ref{rmk:morphismcenterproduct}
	\[
	Z(\calB_1) \otimessh Z(\calB_2) \to Z(\calB_1\otimessh \calB_2)  
	\]
	is a bicontinuous isomorphism. In particular $Z(\calB_1 \otimes^! \calB_2)$ is also a \QCC sheaf. 
\end{lemma}

\begin{proof}
    It is enough to prove that 
$$
Z(\calB_1)(U) \otimessh Z(\calB_2)(U) \simeq Z(\calB_1\otimessh \calB_2)(U) 
$$
for $U$ an open affine subset such that $Z(\calB_1)(U)$ and $\calB_2(U)$ are \QCCF modules. 
Denote by $B_1$, $B_2$. $Z_1$ and $Z_2$ the sections over $U$ of the two algebras and of the
two centers. 

Notice first that, in general, if $M$ is a \QCCF module and $P\subset Q$ is a closed submodule of a complete module $Q$ then the map $M\otimes^! P\lra M\otimes^! Q$ is injective and $M\otimes^! P$ has the induced topology. This implies that we have inclusions
$$
Z_1\otimes^! Z_2 \subset Z_1 \otimes^! B_2 \subset B_1 \otimes^! B_2.
$$
It is also clear that the image of $Z_1 \otimes^! Z_2$ is contained in $Z(B_1\otimes^! B_2)$. Hence we need to prove that $Z(B_1\otimes^! B_2)\subset Z_1 \otimes^! Z_2$. 

Let $z\in Z(B_1\otimes^! B_2)$. We prove first that $z\in  Z_1\otimes^! B_2$. Let $V_2$ be a cofree neighborhood of zero of $B_2$, let $F_2$ be a complement (which is therefore discrete) and \{$e_i\}_{i \in I}$ be a basis of $F_2$. 
We have $B_1\otimes^! B_2 = B_1\otimes^! V_2 \oplus B_1\otimes^! F_2$ and write $z=x+y$ with $x\in B_1\otimes^! V_2$ and $y\in B_1\otimes^! F_2$. Notice that 
$$
B_1\otimes^! F_2 = \limpro \frac{B_1\otimes F_2}{V_1\otimes F_2}=
\limpro \frac{B_1}{V_1} \otimes F_2 .
$$
This implies that there exists uniquely determined $y_i\in B_1$ such that 
\begin{enumerate}[\indent i)]
	\item for all neighborhoods of zero $V_1 \subset B_1$ we have $y_i\in V_1$ for all but finitely many $i$;
	\item for all neighborhoods of zero $V_1 \subset B_1$ we have $y\in \sum y_i\otimes e_i + V_1\otimes^! F_2$.
\end{enumerate}
For any $b\in B_1$, taking the commutator with $b\otimes 1$ preserves the decomposition $ B_1\otimes^! V_2 \oplus B_1\otimes^! F_2$, it follows that $b\otimes 1$ commutes with $y$ and we deduce $y_i\in Z_1$ for all $i$. Hence $z\in B_1\otimes^! V_2 \oplus Z_1\otimes^! F_2$. 

This implies that $z\in Z_1 \otimes^! B_2$. Indeed assume it is not, then, since $Z_1 \otimes^! B_2$ is closed (and recall this is equivalent to $Z_1 \otimes^!B_2 = \cap_{Z_1\otimes^! B_2 \subset V} V$ where $V$ ranges over neighborhoods of zero) there exist a neighborhood of zero $V_1$ of $B_1$ and a cofree neighborhood of zero $V_2$ of $B_2$  such that $z\notin W=B_1\otimes^! V_2 + V_1\otimes^! B_2$ and $W\supset Z_1\otimes B_2$. Let $F_2$ be a complement of $V_2$ as above, then from  $W\supset Z_1\otimes B_2 \supset Z_1\otimes F_2$ we deduce $W\supset B_1\otimes^! V_2 \oplus Z_1\otimes^! F_2$ against the fact that $z\notin W$ and $z\in  B_1\otimes^! V_2 \oplus Z_1\otimes^! F_2$.

This proves that $z\in Z_1\otimes^! B_2$ and in particular $z \in Z(Z_1\otimes^! B_2)$. We now apply the same argument to prove $z\in Z_1\otimes^! Z_2$.
\end{proof}

\subsubsection{The morphism \texorpdfstring{$\Phi^\zeta_\Sigma$}{between the enveloping algebra and the center}}

Having established the main properties of the center of the enveloping algebra $Z_\kappa(\hat{\gog_\Sigma}) = Z(\calU_\kappa(\hat{\gog}_\Sigma))$ and of the chiral algebra $\zeta^\kappa_\Sigma(\gog)$ we can construct a morphism
\[
	\Phi^\zeta_\Sigma : \calU_{\overline{\Sigma}^*}(\zeta^\kappa_\Sigma(\gog)) \to Z_\kappa(\hat{\gog}_\Sigma).
\]

\begin{corollary}\label{coro:morphismbtwcenter}[Of Proposition \ref{prop:factpropertiescalu}]
	The morphism $\Phi_\Sigma : \calU_{\oSigma^*}(\VkSg) \to \calU_\kappa(\hat{\gog}_\Sigma)$ of Proposition \ref{prop:descrenvelopingalg}, when restricted to $\calU_{\oSigma^*}(\zeta^\kappa_\Sigma(\gog))$ induces a continuous morphism
	\[
	\Phi^\zeta_\Sigma : \calU_{\overline{\Sigma}^*}(\zeta^\kappa_\Sigma(\gog)) \to Z_\kappa(\hat{\gog}_\Sigma).
\]
	The collection of the morphisms $\Phi^\zeta_\Sigma$ is compatible with the pseudo-factorization structures on both spaces.
\end{corollary}

\begin{proof}
	Since $Z_{\kappa}(\hat{\gog}_\Sigma)$ is a complete sheaf of algebras, we just need to check that the natural map $\Lie_{\overline{\Sigma}^*}(\zeta^{\kappa_c}_{\Sigma}(\gog)) \to \calU_{\kappa_c}(\hat{\gog}_{\Sigma})$ takes values inside the center, or equivalently that commutes with $X \otimes f$ for any $X\in \gog$ and $f \in \pbarOpoli$. Let us recall that the above morphism, interpreting elements $z \in \zeta^\kappa_\Sigma(\gog)_{\oSigma^*}$ as fields, is induced by $z \mapsto z(1)$. Then, the condition that $z(1)$ commutes with $X \otimes f$ for any $f \in \pbarOpoli, X \in \gog$, follows by the fact that given $z \in \left(\zeta^{\kappa_c}_{\Sigma}(\gog)\right)_{\overline{\Sigma}^*}$ and $X \in \gog \subset \calV^{\kappa_c}_\Sigma(\gog)$ we have that, $z$ being central, the bracket as fields
	\[
	[z,X] = \mu(z\boxtimes X) = 0.
	\]
	In particular for any $f \in \pbarOpoli$ we have $[z(1),X\otimes f] = [z,X](1\otimes f)= 0$.

	The fact that $\Phi^\zeta_\Sigma$ preserves the (pseudo) factorization structures follows by the fact that the natural map $\calU_{\oSigma^*}(\zeta^\kappa_\Sigma(\gog)) \to \calU_{\oSigma^*}(\VkSg)$ is of complete factorization algebras, by the fact that the pseudo factorization structure on $Z_\kappa(\hat{\gog}_\Sigma)$ is induced by the factorization structure of $\calU_\kappa(\hat{\gog}_\Sigma)$, and by the fact that $\Phi_{\Sigma}$ preserves the factorization structures (see Proposition \ref{prop:factpropertiescalu}).
\end{proof}

\subsection{The center in the case of a single section}

With not too much effort it is possible to upgrade the Feigin-Frenkel Theorem (c.f. Theorem \ref{thm:knownfeiginfrenkel}) to our geometric setting, in the case when $\Sigma = \{\sigma\}$ consists of a single section. In this section we give a sketch on how to do it, but for brevity we decided to move the heavy work to Appendix \ref{sec:feiginfrenkelclassical}, where we also compare Feigin and Frenkel's construction to ours. The reader can therefore limit herself to read the statement of the following Propositions as enhancements of the usual Feigin-Frenkel isomorphism. The reader who wants to dive into the proofs, should first take a look at Appendix \ref{sec:feiginfrenkelclassical}.

The main ingredient in order to prove the desired enhancement of the Feigin-Frenkel isomorphism is the construction $\calU_{\overline{\sigma}^*}(\zeta^{\kappa_c}_\sigma(\gog))$. The proof then follows two main steps:

\begin{enumerate}
	\item There is a canonical isomorphism
	\(
		\gamma_{\sigma} : \calU_{\overline{\sigma}^*}\left(\zeta^{\kappa_c}_\sigma(\gog)\right) \to \Fun(\Op_{\gogl}(\overline{\sigma}^*));
	\)
	\item The morphism constructed in Corollary \ref{coro:morphismbtwcenter}, $ \Phi^\zeta_{\sigma} : \calU_{\overline{\sigma}^*}\left(\zeta^{\kappa_c}_\sigma(\gog)\right) \to Z_{\kappa_c}(\hat{\gog}_\sigma) = Z(\calU_{\kappa_c}(\hat{\gog}_{\sigma}))$ is an isomorphism.
\end{enumerate}

We will deal with them separately, starting with $\gamma_{\sigma}$. We look first at the case where $S$ is affine, well covered and with a coordinate $t$. The idea to construct $\gamma_\sigma$ locally is simple: the complete topological algebras in question can be shown to have the same topological bases. The way we write these topological bases is via suitable isomorphism with the same completed symmetric algebra (c.f. the isomorphisms $\chi^*_{\sigma,t},\chi^*_{\Sigma,U}$).

We will consider $\gog$ and $\gogl$ at the same time and refer to the constructions of Section \ref{ssec:canonicalrepresentativesopers}; we denote by $\check{V}^{\can}$ the analogue of the vector space $V^{\can}$ relative to the Lie algebra $\gogl$, and consider the isomorphism 

\[
	\chi^*_{\sigma,t} : \overline{\Sym}_{\lowOS}((\check{V}^{\can})^*\otimes\calO_{\overline{\sigma}^*}) \to \Fun\left( \Op_{\gogl}(\overline{\sigma}^*) \right).
\]
induced by the description of opers via canonical representatives.

Let us take a look for a moment at the commutative algebra $\mC[\Op_{\gog}(D)]$, which is described via the above constructions in the case where $S = \Spec \mC$, $X = \Spec \mC[z]$ and $\sigma = \sigma_0$ is the $0$ section; when looking at this particular case we will always think of it with a specified coordinate $z$. Via the isomorphism $\chi_{\sigma_0,z} : \Sym(\check{V}^{\can}\otimes \mC((z))/\mC[[z]]) \to \mC[\Op_{\gogl}(D)]$ we get an injection $ (\check{V}^{\can})^* \hookrightarrow \mC[\Op_{\gogl}(D)]$ by considering the restriction of $\chi^*_{\sigma_0,z}$ to $(\check{V}^{\can})^* \otimes 1/z$. This immersion identifies $\mC[\Op_{\gogl}(D)]$ with the universal vertex algebra $V^0((\check{V}^{\can})^*)$ relative to the commutative Lie algebra $(\check{V}^{\can})^*$ at level $0$.

It then follows by Proposition \ref{prop:descrenvelopingalg} that the composition 
\[\begin{tikzcd}[column sep=tiny]
	{(\check{V}^{\can})^*\otimes\calO_{\overline{\sigma}^*}} & {\mC[\Op_{\gogl}(D)] \otimes \calO_{\overline{\sigma}^*}} & {\Lie_{\overline{\sigma}^*}\left(\mC[\Op_{\gogl}(D)] \otimes \calO_{\overline{\sigma}}\right) } & {\calU_{\overline{\sigma}^*}\left( \mC[\Op_{\gogl}(D)]\otimes\calO_{\overline{\sigma}}\right)}
	\arrow[from=1-1, to=1-2]
	\arrow[from=1-2, to=1-3]
	\arrow[from=1-3, to=1-4]
\end{tikzcd}\]
induces an isomorphism
\[
		\chi^*_{\Sigma,U} : \overline{\Sym}_{\lowOS}\left((\check{V}^{\can})^*\otimes\calO_{\overline{\sigma}^*}\right) \to \calU_{\overline{\sigma}^*}\left( \mC[\Op_{\gogl}(D)]\otimes\calO_{\overline{\sigma}}\right).
\]

\begin{proposition}\label{coro:descralphacenterv2}
		In the case where $S = \Spec A$ is affine, noetherian and equipped with a coordinate $t$, the isomorphism $\gamma_{\sigma}$ which makes the following diagram commute
		\begin{equation}\label{eq:diagramgammaonesection}
			\begin{tikzcd}
				{\calU_{\overline{\sigma}^*}(\zeta^{\kappa_c}_{\sigma}({\gog}))} && {\Fun\left( \Op_{\gogl}(\overline{\sigma}^*) \right)} \\
				\\
				{\calU_{\overline{\sigma}^*}\left(\mC[\Op_{\gogl}(D)]\otimes\calO_{\overline{\sigma}}\right)} && {\overline{\Sym}_{\lowOS}\left( (\check{V}^{\can})^* \otimes\calO_{\overline{\sigma}^*} \right)}
				\arrow["\gamma_{\sigma}", from=1-1, to=1-3]
				\arrow["{\calU_{\overline{\sigma}^*}(\calY_t)}", from=3-1, to=1-1]
				\arrow["{\chi_{\sigma,t}^*}"', from=3-3, to=1-3]
					\arrow["{\chi^*_{\sigma,U}}", from=3-3, to=3-1]
			\end{tikzcd}
		\end{equation}
		is independent from the choice of the coordinate $t$. In particular for an arbitrary quasi separated and noetherian $S$ there is a canonical isomorphism $\gamma_\sigma : \calU_{\overline{\sigma}^*}(\zeta^{\kappa_c}_\sigma(\gog)) \to \Fun(\Op_{\gogl}(\overline{\sigma}^*))$ which, restricted to any affine well covered open subset, makes the above diagram commute.
\end{proposition}
	
\begin{proof}
	One can assume that $\calO_{\overline{\sigma}^*} \simeq A((z))$ so that the statement of the Proposition becomes an $A$-linear version of Proposition \ref{rmk:descralphacenter}. The proof given in Appendix \ref{sec:feiginfrenkelclassical} of Propositions \ref{prop:gammaisoFFnocoord} and \ref{rmk:descralphacenter} (which deal with the case $S = \Spec \mC$) translate verbatim to their $A$-linear versions.
\end{proof}

\begin{proposition}\label{cor:feiginfrenkel1section}
    Let $\gog$ be a finite simple Lie algebra over $\mC$ and let $\gogl$ be its Langlands dual Lie algebra. Let $X \to S$ be a smooth family of curves over a fixed base quasi-separated, noetherian scheme $S$ and let $\sigma : S \to X$ be a section. Then the map constructed in Corollary \ref{coro:morphismbtwcenter}, $ \Phi^\zeta_{\sigma} : \calU_{\overline{\sigma}^*}\left(\zeta^{\kappa_c}_\sigma(\gog)\right) \to Z_{\kappa_c}(\hat{\gog}_\sigma) = Z(\calU_{\kappa_c}(\hat{\gog}_{\sigma}))$ is an isomorphism.

    We will refer to the canonical identification
    \[
        \Psi_{\sigma} = \Phi^\zeta_{\sigma}\circ\gamma_\sigma^{-1} : \Fun\left( \Op_{\gogl}\left(\overline{\sigma}^*\right) \right) = Z_{\kappa_c}(\hat{\gog}_{\sigma})
    \]
    as the \emph{Feigin-Frenkel isomorphism}.
\end{proposition}

\begin{proof}
    Remark that the statement is a version of Theorem \ref{thm:feiginfrenkel} in families. Let us start by noticing that the discussion of Remark \ref{rmk:classicalfeifreiso} and Theorem \ref{thm:feiginfrenkel} hold over any base ring $A \in \Aff_{\mC}$, replacing $\mC[[z]]$ with $A[[z]]$. The proof can be emulated in the $A$-linear setting with essentially no significant changes. The only point which requires some work is \textquote{Step $(b)$}: proving an analogue of \cite[Prop. 4.3.4]{frenkel2004vertex} in the $A$-linear setting. This is dealt with in Section 5 of \cite{cas2023}.

    The first point may be checked locally, so that we may assume that $S = \Spec A$ is affine and well covered. By an $A$-linear version of Remark \ref{rmk:comparisonutildePHI}, the choice of any coordinate identifies $\Phi^\zeta_\sigma$ with the morphism $\Phi^\zeta_{\text{FF}}$ appearing in the proof of Remark \ref{rmk:classicalfeifreiso}, so that by the $A$-linear version of Theorem \ref{thm:feiginfrenkel} we know it is an isomorphism.
\end{proof}

We conclude with a remark on how the topologies of $\Fun(\Op_{\gogl}(\overline{\sigma}^*))$ and $Z_{\kappa_c}(\hat{\gog}_\sigma)$ are related via the Feigin-Frenkel isomorphism; following Proposition \ref{prop:comparisontopologiesonesingularity}.

\begin{definition}\label{def:idealsonesection}
	Consider a non negative integer $k$; we define the following sheaves:
	\begin{itemize}
		\item We write $\calI_{\sigma,k}$ for the closed left ideal $\calU_{\kappa_c}(\hat{\gog}_\sigma)\cdot (\gog \otimes \calO_{\overline{\sigma}}(-k)) \subset \calU_{\kappa_c}(\hat{\gog}_\sigma)$ and $\calI^Z_{\sigma,k} = Z_{\kappa_c}(\hat{\gog}_{\sigma}) \cap \calI_{\sigma,k}$. These form a basis for the topology of $Z_{\kappa_c}(\hat{\gog}_\sigma)$; 
		\item Consider a basis $x_l$ of $\check{V}^{\can}$ such that $[\frac{1}{2}\check{h}_0,x_l] = d_lx_l$ (c.f. Section \ref{ssec:canonicalrepresentativesopers}) and let $x_l^*$ be its dual basis; we write $\calJ_{\sigma,k}$ for the closed ideal of $\overline{\Sym}((\check{V}^{\can})^*\otimes\calO_{\overline{\sigma}^*})$ generated by $x_l^* \otimes \calO_{\overline{\sigma}}(-d_lk)$ so that 
		\[
			\overline{\Sym}_{\lowOS}((\check{V}^{\can})^*\otimes\calO_{\overline{\sigma}^*})/\calJ_{\sigma,k} = \Sym_{\lowOS}\left( \bigoplus_l x_l^* \otimes (\calO_{\overline{\sigma}^*}/\calO_{\overline{\sigma}}(-d_lk)) \right).
		\] 
		These are open ideals and form a basis for the topology of the sheaf $\overline{\Sym}_{\lowOS}((\check{V}^{\can})^*\otimes\calO_{\overline{\sigma}^*})$;
		\item In the case where $S$ is affine and well covered with a fixed coordinate $t$, we define the ideal $\calJ^{\Op,t}_{\sigma,k} \subset \Fun(\Op_{\gogl}(\overline{\sigma}^*))$ to be the image of $\calJ_{\sigma,k}$ along the topological isomorphism
		\[
			\chi^*_{\sigma,t} : \overline{\Sym}_{\lowOS}\left((\check{V}^{\can})^* \otimes \calO_{\overline{\sigma}^*}\right) \to \Fun \left( \Op_{\gogl}(\overline{\sigma}^*) \right).
		\]
		Just as before, these ideals form a basis for the topology of the sheaf $\Fun(\Op_{\gogl}(\overline{\sigma}^*))$.
	\end{itemize}
\end{definition}

\begin{lemma}\label{cor:comparisontopologiesonesection}
	Assume that $S$ is affine and well covered. Then, after the choice of a coordinate $t$, the Feigin-Frenkel isomorphism $\Psi_\sigma$ identifies $\calJ^{\Op,t}_{\sigma,k}$ with $\calI^Z_{\sigma,k}$.
\end{lemma}

\begin{proof}
	Again, as in the case of the previous Propositions, this Lemma is an $A$-linear version of Proposition \ref{prop:comparisontopologiesonesingularity}. 
\end{proof}

\subsection{The map \texorpdfstring{$\gamma$}{gamma} in case of multiple sections}

Here we construct an analogue of the isomorphism $\gamma$ in the case of multiple sections. The strategy to define such a canonical isomorphism is to construct it locally with the aid of a coordinate $t$ and then show that its formation is actually independent from the choice of such a coordinate. This will be achieved by exploiting the factorization properties of $\gamma$ in combination with Proposition \ref{coro:descralphacenterv2}.

Recall the description of $\Fun (\Op_{\gogl}(\oSigma^*))$ and of $\Fun(\Op_{\gogl}(\oSigma))$ given in Section \ref{ssec:canonicalrepresentativesopers}; in particular in the case where $S$ is affine and well covered with a coordinate $t$, consider the isomorphisms therein constructed:
\begin{align*}
	\chi^*_{\Sigma,t} &: \overline{\Sym}_{\lowOS}\left( \check{V}^{\can} \otimes \pbarOpoli \right) \to \Fun\left( \Op_{\gogl}(\oSigma^*)\right), \\
	\chi^*_{\Sigma,t} &: \Sym_{\lowOS}\left( \check{V}^{\can} \otimes \pbarOpoli/\pbarO \right) \to \Fun\left( \Op_{\gogl}(\oSigma)\right).
\end{align*}
Recall that $\chi^*_{\Sigma,t}$ preserves the factorization structures as well.
\smallskip

As in the previous section let us consider the commutative $\mC$-algebra $\mC[\Op_{\gogl}(D)]$, and as done before, consider $D$ equipped with a specified coordinate $z$. In this case, the isomorphism $\chi^*_z: \Sym(\check{V}^{\can} \otimes \mC((z))/\mC[[z]]) \to \mC[\Op_{\gogl}(D)]$ induces an injection $(\check{V}^{\can})^* \hookrightarrow \mC[\Op_{\gogl}(D)]$ by considering the restriction of $\chi^*_z$ to $(\check{V}^{\can})^*\otimes 1/z$. This injection induces an isomorphism of commutative vertex algebras between $\mC[\Op_{\gogl}(D)]$ and $V^0((\check{V}^{\can})^*)$, the universal vertex algebra of level $0$ for the abelian Lie algebra $(\check{V}^{\can})^*$. It then follows from Proposition \ref{prop:descrenvelopingalg} that the composition
\[\begin{tikzcd}[column sep=tiny]
	{(\check{V}^{\can})^*\otimes\pbarOpoli} & {\mC[\Op_{\gogl}(D)] \otimes \pbarOpoli} & {\Lie_{\overline{\Sigma}^*}\left(\mC[\Op_{\gog}(D)] \otimes \pbarO\right) } & {\calU_{\overline{\Sigma}^*}\left( \mC[\Op_\gog(D)]\otimes\pbarO\right)}
	\arrow[from=1-1, to=1-2]
	\arrow[from=1-2, to=1-3]
	\arrow[from=1-3, to=1-4]
\end{tikzcd}\]
induces an isomorphism
\[
	\chi^*_{\Sigma,U} : \overline{\Sym}_{\lowOS}\left((\check{V}^{\can})^*\otimes\pbarOpoli\right) \to \calU_{\overline{\Sigma}^*}\left( \mC[\Op_{\gogl}(D)]\otimes\pbarO\right)	\]
which is compatible with the factorization structures on both sides.

\smallskip

Using the isomorphisms $\chi^*_{\Sigma,t},\chi^*_{\Sigma,U}$, we now construct an isomorphism
\[
	\gamma_{\Sigma,t} : \calU_{\oSigma^*}\left( \zeta^{\kappa_c}_{\Sigma}(\gog) \right) \to \Fun\left(\Op_{\gogl}(\oSigma^*) \right)
\] 
locally and then show that it is independent from the choice of a coordinate in Theorem \ref{thm:feiginfrenkelcasarinmaffei}. The construction of $\gamma_{\Sigma,t}$ is done in analogy with diagram \eqref{eq:diagramgammaonesection}.

\begin{definition} Assume that $S$ is affine and well covered and let $t \in \pbarO(S)$ be a coordinate. We define locally $\gamma_{\Sigma,t} : \calU_{\overline{\Sigma}^*}\left(\zeta^{\kappa_c}_\Sigma(\gog)\right) \to \Fun\left( \Op_{\gogl}(\overline{\Sigma}^*) \right)$ in order to make the following diagram commute:

\[\begin{tikzcd}
	{\calU_{\overline{\Sigma}^*}(\zeta^{\kappa_c}_{\Sigma}(\gog))} && {\Fun\left( \Op_{\gogl}(\overline{\Sigma}^*) \right)} \\
	\\
	{\calU_{\overline{\Sigma}^*}\left(\mC[\Op_{\gogl}(D)]\otimes\pbarOpoli\right)} && {\overline{\Sym}_{\lowOS}\left( (\check{V}^{\can})^* \otimes\pbarOpoli \right)}
	\arrow["{\gamma_{\Sigma,t}}", from=1-1, to=1-3]
	\arrow["{\calU_{\overline{\Sigma}^*}(\calY_t)}", from=3-1, to=1-1]
	\arrow["{\chi_{\Sigma,t}^*}"', from=3-3, to=1-3]
	\arrow["{\chi^*_{\Sigma,U}}", from=3-3, to=3-1]
\end{tikzcd}\]
This is possible since every other map is an isomorphism, so that $\gamma_{\Sigma,t}$ is well defined and an isomorphism. As the left vertical map is concerned, recall that $\calY_t : \mC[\Op_{\gogl}(D)]\otimes\pbarO \to \zeta^{\kappa_c}_{\Sigma}(\gog)$ is an isomorphism by \cite[Corollary \ref{coro:subchiralalg}]{casmaffei1} (see also Section \ref{ssec:richiamiVkSg}).
\end{definition}

\begin{remark}\label{rmk:factorizationofgamma}
	The isomorphism $\gamma_{\Sigma,t}$ is compatible with the factorization structures. This follows by the fact that it is constructed as a composition of isomorphisms compatible with the factorization structures.
\end{remark}

\begin{theorem}\label{thm:feiginfrenkelcasarinmaffei}
	Assume that $S$ is affine, noetherian, integral and well covered. Let $\Sigma$ be an arbitrary finite set of sections of $X \to S$, then the morphism $\gamma_{\Sigma,t}$ is independent from the choice of the coordinate $t$. It follows that for an arbitrary quasi-separated, noetherian and integral scheme there is a canonical isomorphism of commutative complete topological $\calO_S$-algebras
	\[
		\gamma_{\Sigma} : \calU_{\overline{\Sigma}^*}\left(\zeta^{\kappa_c}_\Sigma(\gog)\right) \to \Fun\left( \Op_{\gogl}(\overline{\Sigma}^*) \right).
	\]
\end{theorem}

\begin{proof}
	Recall that by Proposition \ref{coro:descralphacenterv2} we already know the assertion of the Proposition in the case of a single section, since $\gamma_{\Sigma,t}$ is exactly the isomorphism provided by the Feigin-Frenkel Theorem. Without loss of generality we can assume that there are no components of $\Sigma$ which are equal and consider $j : S_{\neq} \hookrightarrow S$ be the open immersion corresponding to the open subset where all sections are different to one another (with respect to Definition \ref{def:openclosedfactorization}, we have $S_{\neq} = U_{I/I}$), which is therefore non-empty. Given two different coordinates $t,s$ consider the diagram
	\[\begin{tikzcd}[column sep=tiny]
	{\calU_{\overline{\Sigma}^*}\left(\zeta^{\kappa_c}_\Sigma(\gog)\right)} && {\calU_{\overline{\Sigma}^*}\left(\zeta^{\kappa_c}_\Sigma(\gog)\right)} \\
	& {\Fun\left( \Op_{\gogl}(\overline{\Sigma}^*) \right)} \\
	{j_*\hat{j}^*\left( \calU_{\overline{\Sigma}^*}\left(\zeta^{\kappa_c}_\Sigma(\gog)\right) \right) } && {j_*\hat{j}^*\left( \calU_{\overline{\Sigma}^*}\left(\zeta^{\kappa_c}_\Sigma(\gog)\right) \right) } \\
	& {j_*\hat{j}^*\left( \Fun\left( \Op_{\gogl}(\overline{\Sigma}^*) \right) \right)} \\
	{j_*\hat{j}^*\left( \bigotimes^!_{i\in I} \calU_{\overline{\sigma}_i^*}\left(\zeta^{\kappa_c}_{\sigma_i}(\gog)\right) \right)}&& {j_*\hat{j}^*\left( \bigotimes^!_{i\in I} \calU_{\overline{\sigma}_i^*}\left(\zeta^{\kappa_c}_{\sigma_i}(\gog)\right) \right)} \\
	& {j_*\hat{j}^*\left( \bigotimes^!_{i\in I} \Fun\left( \Op_{\gogl}(\overline{\sigma}_i^*) \right)\right)} \\
	&& {}
	\arrow[dashed, no head, from=1-1, to=1-3]
	\arrow["{\gamma_{\Sigma,s}}", from=1-1, to=2-2]
	\arrow[hook, from=1-1, to=3-1]
	\arrow["{\gamma_{\Sigma,t}}"', from=1-3, to=2-2]
	\arrow[hook, from=1-3, to=3-3]
	\arrow[hook,from=2-2, to=4-2]
	\arrow[dashed, no head, from=3-1, to=3-3]
	\arrow["{j_*\hat{j}^*\gamma_{\Sigma,s}}"', from=3-1, to=4-2]
	\arrow["{(\fact^{\calU})^{-1}}"', from=3-1, to=5-1]
	\arrow["{j_*\hat{j}^*\gamma_{\Sigma,t}}", from=3-3, to=4-2]
	\arrow["{(\fact^{\calU})^{-1}}", from=3-3, to=5-3]
	\arrow["{(\fact^{\Op})^{-1}}"{description,pos=0.3}, from=4-2, to=6-2]
	\arrow[dashed, no head, from=5-1, to=5-3]
	\arrow["{j_*\hat{j}^*(\otimes^! \gamma_{\sigma_i,s_i})}"', from=5-1, to=6-2]
	\arrow["{j_*\hat{j}^*(\otimes^! \gamma_{\sigma_i,t_i})}", from=5-3, to=6-2]
\end{tikzcd}\]
	where the dotted lines stand for the identity morphisms. The top vertical arrows are injective because we are assuming that $S$ is integral and the sheaves we are considering are QCCF By the case of a single section we know that the bottom triangle commutes. We know that each single square commutes as well by Remark \ref{rmk:factorizationofgamma}. Since the first vertical arrows are injections it follows that the top triangle commutes as well and hence $\gamma_{\Sigma,s} = \gamma_{\Sigma,t}$.

	Since the isomorphism $\gamma_{\Sigma,t}$ is independent from the choice of the coordinate $t$ we may glue the local isomorphisms to a global canonical isomorphism
	\[
		\gamma_\Sigma : \calU_{\overline{\Sigma}^*}\left(\zeta^{\kappa_c}_\Sigma(\gog)\right) \to \Fun\left( \Op_{\gogl}(\overline{\Sigma}^*) \right). \qedhere
	\]
\end{proof}

\subsection{The center in the case of multiple sections}

In this section we prove that the map $\Phi^\zeta_\Sigma$ from Corollary \ref{coro:morphismbtwcenter} is an isomorphism. We will directly consider the composition 
\[
	\Psi_{\Sigma} = \Phi^{\zeta}_{\Sigma} \circ \gamma_{\Sigma}^{-1} : \Fun\left( \Op_{\gogl}(\oSigma^*) \right) \to Z_{\kappa_c}(\hat{\gog}_{\Sigma})
\]
of the map $\Phi^{\zeta}_{\Sigma}$ of Corollary \ref{coro:morphismbtwcenter} and the inverse of the map $\gamma_{\Sigma}$ of Theorem \ref{thm:feiginfrenkelcasarinmaffei}.
\noindent Recall the results on the center $Z_{\kappa_c}(\hat{\gog}_\Sigma) = Z(\calU_{\kappa_c}(\hat{\gog}_\Sigma))$ of Section \ref{ssec:factorizationcenter}.

\begin{comment}

\begin{proposition}\label{lem:isocriteria}
	Assume that $M,N$ are complete topological $A$ modules, assume that the topology on both modules is countable and let $\{M_m\}_{m \in \mN},\{N_n\}_{n \in \mN}$ be fundamental system of neighborhoods of $0$ for $M$ and $N$ respectively. Let $f \in A$ be such that $M/M_m,N/N_n$ have no $f$-torsion. Denote by $i : \Spec A/f \to \Spec A, j : \Spec A_f \to \Spec A$ the corresponding closed and open immersions. Denote by $\hat{i}^*,\hat{j}^*$ the corresponding completed pullbacks between the categories of complete topological modules. Let $\varphi : M \to N$ be a continuous morphism. 
	
	If $\hat{j}^*\varphi$ and $\hat{i}^*\varphi$ are bicontinuous isomorphisms so is $\varphi$.
\end{proposition}

\end{comment}

We will start with some technical Lemmata which will allow us to prove Theorem \ref{thm:teofinale1} by induction using the factorization properties. 

\noindent Let us start by recalling that  if
\[\xymatrix{
	0 \ar[r]&  X_n \ar[r] & Y_n \ar[r]& Z_n \ar[r] & 0
}\]
is a projective system of exact sequences indexed by $\mN$ and $X_n\lra X_{n-1}$ are surjective maps, then the sequence of their limits is exact (see for example \cite[\href{https://stacks.math.columbia.edu/tag/0598}{Tag 0598}]{stacks-project}).

In particular if $N$ is a complete topological module with a countable fundamental system of neighborhoods of $0$, and $M$ is a closed submodule the quotient $N/M$ with the quotient topology is complete. Indeed if $V_n$ is a decreasing fundamental system of neighborhoods of $0$ for $N$, then we have exact sequences
\[\xymatrix{
	0 \ar[r]&  \frac{M}{V_n\cap M} \ar[r] & \frac{N}{V_n} \ar[r] & \frac{N}{M+V_n} \ar[r] & 0
}\]
and taking the limit we get the desired claim. 

\begin{lemma}\label{lem:isocriterialemma}
	Assume that $M$ is a complete topological $A$-module whose topology is defined by a countable system of noz $\{ M_n \} _ { n \in \mN }$. 
	Let $f \in A$ be such that $M/M_n$ has no $f$-torsion for all $n$. 
	Denote by $i : \Spec A/f \to \Spec A, j : \Spec A_f \to \Spec A$ the corresponding closed and open immersions. 
	Denote by $\hat{i}^*,\hat{j}^*$ the corresponding completed pullbacks between the categories of complete topological modules. Then 
	\begin{enumerate}
		\item $\hat{i}^*M\simeq M/fM$ with the quotient topology. In particular $fM$ is closed in $M$ and the quotient $M/fM$ is complete.  
		\item The natural map from $M$ to $\hat j^*M$ is injective, the image of $M$ is closed and $M$ has the subspace topology. More precisely $\hat j ^* M_n\cap M=M_n$. 
	\end{enumerate}
\end{lemma}

\begin{proof}
	Recall that $\hat{j}^*M = \varprojlim_{n \in \mN} (M/M_n)_f$ and that 
	\[
	\hat{i}^*M = \varprojlim_{n \in \mN} \frac{M/M_n}{f(M/M_n)}=\varprojlim_{n \in \mN} \frac{M}{fM+M_n}.
	\]
	By assumption the sequence
	\[\xymatrix{
		0 \ar[r]& 
		\frac{M}{M_n} \ar[r]^{\cdot f}&
		\frac{M}{M_n} \ar[r]&
		\frac{M/M_n}{f(M/M_n)} \ar[r]&
		0
	}\]
	is exact.
	%. By \cite[\href{https://stacks.math.columbia.edu/tag/02MY}{Tag 02MY}]{stacks-project}, since the modules on the left satisfy the Mittag-Leffler condition, taking the limit of this system preserves exactness. 
	Taking the limit we obtain the claim about $\hat{i}^*M$.
	
	Let us now prove the statement about the module $\hat j^*M$. Since $M/M_n$ has no $f$ torsion the map $M/M_n\lra (M/M_n)_f$ is injective hence if $x\in M\cap \hat j^*M_n$ then the image of $x$ in $M/M_n$ must be zero, hence $x\in M_n$. This implies the last claim. 
	
	Finally the multiplication by $f$ is invertible in $j^*M$. 
\end{proof}

In order to prove our main Theorem \ref{thm:teofinale1} we need the following algebraic Lemma.

\begin{lemma}\label{lem:isocriteria}
	Assume that $M,N$ are complete topological $A$ modules, assume that the topology on both modules is countable and let $\{M_k\}_{k \in \mN},\{N_k\}_{k \in \mN}$ be fundamental system of neighborhoods of $0$ for $M$ and $N$ respectively. Let $f \in A$ be an arbitrary element and assume that the quotients $M/M_k,N/N_k$ have no $f$-torsion. Denote by $i : \Spec A/f \to \Spec A, j : \Spec A_f \to \Spec A$ the corresponding closed and open immersions and by $\hat{i}^*,\hat{j}^*$ the corresponding completed pullbacks between the categories of complete topological modules. Let $\varphi : M \to N$ be a continuous morphism. The following hold. 
	\begin{enumerate}
		\item Assuming that $\hat{j}^*\varphi$ is a bicontinuous isomorphism, that $\hat{i}^*\varphi$ is an isomorphism, that $\varphi(M_k) \subset N_k$, and that $\hat{i}^*M_k \to \hat{i}^*N_k$ is an isomorphism for every $k$, then the map $\varphi : M \to N$ is a bicontinuous isomorphism;
		\item If we further assume that $\varphi$ induces bicontinuous isomorphisms $\hat{j}^*M_k \to \hat{j}^*N_k$, then $\varphi$ induces bicontinuous isomorphisms $M_k \to N_k$.
	\end{enumerate}
\end{lemma}

\begin{proof}
By point $2$ of Lemma \ref{lem:isocriterialemma} $M$ and $N$ are submodules of $\hat j^*M$ and $\hat j^* N$ and they have the induced topologies. In particular, since $\hat{j}^*\varphi : \hat{j}^*M \simeq \hat{j}^*N$ this implies that $\grf$ is injective and that $M$ has the topology induced by $N$. Since $M$ is complete, the image of $M$ in $N$ is closed and $Q=N/M$ is complete with respect to the quotient topology. In order to prove point $1$ of the lemma it is therefore enough to prove that $Q=0$. 

Notice that a fundamental system of neighborhoods of $0$ for $Q$ is given by $Q_k=M+N_k/M$. 

We claim that $Q/Q_k = N/(M + N_k)$ has no $f$-torsion. Let us start by noticing that by point $1$ of Lemma \ref{lem:isocriterialemma} and by the assumption that $\hat{i}^*M_k = \hat{i}^*N_k$ it follows that $M_k/fM_k = N_k/fN_k$ so that $N_k = M_k + fN_k$, in addition, by the assumption that $M/fM = N/fN$, it follows that $Q$ has no $f$-torsion. In order to prove that $Q/Q_k = N/(M+N_k)$ has no $f$-torsion, assume we have $n \in N$  such that $fn \in M + N_k$. Write $fn = m + n_s$. Since $N_k = M_k +  fN_k$, we may write $n_s = m' + fn_s'$ for some $m' \in M_k \subset M$ and $n_s' \in N_k$ it follows that $f(n - n_s') = m + m'$ and by the fact that $N/M$ has no $f$-torsion it follows that there exists some $m'' \in M$ such that $n - n_s' = m''$. We conclude that $n = m'' + n_s'$ so that $n = 0$ in $N/(M+N_k)$ and therefore the latter has no $f$-torsion.  

Finally, we claim that $\hat j ^*Q=0$. In order to prove that, let $M'_k=\grf^{-1}(N_k)$, since $M$ has the topology induced by $N$, the submodules $M_k'$ form a fundamental system of neighborhoods of $0$ for $M$ and since $M/M'_k\subset N/N_k$ the modules $M/M_k'$ have no $f$-torsion.  We have exact sequences 
\[\xymatrix{ 0 \ar[r] & 
	\left(\frac{M}{M'_k}\right)_f \ar[r]  &
	\left(\frac{N}{N_k} \right)_f \ar[r]  &
	\left(\frac{Q}{Q_k} \right)_f \ar[r]  & 0 }\]
Taking the limit over $k$ this sequence remains exact and using that $\hat j^* \grf $ is an isomorphism we see that $\hat{j}^*Q = 0$. Since $Q$ satisfies the hypothesis of Lemma \ref{lem:isocriterialemma} ($Q/Q_k$ has no $f$-torsion) we have $Q \subset \hat{j}^*Q$ so that $Q$ must be $0$ as well.

Point $2$ of the Lemma follows by applying point $1$ to $M = M_k,N=N_k$.
\end{proof}

In order to use the above Lemma, we need a specified system of neighborhoods of zero for $\Fun\left( \Op_{\gogl}(\overline{\Sigma}^*) \right)$ and $Z_{\kappa_c}(\hat{\gog}_\Sigma)$. We are led in their definition by Lemma \ref{cor:comparisontopologiesonesection}. 

\begin{definition}\label{def:indices}
	We start with some index notation. Let $\Sigma$ a set of sections of $X \to S$ indexed by a finite set $J$. We consider $J$-tuples of non negative integers, to be denoted by $\underline{k} = (k_j)_{j \in J}$, and to such an $J$-tuple we attach the ideal $\pbarO(-\underline{k}) \subset \pbarO$ which we define as the kernel of $\pbarO \to p_*\left(\calO_X/\prod_{j \in J} \calI_{\sigma_j}^{k_j}\right)$ where $\calI_{\sigma_j}$ is the ideal in $\calO_X$ defining the section $\sigma_j$. These are open ideals and form a refinement of the system of neighborhoods $\pbarO(-n)$ of $0$ for $\pbarO$.

	Given a surjection $\pi : J \twoheadrightarrow I$ we write $\underline{k}_{J/I}$ for the $I$-tuple $(\underline{k}_{J/I})_i = \sum_{\pi(j) = i} k_j$ and $\underline{k}_{J_i}$ for the $J_i = \pi^{-1}(i)$-tuple $(\underline{k}_{J_i})_j = k_j$.
\end{definition}

\begin{remark}\label{rmk:behaviormultiindexbasechange}
	Consider a collection of sections of $X \to S$ indexed by a finite set $J$ and a surjection $J \twoheadrightarrow I$. Let $i_{J/I} : V_{J/I} \to S$ and $j_{J/I} : U_{J/I} \to S$ be the morphisms introduced in Definition \ref{def:openclosedfactorization}. Then
	\begin{enumerate}
		\item Under the identification $\hat{i}_{J/I}^*\pbarO = \calO_{\overline{\Sigma}_I}$ the ideal $\hat{i}_{J/I}^*\pbarO(-\underline{k})$ identifies with the ideal $\calO_{\overline{\Sigma}_I}(-\underline{k}_{J/I})$;
		\item Under the identification $\hat{j}_{J/I}^*\pbarO = \prod_{i \in I} \calO_{\overline{\Sigma}_{J_i}}$ the ideal $\hat{j}_{J/I}^*\pbarO(-\underline{k})$ identifies with the ideal $\prod_{i \in I}\calO_{\overline{\Sigma}_{J_i}}(-\underline{k}_{J_i})$;
	\end{enumerate}
\end{remark}

\noindent We move on defining the system of neighborhoods of interest, taking Definition \ref{def:idealsonesection} as a guideline.

\begin{definition}\label{def:topologicalideals}
	Let $\Sigma$ be a collection of sections of $X \to S$ indexed by a finite set $J$ and let $\underline{k} = (k_j)_{j \in J}$ a $J$-tuple of non-negative integers. We consider the following constructions: 
	\begin{enumerate}
		\item Consider a basis $x_l$ of $\check{V}^{\can}$ such that $[\frac{1}{2}\check{h}_0,x_l] = d_lx_l$ (c.f. Section \ref{ssec:canonicalrepresentativesopers}) and let $x_l^*$ be its dual basis. Let $\calJ_{\Sigma,\underline{k}}$ be the closed ideal of $\overline{\Sym}_{\lowOS}((\check{V}^{\can})^*\otimes \pbarOpoli)$ defined as the kernel of 
			\[
				\overline{\Sym}_{\lowOS}((\check{V}^{\can})^* \otimes \pbarOpoli) \to \Sym_{\lowOS}\left( \bigoplus_l x_l^* \otimes \frac{\pbarOpoli}{\pbarO(-d_l\underline{k})} \right).
			\]
		\item In the case $S$ is well covered with a coordinate $t$, we define $\calJ^{\Op,t}_{\Sigma,\underline{k}}$ to be the image of $\calJ_{\Sigma,\underline{k}}$ along the topological isomorphism
		\[
			\chi^*_t : \overline{\Sym}_{\lowOS}\left( (\check{V}^{\can})^* \otimes \pbarOpoli \right) \simeq \Fun\left( \Op_{\gogl}(\overline{\Sigma}^*) \right).
		\]
		These form a fundamental system of neighborhoods for $\Fun(\Op_{\gogl}(\oSigma^*))$;
		\item Let $\calI_{\Sigma,\underline{k}}$ be the left ideal of $\calU_{\kappa_c}(\hat{\gog}_\Sigma)$ defined by 
			\[
				\calI_{\Sigma,\underline{k}} = \ugsig{\kappa_c} \cdot \gog \otimes \pbarO(-\underline{k})
			\]
		and let $\calI^Z_{\Sigma,\underline{k}} = \calI_{\Sigma,\underline{k}}\cap Z_{\kappa_c}(\hat{\gog}_\Sigma)$.
	\end{enumerate}
\end{definition}

The following Lemma directly follows from Remark \ref{rmk:behaviormultiindexbasechange}; we state it precisely since the following are the objects that will be needed in the proof of Theorem \ref{thm:teofinale1}.

\begin{lemma}\label{lem:behaviormultiindexideals}
	Fix a collection of sections $\Sigma$ indexed by a finite set $J$ and fix a surjection $J \twoheadrightarrow I$. Consider the notation of Definition \ref{def:indices} and the morphisms $i_{J/I} : V_{J/I} \to S$, $j_{J/I} : U_{J/I} \to S$ of Definition \ref{def:openclosedfactorization}.
	\begin{enumerate}
		\item The isomorphism $\Ran^{L_{\oSigma}\check{V}^{\can}}_{J/I} : \hat{i}^*_{J/I}\overline{\Sym}_{\lowOS}((\check{V}^{\can})^*\otimes\pbarOpoli) = \overline{\Sym}_{\lowOS}((\check{V}^{\can})^*\otimes\calO_{\overline{\Sigma}^*_I})$ identifies $\hat{i}^*_{J/I}\calJ_{\Sigma,\underline{k}}$ with $\calJ_{\Sigma_I, \underline{k}_{J/I}}$; 
		\item The isomorphism $\Ran^{\Op}_{J/I} : \hat{i}^*_{J/I}\Fun(\Op_{\gogl}(\oSigma^*))  = \Fun(\Op_{\gogl}(\overline{\Sigma}^*_{I}))$ identifies $\hat{i}^*_{J/I}\calJ^{\Op}_{\Sigma,\underline{k}}$ with $\calJ^{\Op,t}_{\Sigma_I,\underline{k}_{J/I}}$;
		\item The isomorphism $\Ran^{\calU}_{J/I} : \hat{i}^*_{J/I}\ugsig{\kappa_c} = \calU_{\kappa_c}(\hat{\gog}_{\Sigma_I})$ identifies $\hat{i}^*_{J/I}\calI_{\Sigma,\underline{k}}$ with $\calI_{\Sigma_I, \underline{k}_{J/I}}$;
		\item The isomorphism $\fact^{L_{\oSigma}\check{V}^{\can}}_{J/I} : \hat{j}^*_{J/I}\overline{\Sym}_{\lowOS}((\check{V}^{\can})^*\otimes\pbarOpoli) = \hat{j}^*_{J/I}\otimes^!_{i \in I}\overline{\Sym}_{\lowOS}((\check{V}^{\can})^*\otimes\calO_{\overline{\Sigma}_{J_i}})$ identifies $\hat{j}^*_{J/I}\calJ_{\Sigma,\underline{k}}$ with 
		$$\ker \left(\hat{j}^*_{J/I} \otimes^!_{i \in I}  \overline{\Sym}_{\lowOS}((\check{V}^{\can})^*\otimes\calO_{\overline{\Sigma}_{J_i}}) \to j_{J/I}^*\otimes_{i\in I} \overline{\Sym}_{\lowOS}((\check{V}^{\can})^*\otimes\calO_{\overline{\Sigma}_{J_i}})/\calJ_{\Sigma_I, \underline{k}_{J_i}} \right);$$
		\item The isomorphism $\fact^{\Op}_{J/I} : \hat{j}^*_{J/I}\Fun(\Op_{\gogl}(\oSigma^*)) = \hat{j}^*_{J/I}\otimes^!_{i \in I}\Fun(\Op_{\gogl}(\overline{\Sigma}^*_{J_i}))$ identifies $\hat{j}^*_{J/I}\calJ^{\Op,t}_{\Sigma,\underline{k}}$ with 
		$$\ker \left(\hat{j}^*_{J/I} \otimes^!_{i \in I}  \Fun(\Op_{\gogl}(\overline{\Sigma}^*_{J_i})) \to j_{J/I}^*\otimes_{i\in I} \Fun(\Op_{\gogl}(\overline{\Sigma}^*_{J_i}))/\calJ^{\Op,t}_{\Sigma_I, \underline{k}_{J_i}} \right);$$
		\item The isomorphism $\fact^{\calU}_{J/I} : \hat{j}^*_{J/I}\ugsig{\kappa_c} = \hat{j}^*_{J/I}\otimes^!_{i \in I}\calU_{\kappa_c}(\hat{\gog}_{\Sigma_{J_i}})$ identifies $\hat{j}^*_{J/I}\calI_{\Sigma,\underline{k}}$ with $$\ker \left( \hat{j}^*_{J/I}\otimes^!_{i \in I}  \calU_{\kappa_c}(\hat{\gog}_{\Sigma_{J_i}}) \to j^{*}_{J/I}\otimes_{i\in I} \calU_{\kappa_c}(\hat{\gog}_{\Sigma_{J_i}})/\calI_{\Sigma_I, \underline{k}_{J_i}} \right).$$
	\end{enumerate}
\end{lemma}

\begin{proof}
	Points $1$ and $4$ follow by the fact that the factorization structure on $\overline{\Sym}_{\lowOS}((\check{V}^{\can})^*\otimes \pbarOpoli)$ is induced by that of $\pbarOpoli$ together with Remark \ref{rmk:behaviormultiindexbasechange}; analogously points $3$ and $6$ follow by the fact that the factorization structure on $\calU_{\kappa_c}(\hat{\gog}_{\Sigma})$ is induced by that of $\pbarOpoli$ together with Remark \ref{rmk:behaviormultiindexbasechange}. Points $2$ and $5$ follow by the fact that $\chi^*_{\Sigma,t}$ preserves the factorization structures. 
\end{proof}

\begin{corollary}
	Let $\Sigma$ be a collection of sections of $X \to S$ indexed by a finite set $J$ and assume all sections $\sigma_j$ are distinct, so that the open set $U_{J/J} \subset S$ (c.f. Definition \ref{def:openclosedfactorization}) where all sections are disjoint is not empty. Assume that $S$ is affine, integral and well covered with a coordinate $t$ as well. Then with respect to the above notation $\Psi_\Sigma\left( \calJ^{\Op,t}_{\Sigma,\underline{k}}\right) \subset \calI^Z_{\Sigma,\underline{k}}$.
\end{corollary}

\begin{proof}
	Since $\calI^Z_{\Sigma,\underline{k}} = Z_{\kappa_c}(\hat{\gog}_\Sigma) \cap \calI_{\Sigma,\underline{k}}$ it is enough to check that the image of $\calJ^{\Op,t}_{\Sigma,\underline{k}}$ along the composition
	\(
		\Fun\left( \Op_{\gogl}(\oSigma^*)\right) \xrightarrow{\Psi_{\Sigma}} Z_{\kappa_c}(\hat{\gog}_\Sigma) \subset \ugsig{\kappa_c} \to \ugsig{\kappa_c}/\calI_{\overline{\Sigma},\underline{k}}
	\)
	is $0$. The sheaf $\ugsig{\kappa_c}/\calI_{\Sigma,\underline{k}}$ is free as an $\calO_S$-module, so that from the fact that $S$ is integral it follows that $\ugsig{\kappa_c}/\calI_{\Sigma,\underline{k}} \subset (j_{J/J})_*j_{J/J}^*(\ugsig{\kappa_c}/\calI_{\Sigma,\underline{k}})$. We reduce ourselves to check that the image of $\calJ^{\Op,t}_{\Sigma,\underline{k}}$ is zero along
	\[
		\Fun\left( \Op_{\gogl}(\oSigma^*)\right) \xrightarrow{\Psi_{\Sigma}} \ugsig{\kappa_c} \to \ugsig{\kappa_c}/\calI_{\overline{\Sigma},\underline{k}} \to (j_{J/J})_*(j_{J/J})^*\left(\ugsig{\kappa_c}/\calI_{\Sigma,\underline{k}}\right).
	\]
	Consider the commutative diagram
	\[\begin{tikzcd}
	{\Fun\left( \Op_{\gogl}(\oSigma^*)\right)} & {\ugsig{\kappa_c}} & {\ugsig{\kappa_c}/\calI_{\Sigma,\underline{k}}} \\
	\\
	{(j_{I/I})_*\hat{j}^*_{I/I}\left(\Fun\left( \Op_{\gogl}(\oSigma^*)\right) \right)} & {(j_{I/I})_*\hat{j}^*_{I/I}\left(\calU_{\kappa_c}(\hat{\gog}_{\Sigma}) \right)} & {(j_{I/I})_*j_{I/I}^*\left( \ugsig{\kappa_c}/\calI_{\Sigma,\underline{k}} \right)} \\
	\\
	{\bigotimes_{i \in I} \left( \Fun\left( \Op_{\gogl}(\overline{\sigma}^*_i)\right)/\calJ^{\Op,t}_{\sigma_i,k_i} \right)} & {\bigotimes_{i \in I}\left( \calU_{\kappa_c}(\hat{\gog}_{\sigma_i})/\calI_{\sigma_i,k_i} \right)}
	\arrow[from=1-1, to=1-2]
	\arrow[hook, from=1-1, to=3-1]
	\arrow[from=1-2, to=1-3]
	\arrow[hook, from=1-2, to=3-2]
	\arrow[hook, from=1-3, to=3-3]
	\arrow[from=3-1, to=3-2]
	\arrow[two heads, from=3-1, to=5-1]
	\arrow[from=3-2, to=3-3]
	\arrow[two heads, from=3-2, to=5-2]
	\arrow[from=5-1, to=5-2]
	\arrow["{=}"{description}, no head, from=5-2, to=3-3]
\end{tikzcd}\]
	The top left square commutes by the factorization properties of $\Psi_{\Sigma}$, the bottom left square commutes by Lemma \ref{cor:comparisontopologiesonesection} (which treats the case of a single section), the top right square commutes because $\calI_{\Sigma,\underline{k}}$ is an open ideal, and the bottom right triangle commutes by Lemma \ref{lem:behaviormultiindexideals}. 
	
	By construction $\calJ^{\Op,t}_{\Sigma,\underline{k}}$ maps to $(j_{I/I})_*\hat{j}^*_{I/I}\calJ^{\Op,t}_{\Sigma,\underline{k}}$ along the first left vertical map, while the latter is the kernel of the second left vertical map, by Lemma \ref{lem:behaviormultiindexideals}. It follows that $\calJ^{\Op,t}_{\Sigma,\underline{k}}$ is sent to $0$ along the composition of the left vertical maps. The statement of the Corollary follows.
\end{proof}

\begin{theorem}\label{thm:teofinale1}
	Assume that $S$ is integral, noetherian and quasi-separated. Let $X \to S$ be a smooth family of curves and $\Sigma = \{ \sigma_j \}_{j \in J} : S \to X^J$ be a collection of sections. Assume that for any surjection $J \twoheadrightarrow I$ the closed subscheme $V_{J/I} \subset S$ (c.f. Definition \ref{def:openclosedfactorization}) is integral. Let
	\[
		\Psi_\Sigma : \Fun\left( \Op_{\gogl}(\overline{\Sigma}^*) \right) \to Z_{\kappa_c}(\hat{\gog}_\Sigma)
	\] 
	be the composition $\Phi^{\zeta}_{\Sigma}\circ \gamma_{\Sigma}^{-1}$, of the map $\Phi^{\zeta}_{\Sigma}$ of Corollary \ref{coro:morphismbtwcenter} and the inverse of the map $\gamma_{\Sigma}$ of Theorem \ref{thm:feiginfrenkelcasarinmaffei}. Then 
	\begin{enumerate}
		\item $\Psi_{\Sigma}$ is a bicontinuous isomorphism  which respects the factorization structure on both spaces. In particular, $Z_{\kappa_c}(\hat{\gog}_\Sigma)$ is a QCCF sheaf and the collection $Z_{\kappa_c}(\hat{\gog}_\Sigma)$ is a complete topological factorization algebra;
		\item In addition, if $S$ is affine, well covered, equipped with a coordinate $t$, the morphism $\Psi_{\Sigma}$ identifies $\calJ^{\Op,t}_{\Sigma,\underline{k}}$ with $\calI^Z_{\Sigma,\underline{k}}$ (c.f. Definition \ref{def:topologicalideals}).
	\end{enumerate}
\end{theorem}

\begin{proof}
		We prove point $1$ and point $2$ looking at the morphism induced by $\Psi_\Sigma$ among the module of sections of the two sheaves over an arbitrary affine and well covered subset. This implies all claims of the Theorem. We can therefore assume that $S =\Spec A$ and $\pbarO$ is equipped with a coordinate $t \in \pbarO(S)$. 
		
		Let us notice at the beginning that since $\Psi_\Sigma$ is obtained as a composition of morphisms compatible with the (pseudo) factorization structures (see Proposition \ref{prop:factpropertiescalu}) it is itself compatible with the (pseudo) factorization structures. 
		
		We prove both results (1 and 2) together and by induction on $n$, the number of distinct sections which $\Sigma$ is composed of. If $n = 1$ the statement is exactly Proposition \ref{cor:feiginfrenkel1section} together with Lemma \ref{cor:comparisontopologiesonesection}. So suppose that the result holds for every collection of $n'\leq n$ different sections of $X \to S$, and let $\Sigma$ consists of $n+1$ different sections. We write $\Sigma = \{\sigma_1,\dots,\sigma_{n+1} \}$ Recall that by \cite[Section \ref{ssec:descrizionelocale1}]{casmaffei1} (see also Section \ref{ssec:richiamigeometricsetting}), since $S$ is affine and well covered, for any two given indices $l_1,l_2 \leq n+1$, the locus $\{\sigma_{l_1} = \sigma_{l_2}\} \subset S$ is either empty or principal: $\{ \sigma_{l_1} = \sigma_{l_2} \} = \{ f_{l_1,l_2} = 0 \}$ for some $f_{l_1,l_2} \in A$ (in particular for $\sigma_{l_1}(S)\cap\sigma_{l_2}(S) = \emptyset$ we can set $f_{l_1,l_2} = 1$). We can assume without loss of generality that all $f_{l_1,l_2}$ are non zero (if not $\sigma_{l_1} = \sigma_{l_2}$ and we are done by the inductive step) and therefore a nonzero divisors in $A$. 
		
		Consider an arbitrary non-zero divisor $f \in A$. Notice that the quotients $\Fun(\Op_{\gogl}(\oSigma^*))/\calJ^{\Op,t}_{\Sigma,\underline{k}}$ are without $f$-torsion because they are free, while $Z_{\kappa_c}(\hat{\gog}_\Sigma)/\calI^Z_{\Sigma,\underline{k}}$ is without $f$-torsion since it embeds in $\calU_{\kappa_c}(\hat{\gog}_\Sigma)/\calI_{\Sigma,\underline{k}}$, which is free. In order to get our induction going we therefore repeatedly apply Lemma \ref{lem:isocriteria} (with $M = \Fun(\Op_{\gogl}(\oSigma^*))$, $N= Z_{\kappa_c}(\hat{\gog}_\Sigma)$, $M_{\underline{k}} = \calJ^{\Op,t}_{\Sigma,\underline{k}}$ and $ N_{\underline{k}} = \calI^Z_{\Sigma,\underline{k}}$) to reduce to the case where all sections are disjoint. Notice that the Lemma applies in this case as well, even if we are indexing our systems of neighborhoods with multi-indices $\underline{k}$ instead of natural numbers. This follows from the fact that the former system has a cofinal system isomorphic to $\mN$.
		
		In order to do this we proceed as follows: write $J = \{ 1, \dots, n+1 \}$, for $l\neq h \in J$ let $U_{lh} = \Spec A_{f_{lh}},V_{lh} = \Spec A/(f_{lh})$ be the loci where $\sigma_l \neq \sigma_h$ and $\sigma_l = \sigma_h$ respectively, denote by $j_{lh} : U_{lh} \to S, i_{lh} : V_{lh} \to S$ the corresponding immersions. Notice that on each $V_{lh}$, since $\sigma_l = \sigma_h$ on it, $\Sigma_{|V_{lh}}$ is composed of at most $n$ different section; we claim that $\hat{i}^*_{lh}\Psi_\Sigma$ is an isomorphism as well which identifies with $\Psi_{\Sigma_{i=j}}$ via the morphisms $\Ran$ of the factorization structures. In order to prove this, consider the commutative diagram induced by the factorization structures 
		
		\[\begin{tikzcd}
	{\hat{i}^*_{lh}\Fun\left( \Op_{\gogl}(\overline{\Sigma}^* )\right)} && {\Fun\left( \Op_{\gogl}(\overline{\Sigma}^*_{i=j})\right)} \\
	\\
	{\hat{i}_{lh}^*Z_{\kappa_c}(\hat{\gog}_{\Sigma})} && {Z_{\kappa_c}(\hat{\gog}_{\Sigma_{l=h}})}
	\arrow["\Ran^{\Op}", from=1-1, to=1-3]
	\arrow["{\hat{i}_{lh}^*\Psi_{\Sigma}}"', from=1-1, to=3-1]
	\arrow["{\Psi_{\Sigma_{l=h}}}", from=1-3, to=3-3]
	\arrow["\Ran^Z"', from=3-1, to=3-3]
\end{tikzcd}\]
		
	The top horizontal map is an isomorphism by the factorization properties of $\Fun(\Op_{\gogl}(\overline{\Sigma}^*))$, the right vertical arrow is an isomorphism by the inductive assumption. By Lemma \ref{lem:isocriterialemma} $\hat{i}^*_{lh}Z_{\kappa_c}(\hat{\gog}_{\Sigma}) = Z_{\kappa_c}(\hat{\gog}_{\Sigma})/(f_{lh})Z_{\kappa_c}(\hat{\gog}_{\Sigma})$. Notice that the center is $f_{lh}$-saturated ($(f_{lh})z \in Z$ implies $z \in Z$) and therefore the bottom horizontal map is injective. It follows from the fact that $\Psi_{\Sigma_{l=h} \circ \Ran^{\Op}}$ is a topological isomorphism that $\hat{i}^*_{lh}\Psi_{\Sigma}$ and $\Ran^Z$ are topological isomorphisms as well, and that $\hat{i}^*_{lh}\Psi_{\Sigma}$ identifies with $\Psi_{\Sigma_{l=k}}$. 
	
	Next, we claim that $\hat{i}^*_{lh}\Psi_{\Sigma}$ induces an isomorphism $\hat{i}^*_{lh}\calJ^{\Op,t}_{\Sigma,\underline{k}} \to \hat{i}^*_{lh}\calI^Z_{\Sigma,\underline{k}}$ (this is the last hypothesis we need to apply Lemma \ref{lem:isocriteria}). Notice that since $Z_{\kappa_c}(\hat{\gog}_{\Sigma})$ and the quotients $Z_{\kappa_c}(\hat{\gog}_{\Sigma})/\calI^Z_{\Sigma,\underline{k}}$ have no $f_{lh}$ torsion, the map $\hat{i}^*_{lh}\calI^Z_{\Sigma,\underline{k}} \to \hat{i}^*_{lh}Z_{\kappa_c}(\hat{\gog}_\Sigma) = Z_{\kappa_c}(\hat{\gog}_{\Sigma_{l=h}})$ is injective, from this it follows that $\hat{i}^*_{l=h}\calI^Z_{\Sigma,\underline{k}} \subset \calI^Z_{\Sigma_{l=h},\underline{k}_{l=h}}$. Thus, using Lemma \ref{lem:behaviormultiindexideals}, we get a morphism $$\calJ^{\Op,t}_{\Sigma_{l=h},\underline{k}_{l=h}} = \hat{i}^*_{lh}\calJ^{\Op,t}_{\Sigma,\underline{k}} \to \hat{i}^*_{lh}\calI^Z_{\Sigma,\underline{k}} \subset \calI^Z_{\Sigma_{l=h},\underline{k}_{l=h}}$$ which, by the inductive hypothesis, is an isomorphism; it follows that $\hat{i}^*_{lh}\calJ^{\Op,t}_{\Sigma,\underline{k}} \to \hat{i}^*_{lh}\calI^Z_{\Sigma,\underline{k}}$ is an isomorphism as well.
		
	It follows that we are in the setting of Lemma \ref{lem:isocriteria}, so that we reduce to prove that, for example, $\hat{j}_{1,2}^*\Psi$ is an isomorphism. With a similar argument, and applying again Lemma \ref{lem:isocriteria}, we see that it is enough to prove that $\Psi$ is an isomorphism when restricted to $U_{1,2}\cap U_{1,3}$, since it is an isomorphism when restricted to $U_{1,2}\cap V_{1,3}$. We can do this for all possible different indices thanks to the hypothesis on the subschemes $V_{J/I} \subset S$, which, being integral, allow us to apply Lemma \ref{lem:isocriteria}, since the various $f_{lh}$ (if not $0$) remain non zero-divisors also in $V_{J/I}$. 
    We reduce in this way to the case where all sections are disjoint which we assume from now on. Repeating the above reasoning replacing $\Fun(\Op_{\gogl}(\oSigma^*))$ with $\calJ^{\Op,t}_{\Sigma,\underline{k}}$ and $Z_{\kappa_c}(\hat{\gog}_\Sigma)$ with $\calI^Z_{\Sigma,\underline{k}}$ we can prove point $2$ as well. The last ingredient that we need is that in the case where all sections are disjoint $\Psi_{\Sigma}$ induces an isomorphism between $\calJ^{\Op,t}_{\Sigma,\underline{k}}$ and $\calI^Z_{\Sigma,\underline{k}}$.
	
	So to prove point $1$, we need to show that $\Psi_{\Sigma}$ is an isomorphism in the case where all sections are disjoint. In this case, by Proposition \ref{prop:factpropertiescalu}, looking at global sections we have a commutative diagram
	\[\begin{tikzcd}
	{\Fun\left( \Op_{\gogl}(\overline{\Sigma}^*) \right)(S)} && {Z_{\kappa_c}(\hat{\gog}_\Sigma)(S)} \\
	\\
	{\bigotimes^!_{j \in J}\left(\Fun\left( \Op_{\gogl}(\overline{\sigma}_j^*) \right)(S) \right)} && { \bigotimes^!_{j \in J} \left(Z_{\kappa_c}(\hat{\gog}_{\sigma_j})(S)\right)}
	\arrow["{\Psi_\Sigma}", from=1-1, to=1-3]
	\arrow["{\fact}", from=3-1, to=1-1]
	\arrow["{\fact^Z}"', from=3-3, to=1-3]
	\arrow["{\otimes^!_{j\in J}\Psi_{\sigma_j}}"', from=3-1, to=3-3]
\end{tikzcd}\]
	We know that the lower horizontal arrow is an isomorphism by the $1$ section case, by Proposition \ref{prop:factpropertiescalu} the left vertical arrow is an isomorphism as well. Recall that that by Proposition \ref{prop:factorizationcenter} the right vertical map is the restriction of $\left( \otimes^!_{j \in J} \calU_{\kappa_c}(\hat{\gog}_{\sigma_j}) \right) \to \calU_{\kappa_c}(\hat{\gog}_{\Sigma})$ which is an isomorphism by the factorization properties of $\calU_{\kappa_c}(\hat{\gog}_\Sigma)$ (see \cite[Proposition \ref{prop:factpropertiesgogcalugog}]{casmaffei1}). Since by the $1$-section case we know $Z_{\kappa_c}(\hat{\gog}_{\sigma_j})$ is QCCF, so that we can apply Lemma \ref{lem:isoprodcentro} and deduce that $\fact^Z$ is an isomorphism. It follows that $\Psi_{\Sigma}$ is an isomorphism as well. 
	
	In the same way, to prove point $2$ we are reduced to the case where all sections are disjoint. Recall that by Lemma \ref{cor:comparisontopologiesonesection},we have that $\Psi_{\sigma_j}$ identifies $\calJ^{\Op,t}_{\sigma_j,k}$ with $\calI^Z_{\sigma_j,k}$ and by Lemma \ref{lem:behaviormultiindexideals} we have exact sequences 
	\[\begin{tikzcd}
	0 & {\calJ^{\Op,t}_{\Sigma,\underline{k}}} & {\bigotimes^!_{j\in J}\Fun\left(\Op_{\gogl}(\overline{\sigma_j}^*)\right)} & {\bigotimes_{j \in J}\left( \Fun(\Op_{\gogl}(\overline{\sigma_j}^*))/\calJ^{\Op,t}_{\sigma_j,k_j} \right)} & 0 \\
	\\
	0 & {\calI^Z_{\Sigma,\underline{k}}} & {\bigotimes^!_{j \in J}Z_{\kappa_c}(\hat{\gog}_{\sigma_j})} & {\bigotimes_{j \in J}\left( Z_{\kappa_c}(\hat{\gog}_{\sigma_j})/\calI^Z_{\sigma_j,k_j}\right)} & 0
	\arrow[from=1-1, to=1-2]
	\arrow[from=1-2, to=1-3]
	\arrow[from=1-2, to=3-2]
	\arrow[from=1-3, to=1-4]
	\arrow[from=1-3, to=3-3]
	\arrow[from=1-4, to=1-5]
	\arrow[from=1-4, to=3-4]
	\arrow[from=3-1, to=3-2]
	\arrow[from=3-2, to=3-3]
	\arrow[from=3-3, to=3-4]
	\arrow[from=3-4, to=3-5]
\end{tikzcd}\]
	notice that these are also exact at the level of sections since we are dealing with \QCC sheaves on an affine scheme. The middle vertical map and the right vertical map are isomorphisms by the one section case (see Proposition \ref{cor:feiginfrenkel1section} and Lemma \ref{cor:comparisontopologiesonesection}); it follows that the left vertical map is an isomorphism as well. Since this holds for any choice of $\underline{k}$ it is a topological isomorphism as well so that $\Psi_{\Sigma}$ induces a topological isomorphism between $\calJ^{\Op,t}_{\Sigma,\underline{k}}$ and $\calI^Z_{\Sigma,\underline{k}}$ as claimed.
\end{proof}

\subsection{The factorizable Feigin-Frenkel center}\label{ssec:thefactorizablefeiginfrenkelcenter}

We apply the results obtained so far to prove Theorem \ref{thm:mainteointro}. Let $C$ be a smooth curve over $\mC$. We specialize our results to the universal families $X_J = C\times C^J$ over $S_J = C^J$ introduced in Remark \ref{rmk:factsigmaugualefactcurva}. We take as our $\Sigma$ the $J$-collection of sections:

\[
	\sigma^{\univ}_{C,J,j} : S_J = C^J \to C\times C^J = X_J \qquad (x_{j'})_{j' \in J} \mapsto (x_j, (x_{j'})_{j' \in J}).
\]
we denote this collection by $\Sigma^{\univ}_{C,J}$.

Definition \ref{def:constructiongogsigma} provides us with a sheaf of Lie algebras (in the $\Sigma$ notation $\hat{\gog}_{\Sigma,\kappa_c}$) on $C^J$ which we denote by $\hat{\gog}_{C^J,\kappa_c}$. Let $Z_{\kappa_c}(\hat{\gog}_{C^J})$ the center of the sheaf of complete enveloping algebras of $\hat{\gog}_{C^J,\kappa_c}$. Recall as well the space of opers $\Op_{\gogl}(D^*)_{C^J}$, defined in Definition \ref{def:opersCI}.

\begin{theorem}\label{thm:teofinalefinale}
	For any finite set $J$ there is a canonical isomorphism
	\[
		Z_{\kappa_c}(\hat{\gog}_{C^J}) = \Fun \left( \Op_{\gogl}(D^*)_{C^J} \right)
	\]
	which is compatible with the factorization structure on both spaces.
\end{theorem}

\begin{proof}
	Recall that by Remark \ref{rmk:operCopersigma} in this case $\Op_{\gogl}(\overline{\Sigma}^{\univ}_{C,J})$ coincides with the space $\Op_{\gogl}(D^*)_{C^J}$. Notice in addition that the family $X_J \to S_J$ satisfies the hypothesis of Theorem \ref{thm:teofinale1} since, for any surjection $J \twoheadrightarrow I$, the closed subscheme $V_{J/I} \subset S_J = C^J$ identifies with the diagonal embedding $\Delta(J/I) : C^I \to C^J$. Then for any finite set $J$ we get a canonical isomorphism of QCC sheaves on $C^J$
    \[
        Z_{\kappa_c}\left(\hat{\gog}_{C^J}\right) = \Fun\left( \Op_{\check{\gog}}(D^*)_{C^J} \right),
    \]
    so that we only need to show that it compatible with the factorization structures. This is already part of the statement of Theorem \ref{thm:teofinale1}, but we need to translate our $X,S,\Sigma$ factorization notion to the usual one. So fix a surjection $J\twoheadrightarrow I$, recall Notation \ref{sssez:notazionidiagonali}, Definitions \ref{def:factorizationspace}, \ref{def:openclosedfactorization}, and Remark \ref{rmk:factsigmaugualefactcurva} and notice that the following:
    \begin{itemize}
        \item The closed subscheme $V_{J/I} \subset S_J = C^J$ corresponds to the diagonal immersion $\Delta(J/I) : C^I \subset C^J$ induced by $J \twoheadrightarrow I$. The collection of $I$ sections induced by $\Sigma^{\univ}_{C,J}$ is exactly $\Sigma^{\univ}_{C,I}$;
        \item The open subset $U_{J/I} \subset S_J$ corresponds to the complement of $\nabla(J/I) \subset C^J$ so that the two different notions of $U_{J/I}$ of notations \ref{sssez:art2notazionidiagonali} and \ref{ntz:factorization} actually coincide in this case. For any $i \in I$ the collection of sections $(\Sigma^{\univ}_{C,J})_{J_i} \subset \Sigma^{\univ}_{C,J}$ induced by $J_i \subset J$ corresponds to the pullback along $U_{J/I} \subset C^J \to C^{J_i}$ of the canonical set of sections $\Sigma^{\univ}_{C,J_i}$;
    \end{itemize}
    After this remark, the only difference between our notion of $X,S,\Sigma$ factorization structure and the one of Definition \ref{def:factorizationspace} concerns the maps $\fact$. Indeed, with respect with the $X,S,\Sigma$ factorization picture the factorization maps that we get are of the form
    \begin{align*}
        \fact^{\Op}_{J/I} &: \hat{j}^*_{J/I} \bigotimes^!_{i \in I} \Fun\left( \Op_{\check{\gog}}\left((\Sigma^{\univ}_{C,J})_{J_i}\right) \right) \to \hat{j}^*_{J/I} \Fun\left( \Op_{\check{\gog}}(D^*)_{C^I} \right) \\
        \fact^Z_{J/I} &: \hat{j}^*_{J/I} \bigotimes^!_{i \in I} Z_{\kappa_c}(\hat{\gog}_{(\Sigma^{\univ}_{C,J})_{J_i}}) \to \hat{j}^*_{J/I} Z_{\kappa_c}(\hat{\gog}_{C^I}) 
    \end{align*}
    while we are actually looking for maps of the form
    \begin{align*}
        \fact^{\Op}_{J/I} &: \hat{j}^*_{J/I} \boxtimes^!_{i \in I} \Fun\left( \Op_{\check{\gog}}(D^*)_{C^{J_i}} \right) \to \hat{j}^*_{J/I} \Fun\left( \Op_{\check{\gog}}(D^*)_{C^I} \right) \\
        \fact^Z_{J/I} &: \hat{j}^*_{J/I} \boxtimes^!_{i \in I} Z_{\kappa_c}(\hat{\gog}_{C^{J_i}}) \to \hat{j}^*_{J/I} Z_{\kappa_c}(\hat{\gog}_{C^I}) 
    \end{align*}
    so that the Theorem follows if we prove that, calling $\pi_i : C^J \to C^{J_i}$ the canonical projection, there are canonical isomorphisms
    \begin{align*}
        \hat{j}^*_{J/I} \Fun\left( \Op_{\check{\gog}}\left((\Sigma^{\univ}_{C,J})_{J_i}\right) \right) &= \hat{j}^*_{J/I}\hat{\pi}_i^*\Fun\left( \Op_{\check{\gog}}(D^*)_{C^{J_i}} \right), \\
        \hat{j}^*_{J/I}Z_{\kappa_c}(\hat{\gog}_{(\Sigma^{\univ}_{C,J})_{J_i}}) &= \hat{j}^*_{J/I}\hat{\pi}_i^*Z_{\kappa_c}(\hat{\gog}_{C^{J_i}}).
    \end{align*}
    The first claim easily follows by definition of the spaces $\Op_{\check{\gog}}\left((\Sigma^{\univ}_{C,J})_{J_i}\right)$ and $\Op_{\check{\gog}}(D^*)_{C^{J_i}}$. We claim that at the level of spaces we have $\Op_{\check{\gog}}\left((\Sigma^{\univ}_{C,J})_{J_i}\right)= \pi_i^ * \Op_{\gogl}(D^*)_{C^{J_i}}$, so that the dual statement on their algebra of functions immediately follows. Both of these are functors $\Aff^{\op}_{C^J} \to \Set$. The first one takes $\underline{r} = (r_j)_{j \in J} \in \Hom(\Spec R,C^J)$ to the set of isomorphism classes of opers over $\overline{\Sigma}^*_{J_i,\underline{r}}$, where the latter is constructed using the sections $r_j : \Spec R \to C \times \Spec R$ for $j \in J_i$. The same description is easily seen to hold for $\pi^*_i\Op_{\gogl}(D^*)_{C^{J_i}}$, since we first forget all $r_j$ except from those for $j \in J_i$ and then compute opers on $\overline{\Sigma}^*_{J_i,\underline{r}}$.
	
    As the second statement is concerned, we prove that $Z_{\kappa_c}(\hat{\gog}_{(\Sigma^{\univ}_{C,J})_{J_i}}) = \hat{\pi}_i^*Z_{\kappa_c}(\hat{\gog}_{C^{J_i}})$ as well.
    Notice that unravelling the definitions and by the pullback properties of $\calU_{\kappa_c}(\hat{\gog}_\Sigma)$ (see \cite[Proposition \ref{prop:pullbackgu}]{casmaffei1}) we have
    \[
        Z_{\kappa_c}(\hat{\gog}_{(\Sigma^{\univ}_{C,J})_{J_i}}) = Z\left( \hat{\pi}_i^*\calU_{\kappa_c}(\hat{\gog}_{\Sigma^{\univ}_{C,J_i}}) \right).
    \]
    Hence, we need to prove compatibility of the center under pullback. By Remark \ref{rmk:centerandpullback}, there is a natural morphism $\hat{\pi}^*_i Z_{\kappa_c}(\hat{\gog}_{C^{J_i}}) \to Z\left( \hat{\pi}_i^*\calU_{\kappa_c}(\hat{\gog}_{\Sigma^{\univ}_{C,J_i}}) \right)$ and we need to check that it is an isomorphism. We can assume that $C=\Spec R$ is affine. Then $C^J \simeq \Spec R^{\otimes J}$, $C^{J_i} =\Spec R^{\otimes J_i}$ so that the claim follows by Lemma \ref{lem:isoprodcentro}, applied to the case $A = R^{\otimes J_i}$, $\calB_1= \calU_{\kappa_c}(\hat{\gog}_{\Sigma^{\univ}_{C,J_i}})$ and $\calB_2 = R^{\otimes J}$.
\end{proof}

\appendix

\section{The Feigin-Frenkel center}\label{sec:feiginfrenkelclassical}

In this appendix we elaborate on the classical Feigin-Frenkel center, formulate it in its coordinate independent way and finally relate it to our constructions. 

\subsection{A quick overview}

Let us start by recalling the original statement of the Feigin-Frenkel Theorem, first in its vertex algebra version and then in its enveloping algebra version.

\begin{theorem}[\cite{feigin1992affine}]\label{thm:feiginfrenkelva}
	Assume that $\gog$ is a simple Lie algebra, let $\hat{\gog}_{\kappa_c}$ be the affine Lie algebra at the critical level and let $V^{\kappa_c}(\gog)$ be its attached affine vertex algebra. Then there is an $(\Aut O(\mC),\Der O)$-equivariant isomorphism
	\[
		\zeta(\hat{\gog}) \stackrel{\mathrm{def}}{=} \zeta\left(V^{\kappa_c}(\gog)\right) \simeq \mC[\Op_{\gogl}(D)],
	\]\index{$\zeta(\hat{\gog})$}
	where $\zeta\left(V^{\kappa_c}(\gog)\right)$ denotes the center of the vertex algebra.
\end{theorem}

\begin{theorem}[\cite{feigin1992affine}]\label{thm:feiginfrenkel}
	Assume that $\gog$ is a simple Lie algebra, let $\hat{\gog}_{\kappa_c}$ be the affine Lie algebra at the critical level and let $U_{\kappa_c}(\hat{\gog})$ be its completed enveloping algebra. There is an $(\Aut O (\mC),\Der O)$-equivariant isomorphism 
	\[
		Z_{\kappa_c}(\hat{\gog}) \stackrel{\mathrm{def}}{=} Z(U_{\kappa_c}(\hat{\gog})) = \mC[\Op_{\gogl}(D^*)],
	\]
	where $Z(U_{\kappa_c}(\hat{\gog}))$ is the center of the completed enveloping algebra.
\end{theorem}

The present exposition is organized with the goal of showing how to extend the original Feigin-Frenkel theorem to our $X,S,\Sigma$ setting in the case where $\Sigma$ consists of one single section. This is the content of Proposition \ref{cor:feiginfrenkel1section}; its proof is essentially a reformulation of the original proof of the Feigin-Frenkel Theorem in a slightly more general setting. In the following Remark we give a very brief overview of the proof of Feigin and Frenkel and then we expand it in the following Sections, comparing the constructions of Feigin and Frenkel with ours.

\begin{ntz}\label{ntz:changeffnotazione}
	In \cite[Section 3.2]{frenkel2007langlands} the authors construct, attached to a vertex algebra $V$, a Lie algebra $U(V)$ and a complete associative algebra $\tilde{U}(V)$. In order to better compare our constructions with theirs we take the freedom to change the notation and write
	\[
		\Lie_{\mathrm{FF}}(V) = U(V) \qquad \tilde{U}_{\mathrm{FF}}(V) = \tilde{U}(V).
	\]
\end{ntz}

\begin{remark}\label{rmk:classicalfeifreiso} 
The proof of Theorem \ref{thm:feiginfrenkel}, as described in \cite{frenkel2007langlands}, goes as follows (here we will use the notation contained in \emph{loc. cit.} with the changes of Notation \ref{ntz:changeffnotazione}).
\begin{enumerate}
    \item First one proves that the center of the vertex algebra $\zeta(\hat{\gog}) := \zeta(V^{\kappa_c}(\gog))$ is isomorphic to $\mC[\Op_{\gogl}(D)]$ in an $\Aut O(\mC)$-equivariant way. In particular, there is an $\Aut O(\mC)$-equivariant immersion $\mC[\Op_{\gogl}(D)] \subset V^{\kappa_c}(\gog)$;
    \item One moves on by considering the complete associative algebra $\tilde{U}_{\mathrm{FF}}(\mC[\Op_{\gogl}(D)])$, which parallels our $\calU_{\overline{\Sigma}^*}$ of Section \ref{sec:envelopingalgchiralalg} in the case where $S = \mC$ and $\Sigma$ consists of a single section (see Section \ref{sssec:comparisonconstructionsfrenkelcasmaf}). This is constructed as a modification of the completed enveloping algebra of the Lie algebra (see \cite[3.2.1]{frenkel2007langlands} and Notation \ref{ntz:changeffnotazione}) \[\Lie_{\mathrm{FF}}\left(\mC[\Op_{\gogl}(D)]\right) =  \mC[\Op_{\gogl}(D)] \otimes \mC((z)) /\text{Im}(T + \partial_z).\] 
    $\Lie_{\mathrm{FF}}\left(\mC[\Op_{\gogl}(D)]\right)$ comes equipped with a natural $\Aut O$ action on which we elaborate in Section \ref{ssec:comparisonautoactions}; There is a natural isomorphism $$\Phi_{\text{FF}} : \tilde{U}_{\mathrm{FF}}(V^{\kappa_c}(\gog)) \to \tilde{U}_{\kappa_c}(\hat{\gog})$$ with the completed enveloping algebra of $\hat{\gog}_{\kappa_c}$ which is automatically $\Aut O(\mC)$-equivariant. The latter comes from a canonical morphism $\Lie_{\mathrm{FF}}(V^{\kappa_c}(\gog)) \to {U}_{\kappa_c}(\hat{\gog})$ of Lie algebras (see \cite[3.2.2]{frenkel2007langlands} and Remark \ref{rmk:comparisonmaplieenveloping}). By functoriality of $\tilde{U}_{\mathrm{FF}}$ we may restrict the above morphism to a map $$\tilde{U}_{\mathrm{FF}}(\mC[\Op_{\gogl}(D)]) \simeq \tilde{U}_{\mathrm{FF}}(\zeta(\hat{\gog})) \to \tilde{U}_{\kappa_c}(\hat{\gog}),$$ which is easily checked to have image contained in the center $Z(\hat{\gog}) = Z(\tilde{U}_{\kappa_c}(\hat{\gog}))$ (this follows from the fact that the elements in $\mC[\Op_{\gogl}(D)]$ are all central in $V^{\kappa_c}(\gog)$). In particular one gets a morphism
	\[
		\Phi^\zeta_{\text{FF}} : \tilde{U}_{\mathrm{FF}}(\mC[\Op_{\gogl}(D)]) \to Z(\hat{\gog});
	\]
    \item The proof goes on following two different steps. Step $(a)$ is showing that the map $\Phi^\zeta_{\text{FF}} : \tilde{U}_{\mathrm{FF}}(\zeta(\hat{\gog})) \to Z(\hat{\gog})$ is an isomorphism. This is done by considering appropriate quotients of these complete algebras, some appropriate filtrations on them, and their associated graded spaces, using a bit of invariant theory;
    \item Step $(b)$ is concerned with establishing an $\Aut O(\mC)$-equivariant isomorphism \[\gamma_{\mathrm{FF}} : \tilde{U}_{\mathrm{FF}}(\mC[\Op_{\gogl}(D)]) = \mC[\Op_{\gogl}(D^*)];\]
    \item By considering $\Phi^\zeta_{\text{FF}}\gamma_{\mathrm{FF}}^{-1}$ one obtains an $\Aut O(\mC)$-equivariant isomorphism $\mC[\Op_{\gogl}(D^*)] = Z_{\kappa_c}(\hat{\gog})$. This is what we will be calling \emph{the Feigin-Frenkel isomorphism}.
\end{enumerate}
\end{remark}

\subsection{Comparison between our constructions with those of \texorpdfstring{\cite{frenkel2007langlands}}{Frenkel's "Langlands Correspondence for Loop Groups"}}\label{sssec:comparisonconstructionsfrenkelcasmaf}

Let us be more precise on how the constructions of the Lie algebra (recall Notation \ref{ntz:changeffnotazione}) $\Lie_{\mathrm{FF}}(V)$ \cite[3.2.2]{frenkel2007langlands} and our $\Lie_{\overline{\Sigma}^*}$ (see Sect. \ref{sec:envelopingalgchiralalg}), $\tilde{U}_{\mathrm{FF}}(V)$ of \cite[3.2.3]{frenkel2007langlands} and our $\calU_{\overline{\Sigma}^*}$ are related. Here, we are interested in the case where $S = \Spec \mC$ and where we have a privileged coordinate $z \in \pbarO$ (we call it $z$ instead of $t$ because $t$ will have a different role later on). We may therefore assume that $X = \mA^1_{\mC} = \Spec \mC[z]$ and that $\Sigma$ consists of the $0$ section. Topological sheaves on $S$ become $\mC$ topological vectors spaces, slightly abusing our notation, which we believe makes for a better understanding of the situation, we will write $\mC[[z]]$ in place of $\Sigma,\overline{\Sigma}$, so that for instance we will write $V^\kappa_{\mC[[z]]}(\gog)$ in place of the chiral algebra $\VkSg$.

Our geometric setting then boils down to study the space of fields $\mF^1 = \Homcont_{\mC}(\mC((z)),U_\kappa(\hat{\gog}))$, for $U_\kappa(\hat{\gog})$ the usual completed enveloping algebra of the affine algebra $\hat{\gog}_\kappa$. Our construction provide a canonically defined chiral algebra $V^\kappa_{\mC[[z]]}(\gog) \subset \Homcont_{\mC}(\mC((z)),U_\kappa(\hat{\gog}))$ over $\mC[[z]]$ and an isomorphism $\calY_z : V^\kappa(\gog) \otimes \mC[[z]] \to V^\kappa_{\mC[[z]]}(\gog) \subset \mF^1$, which is essentially a reformulation of the usual state/field correspondence (see Section \ref{ssec:recollectionsysigma}).
    
The output of the constructions of \cite{frenkel2007langlands} and our constructions is the same. The main difference between them is the kind of input: the former takes vertex algebras as input, while the latter takes chiral algebras over $\mC[[z]]$ as input. 

Recall that by \cite[Remark \ref{rmk:extvoaoslin}]{casmaffei1} to any vertex algebra $V$ there is a canonically attached chiral algebra $V \otimes {\mC[[z]]}$, with the right action of $\partial_z$ given by $(v\otimes f)\partial_z = -Tv \otimes f - v \otimes \partial_zf$. It is then evident that for any vertex algebra $V$ we have
\[
    \Lie_{\mathrm{FF}}\left( V \right) = \frac{V \otimes \mC((z))}{\text{Im}(T + \partial_z)} = h^0\left( V\otimes\mC((z)) \right) = \Lie_{\mC((z))}\left(V \otimes {\mC[[z]]}\right)
\]

We move on restricting to the case where $V = V^\kappa(\gog)$ comparing it with our canonically defined chiral algebra $V^\kappa_{\mC[[z]]}(\gog)$. As explained above, the choice of the coordinate $z$ induces an identification $\calY_z : V^\kappa(\gog) \otimes \mC[[z]] \simeq V^\kappa_{\mC[[z]]}(\gog)$ (c.f. Section \ref{ssec:richiamiVkSg} and \cite[Theorem \ref{teo:descrizionelocaleVkSg}]{casmaffei1}) so that we get an isomorphism
\begin{equation}\label{eq:isoliefrenkelcasmaffei}
    \Lie_{\mathrm{FF}}(V^\kappa(\gog)) = \Lie_{\mC((z))}\left( V^\kappa_{\mC[[z]]}(\gog)\right)
\end{equation}

\begin{remark}\label{rmk:comparisonmaplieenveloping}
    One can check that the canonical morphism of Lie algebras $\Lie_{\mathrm{FF}}(V^\kappa(\gog)) \to {U}_\kappa(\hat{\gog})$ of \cite[Prop. 3.2.1]{frenkel2007langlands} and \cite[Prop. 4.2.2]{frenkel2004vertex} identifies with our $\beta : \Lie_{\mC((z))}(V^\kappa_{\mC[[z]]}(\gog)) \to U_{\kappa}(\hat{\gog}_{\mC((z))})$ via the identification of Equation (\ref{eq:isoliefrenkelcasmaffei}).
\end{remark}

Having identified the Lie algebras $\Lie_{\mathrm{FF}}(V^\kappa(\gog))$ and $\Lie_{\mC((z))}\left( V^\kappa_{\mC[[z]]}(\gog) \right)$ we can move forward to identify the complete associative algebras $\tilde{U}_{\mathrm{FF}}(V^\kappa(\gog))$ and  $\calU_{\mC((z))}\left( V^\kappa_{\mC[[z]]}(\gog) \right)$. Both of these are obtained in steps:
\begin{enumerate}
    \item Taking the enveloping algebras of the Lie algebras $\Lie_{\mathrm{FF}}(V^\kappa(\gog)) = \Lie_{\mC((z))}V_{\mC[[z]]}^\kappa(\gog)$;
    \item Considering a suitable completion of the latter associative algebras. 
        \begin{enumerate}
        \item For the construction of $\tilde{U}_{\mathrm{FF}}$, in \cite{frenkel2007langlands}, the topology considered on the enveloping algebra of $\Lie_{\mathrm{FF}}(V^\kappa(\gog))$ is the one generated by the left ideals generated by  the subspaces $I_N = \{ [A\otimes t^n] \text{ such that } n \geq N + \deg A \}$, where $A \in V^\kappa(\gog)$, $\deg A$ is the degree with respect of the standard grading of $V^\kappa(\gog)$ while $[A\otimes t^n]$ denotes the class of $A \otimes t^n$ in $U(V^\kappa(\gog))$; 
        \item For our construction $\calU_{\mC((z))}$ we consider the subspace topology of $V^\kappa(\gog) \otimes \mC[[z]] \subset \mF^1$;
        \end{enumerate}
    These topologies are equivalent; thus our completed associative algebra $\calU^1_{\mC((z))}\left(V^\kappa_{\mC[[z]]}(\gog)\right)$ (see Section \ref{ssec:envelopingalgchiralalg}) identifies with Frenkel's $\tilde{U}(\Lie_{\mathrm{FF}}(V^\kappa(\gog)))$;
    \item Factoring out some relations
        \begin{enumerate}
            \item In the construction of $\tilde{U}_{\mathrm{FF}}$, as in \cite[3.2.3]{frenkel2007langlands}, one factors out the relations given by the Fourier coefficients of the series
            \[
                Y[A_{(-1)}B,z] - :Y[A,z]Y[B,z]:
            \]
            we refer to \emph{loc. cit.} for the notation $Y[A,z]$ and for the normally ordered product. In \emph{loc. cit.} it is also claimed that factoring out by these relations preserves completeness of the enveloping algebra;
            \item For our construction $\calU_{\mC((z))}$ we factor out $\calU^1_{\mC((z))}$ by the relations coming from diagram \eqref{eq:diagramcoeq}. Following our notation it is possible to recover the formal power series $Y[A,z]$ as the field $\Psi(A)$ (c.f. Section \ref{ssec:envelopingalgchiralalg}). This follows by embedding $$\Homcont(\mC((z)),\tilde{U}(\Lie_{\mathrm{FF}}(V^\kappa(\gog)))) \subset \tilde{U}(\Lie_{\mathrm{FF}}(V^\kappa(\gog)))[[z^{\pm 1}]]$$ (via the map $\varphi \mapsto \sum_{n \in \mZ} \varphi(z^n)z^{-n-1}$) and then noticing that $Y[A,z]$ and $\Psi(A)$ are defined by the same inductive formulas.
			
			Frenkel's relations then boil down as 
            \[
                \ev_1 \circ \left( \Delta_!\Psi\circ\mu_{\calV_{\overline{\Sigma}^*}} - \mu_{\calW}\circ(\Psi\boxtimes\Psi)\right) \left(A \boxtimes B \frac{1}{t_1 - t_2}\right).
            \]
			This is just a matter of unravelling the construction of $\Psi$ and relate it to Feigin and Frenkel's constructions.
        \end{enumerate}
\end{enumerate}

\noindent It follows that there is a natural, continuous morphism of complete associative algebras
\begin{equation}\label{eq:morphenvalgsvertexchiral}
    \tilde{U}_{\mathrm{FF}}\left(V^\kappa(\gog)\right) \to \calU_{\mC((z))}\left( V^\kappa_{\mC[[z]]}(\gog)\right)
\end{equation}
and that post-composing with our $U(\beta)=\Phi : \calU_{\mC((z))}\left( V^\kappa_{\mC[[z]]}(\gog)\right) \to U_{\kappa}(\hat{\gog})$ we recover the canonical morphism of \cite[Lemma 3.2.2]{frenkel2007langlands}. Since the latter is a topological isomorphism, as our $U(\beta)$, it follows that the canonical map \eqref{eq:morphenvalgsvertexchiral} is a topological isomorphism, so that the constructions $\tilde{U}_{\mathrm{FF}}$ and $\calU_{\mC((z))}$ may be identified in this case. This implies the following remark as well. 

\begin{remark}\label{rmk:comparisonutildePHI}
    The canonical isomorphism of complete associative algebras $$\tilde{U}_{\mathrm{FF}}(V^\kappa(\gog)) \simeq \calU_{\mC((z))}\left( V^\kappa_{\mC[[z]]}(\gog) \right)$$ restricts to a canonical identification between $\tilde{U}_{\mathrm{FF}}(\mC[\Op_{\gogl}(D)])$ and $\calU_{\mC((z))}\left( \zeta^{\kappa_c}_{\mC[[z]]}(\gog) \right)$. Under this identification, the morphism $\Phi$ of \ref{coro:morphismbtwcenter} coincides with the morphism $\Phi_{\text{FF}}$ considered in \cite[4.3.2]{frenkel2007langlands}. This can be checked by choosing a coordinate and using the description $\tilde{U}_{\mathrm{FF}}(\mC[\Op_{\gogl}(D)]) \simeq \overline{\Sym}_{\mC}((\check{V}^{\can})^* \otimes \mC((t)))\simeq \calU_{\mC((z))}\left( \zeta^{\kappa_c}_{\mC[[z]]}(\gog) \right)$ on which we will expand in Section \ref{sssec:moreongamma}.
\end{remark}

\subsubsection{Comparison of \texorpdfstring{$\Autpiu{} O$}{Aut O} actions}\label{ssec:comparisonautoactions}

In \cite{frenkel2007langlands} the construction $\Lie_{\mathrm{FF}}(V)$, whenever $V$ is a quasi-conformal vertex algebra (so it is equipped with operators $L_k$ which satisfy the axioms of \cite[6.2.4]{frenkel2007langlands}), comes equipped with a $\Der O = \mC[[z]]\partial_z$ action defined by the formula
\begin{equation}\label{eq:ffderoaction}
        (-f\partial_z) \cdot_{\text{FF}} [v \otimes g] = \sum_{k \geq -1} \frac{1}{(k+1)!} [ L_kv \otimes g\partial_z^{k+1}f],
\end{equation}
where $[v\otimes f] \in U(V) = (V \otimes \mC((z)))((\mathrm{Im}(T + \partial_z)))$ stands for the class of the element $v \otimes f \in V \otimes \mC((z))$. This induces an action on $\tilde{U}(V)$ as well.

Here, in the case where $V = V^\kappa(\gog)$, we show that the above action, when restricted to $\Der^0 O$, actually comes from an action of $\Autpiu{} O$ on $V^\kappa(\gog) \otimes \mC((z))$. We will do this by studying the action of coordinate changes which our constructions produce and then show that the differential of the latter coincides with the action of formula (\ref{eq:ffderoaction}).

\medskip

The constructions of \cite{casmaffei1} are canonical and don't need the choice of a coordinate $z$ to be defined, so let us use the following terminology: let $\calO$ be a topological ring over $\mC$ isomorphic to $\mC[[z]]$ (such as the completion of the local ring at any point on a smooth curve over $\mC$) and let $\calK \simeq \mC((z))$ be its field of fractions. The set of coordinates on $\calO$ (i.e. those elements $t \in \calO$ which induce an isomorphism $\calO \simeq \mC[[t]]$) is naturally an $\Autpiu{} O$-torsor, so that for any $\tau \in \Autpiu{} O$, the element $\tau(t) \in \calO$ is well defined and a coordinate. Let $\hat{\gog}_{\calK,\kappa}$ be the version of the affine algebra constructed replacing $\mC((z))$ with $\calK$ and let $U_\kappa(\hat{\gog}_\calK)$ be its completed enveloping algebra. In \emph{loc. cit.} a canonical chiral algebra over $\calO$, to be denoted in what follows by $V^\kappa_{\calO}(\gog) \subset \mF^1_{\calK,\gog} = \Homcont_{\mC}(\calK,U_\kappa(\hat{\gog}_{\calK}))$ is constructed. The choice of any coordinate $t \in \calO$ induces an isomorphism $\calY_t : V^\kappa(\gog) \otimes \calO \to V^\kappa_{\calO}(\gog)$, the natural extension $\calY_t : V^\kappa(\gog) \otimes \calK \to \mF^1_{\calK,\gog}$ remains injective. Let $\phi_t : \mC((z)) \to \calK$ denote the isomorphism induced by the choice of the coordinate $t$ and $\tilde{\calY}_t : V^\kappa(\gog) \otimes \mC((z)) \to \mF^1_{\calK,\gog}$ be the composition $\tilde{\calY}_t = \calY_t \circ (\id \otimes \phi_t)$. Via these isomorphisms we get a commutative diagram
\[\begin{tikzcd}
	{V^\kappa(\gog) \otimes\mC((z))} && {V^\kappa_{\calK,\gog}} \\
	\\
	{\Homcont_{\mC}\left(\mC((z)),U_\kappa(\hat{\gog})\right)} && {\mF^1_{\calK,\gog}}
	\arrow["{\tilde{\calY}_t}", from=1-1, to=1-3]
	\arrow["{\calY_{\mathrm{FF}}}"', hook, from=1-1, to=3-1]
	\arrow[hook, from=1-3, to=3-3]
	\arrow["{\ad_{\phi_t}}"', from=3-1, to=3-3]
\end{tikzcd}\]
Here $\calY_{\mathrm{FF}}$ is a slight modification of the usual state/field correspondence (see Section \ref{ssec:recollectionsysigma}); to be more precise, we have $\calY_{\mathrm{FF}}(v\otimes f(z)) = \left( g(z) \mapsto \Res_z\left(g(z)f(z)\cdot Y[v,z]\right)\right)$ (see point 3.b of Proposition \ref{rmk:comparisonmaplieenveloping}); while $\ad_{\phi_t}$ is the isomorphism induced by the identifications $\phi_t : \mC((z)) \to \calK$, $\phi^\gog_t : \tilde{U}_\kappa(\hat{\gog}) \to \tilde{U}_\kappa(\hat{\gog}_{\calK})$. The commutativity of the above diagram implies that, for any $\tau \in \Autpiu{} O$, the automorphism $\tilde{\calY}_{\tau} = \tilde{\calY}_{t}^{-1}\circ\tilde{\calY}_{\tau(t)} : V^\kappa(\gog) \otimes \mC((z)) \to V^\kappa(\gog) \otimes \mC((z))$ coincides with the restriction of action of $\tau \in \Autpiu{} O$ on $\Homcont_{\mC}\left(\mC((z)),U_\kappa(\gog)\right)$ by conjugation along the embedding $\calY_{\mathrm{FF}} : V^\kappa(\gog) \otimes \mC((z)) \hookrightarrow \Homcont_{\mC}\left(\mC((z)),U_\kappa(\gog)\right)$ (for more details see the argument applied in \cite[Lemma \ref{lem:groupactions}]{casmaffei1}).

Let us make an example on how this action by conjugation works: fields $f(z) \mapsto X \otimes f(z)$, attached to elements $X \in \gog$ are invariant; while the residue distribution $f(z) \mapsto \Res_z (f(z)dz)$ satisfies, for $\tau \in \Autzero{} O$, $(\tau \cdot \Res) = \Res \cdot (\partial_z\tau(z))$, where the latter is the distribution $f(z) \mapsto \Res(f(z)\partial_z\tau(z)dz)$. For an arbitrary element $v \otimes f(z) \in V^\kappa(\gog) \otimes \mC((z))$ we don't have a general formula though. What we can compute is the differentiation of this action, as the following proposition shows. 

\begin{proposition}
	Let $L^{\calY}_f$ be the operator on $V^\kappa(\gog) \otimes \mC((z))$ attached to $f\partial_z$ under the differentiation of the action of $\tilde{\calY}_{\tau}$. Then
	\begin{equation}\label{eq:ourderoaction}
		L^\calY_f(v \otimes g) = -\sum_{k\geq 0} \frac{1}{(k+1)!}(L_kv) \otimes g\partial_z^{k+1}f + v\otimes(\partial_z(fg)).
	\end{equation}
\end{proposition}

\begin{proof}
	Recall that by \cite[Lemma \ref{lem:groupactions} and Proposition \ref{prop:deroactionpsi}]{casmaffei1}, we have that $L^{\calY}_f$ is sesquilinear with respect $\mC((z))$, so that $L^{\calY}_f(v\otimes g) = g \cdot L^\calY_f(v \otimes 1) + v\otimes f\partial_zg$, and that
	\[
		[L^{\calY}_f,X_{(m)}] = m\sum_{k \geq 0} \frac{1}{(k+1)!}X_{(k+m)}\otimes \partial_z^{k+1}f,
	\]
	for any $X \in \gog$. In addition, it follows by $\tau \cdot \Res(\_dz) = \Res(\_dz) \cdot \partial_z\tau(z)$, that $L^{\calY}_f(\vac \otimes 1) = \vac \otimes \partial_zf$. 

	Let $\tilde{L}^{\calY}_f$ be the operator appearing in the right hand side of Equation (\ref{eq:ourderoaction}). It is easy to check that $\tilde{L}^{\calY}_f(v\otimes g) = g \cdot \tilde{L}^\calY_f(v \otimes 1) + v\otimes f\partial_zg$, that $L^{\calY}_f(\vac) = \tilde{L}^{\calY}_f(\vac)$ and that the same commutation relations  $[\tilde{L}^{\calY}_f,X_{(m)}]=[L^{\calY}_f,X_{(m)}]$ hold for any $X \in \gog$ (this follows by $[L_k,X_{(m)}] = -mX_{(k+m)}$). Since the operators $X_{(m)}$ generate $V^\kappa(\gog)$ it then follows that $L^{\calY}_f = \tilde{L}^{\calY}_f$ as desired.
\end{proof}

\begin{corollary}
	The $\Der^0 O$ action on $\Lie_{\mathrm{FF}}(V^\kappa(\gog))$ appearing in \cite{frenkel2007langlands} (see Equation (\ref{eq:ffderoaction})) comes from an action of the group $\Autpiu{} O$. The latter agrees with the action ${\calY}_\bullet$ of \cite[Lemma \ref{lem:groupactions}]{casmaffei1}. 
\end{corollary}

\begin{proof}
	We just need to check that
	\[
		[L^{\calY}_{-f}(v \otimes g)] = (-f\partial_z)\cdot_{\mathrm{FF}}[v \otimes g].
	\]
	This follows  by directly comparing Equations (\ref{eq:ffderoaction}) and (\ref{eq:ourderoaction}) and by the fact that in $\Lie_{\mathrm{FF}}(V^{\kappa}(\gog))$ we have $[L_{-1}v \otimes fg] = -[v \otimes \partial_z(fg)]$.
\end{proof}

\subsection{More on \texorpdfstring{$\gamma$}{gamma}}\label{sssec:moreongamma}

We also need to recall how the isomorphism $\gamma : \tilde{U}(\mC[\Op_{\gogl}(D)]) \simeq \mC[\Op_{\gogl}(D^*)]$ is constructed in \cite{frenkel2007langlands}. In order to do this we need to introduce some notations. When writing $\Op_{\gogl}(D)$ we think the disk $D$ with a specified coordinate $z$, so that $\calO_D = \mC[[z]]$ and $\calO_{D^*} = \mC((z))$. Recall that thanks to the canonical form of opers (see Proposition \ref{prop:art2descrlocopervcan}) there are isomorphisms of functors of $\mC$-commutative algebras
\[
    \chi = \chi_z : \Op_{\gogl}(D) \simeq J\check{V}^{\can}, \qquad \chi = \chi_z : \Op_{\gogl}(D^*) \simeq L\check{V}^{\can},
\]
where $J,L$ denote the Jet and Loop constructions, so that $J\check{V}^{\can}(R) = \Vcan \otimes R[[z]]$ and $L\check{V}^{\can} (R) = \check{V}^{\can}\otimes R((z))$. To describe functions on opers we describe functions on $J\check{V}^{\can},L\check{V}^{\can}$, notice that there is a natural closed embedding $J\check{V}^{\can} \subset L\check{V}^{\can}$ which induces a surjection $\Fun(L\check{V}^{\can}) \twoheadrightarrow \Fun(J\check{V}^{\can})$, the isomorphisms $\chi$ intertwine between this embedding and the natural embedding $\Op_{\gogl}(D) \subset \Op_{\gogl}(D^*)$.

Consider the residue pairing $\Res : \mC((z)) \times \mC((z))dz \to \mC$, an element $gdz \in \mC((z))dz$ and an element $h \in (\check{V}^{\can})^*$. Attached to this data there is a canonically defined function on $L\check{V}^{\can}$:
\[
    h\otimes gdz : L\check{V}^{\can} \to \mA^1 \qquad \check{V}^{\can} \otimes R((z)) \ni v \otimes f \mapsto h(v)\Res(fgdz) \in R.
\]
This allows us to construct a map $(\check{V}^{\can})^* \otimes \mC((z))\to \Fun( L\check{V}^{\can})$ which induces isomorphisms
\[
    \overline{\Sym}_{\mC}\left( (\check{V}^{\can})^* \otimes \mC((z)) \right) = \Fun (L\check{V}^{\can}), \qquad  \Sym\left( (\check{V}^{\can})^* \otimes \frac{\mC((z))}{\mC[[z]]} \right) = \Fun(J\check{V}^{\can}).
\]
Here $\overline{\Sym}_{\mC}\left( (\check{V}^{\can})^* \otimes \mC((z)) \right)$ stands for the completion of $\Sym\left( (\check{V}^{\can})^* \otimes \mC((z)) \right)$ along the topology generated by the ideals $\left((\check{V}^{\can})^*\otimes z^n\mC[[z]]\right)$. By considering the isomorphisms $\chi$ with opers, we get isomorphisms
\begin{align*}
    \chi^* &:\overline{\Sym}_{\mC}\left( (\check{V}^{\can})^* \otimes \mC((z)) \right) \simeq \mC[\Op_{\gogl}(D^*)], \\ \chi^* &: \Sym\left( (\check{V}^{\can})^* \otimes \frac{\mC((z))}{\mC[[z]]} \right) \simeq \mC[\Op_{\gogl}(D)].
\end{align*}
By taking the restriction of $\chi^*$, we get a natural embedding
\[
    \chi^*_{V} : (\check{V}^{\can})^* \to \mC[\Op_{\gogl}(D)] \qquad  h \mapsto \chi^* \left(h \otimes \frac{1}{z} \right),
\]
which induces a morphism
\[
	\chi^*_U : (\check{V}^{\can})^* \otimes \mC((z)) \to \mC[\Op_{\gogl}(D)]\otimes \mC((z)) \to \Lie_{\mathrm{FF}}(\mC[\Op_{\gogl}(D)]) \to \tilde{U}_{\mathrm{FF}}\left(\mC[\Op_{\gogl}(D)]\right).
\]
Since $\mC[\Op_{\gogl}(D)]$ is isomorphic, as a plain vertex algebra, to $V^0((\check{V}^{\can})^*)$ (and this isomorphism is compatible with the immersions of $(\check{V}^{\can})^*$ on both spaces) the above map upgrades to an isomorphism
\[
	\chi^*_U : \overline{\Sym}_{\mC}\left((\check{V}^{\can})^*\otimes\mC((z))\right) \to \tilde{U}_{\mathrm{FF}}\left(\mC[\Op_{\gogl}(D)]\right).
\]

Part of the Feigin-Frenkel Theorem is the following Proposition, which is claimed in \cite[Lemma 4.3.5]{frenkel2007langlands}. In the following statement we stress the importance of $\Aut O$ equivariance.

\begin{proposition}\label{prop:gammainvariantclassicalfeifre}
    The isomorphism
    \[
        \gamma_{\mathrm{FF}} = \chi^* \circ (\chi^*_{U})^{-1} : \tilde{U}_{\mathrm{FF}}\left( \mC[\Op_{\gogl}(D)] \right) \to \overline{\Sym}_A\left((\check{V}^{\can})^*\otimes\mC((z))\right) \to \mC[\Op_{\gogl}(D^*)]
    \]
    is $\Aut O$ equivariant.
\end{proposition}

\begin{proof}
    Also in this case one can prove that $\Aut O(\mC) = \Autpiu{} O (\mC)$ equivariance is equivalent to $\Der^0 O$ equivariance. The latter can be found for instance, in the particular case of $n=1$, in Proposition 6.3.1 of \cite{cas2023}.
\end{proof}

We move forward to compare the isomorphism $\gamma_{\mathrm{FF}}$ with our constructions. First, notice that via the identifications 
\begin{align*}
	&\Lie_{\mathrm{FF}}(\mC[\Op_{\gogl}(D)]) = \Lie_{\mC((z))}\left(\mC[\Op_{\gogl}(D)]\otimes\mC[[z]]\right) \\ &\tilde{U}_{\mathrm{FF}}\left(\mC[\Op_{\gogl}(D)]\right) = \calU_{\mC((z))}\left(\mC[\Op_{\gogl}(D)]\otimes\mC[[z]] \right)
\end{align*}
following the construction $\chi^*_U$ we may define an analogue of the isomorphism $$\chi^*_U : \overline{\Sym}_{\mC}\left( (\check{V}^{\can})^*\otimes\mC((z)) \right) \to \calU_{\mC((z))}\left(\mC[\Op_{\gogl}(D)]\otimes\mC[[z]]\right)$$ and translate Proposition \ref{prop:gammainvariantclassicalfeifre} as follows.
\begin{corollary}[Of Prop. \ref{prop:gammainvariantclassicalfeifre}]\label{coro:gammainvariantfeifrecasmaf}
    The isomorphism
    \[
        \tilde{\gamma}_{\mathrm{FF}} = \chi^* \circ (\chi^*_{U})^{-1} : \calU_{\mC((z))}\left( \mC[\Op_{\gogl}(D)]\otimes \mC[[z]] \right) \to \overline{\Sym}_A\left((\check{V}^{\can})^*\otimes\mC((z))\right) \to \mC[\Op_{\gogl}(D^*)]
    \]
    is $\Aut O$ equivariant, where the action of $\Aut O$ on the left hand side comes from the action of $\calY_\bullet$ of \cite[Lemma \ref{lem:groupactions}]{casmaffei1}.
\end{corollary}

We now move on and upgrade the above corollary to its coordinate free version, relating it to our chiral algebra construction as well. In order to do so we will need to consider both Lie algebras $\gog,\gogl$ at the same time.

Recall the notation we introduced in the second paragraph of Section \ref{ssec:comparisonautoactions}, so fix a complete topological ring $\calO$ (isomorphic to $\mC[[z]]$) and let $\calK$ be its field of fractions. Let $V^\kappa_{\calO}(\gog) \subset \Homcont(\calK,U_\kappa(\hat{\gog}_\calK))$ be our canonically defined chiral algebra and let $\zeta^{\kappa_c}_{\calO}(\hat{\gog})$ be its center. Consider $\Op_{\gogl}(\Spec \calK)$, the space of $\gogl$-opers over $\Spec \calK$; the choice of a coordinate $t \in \calO$ (i.e. a function which induces a topological isomorphism $\phi_t : \mC[[z]] \simeq \calO, (z \mapsto t)$) induces isomorphisms 
\begin{align*}
	\tilde{\calY}_t &: \mC[\Op_{\check{\gog}}(D)]\otimes\mC[[z]] = \zeta^{\kappa_c}(\hat{\gog}) \otimes \mC[[z]] \simeq \zeta^{\kappa_c}_{\calO}(\gog); \\
	\calU(\tilde{\calY}_t) &: \calU_{\mC((z))}\left( \mC[\Op_{\gogl}(D)]\otimes \mC[[z]] \right) \to \calU_{\calK}\left( \zeta_{\calO}^{\kappa_c}(\hat{\gog}) \right); \\
	\phi^{\Op}_t &: \mC\left[\Op_{\gogl}(\Spec \calK)\right] \to \mC\left[\Op_{\gogl}(D^*)\right].
\end{align*}
	
\begin{proposition}\label{prop:gammaisoFFnocoord}
	There exists a canonical isomorphism
	\[
		\gamma_{\calK} : \calU_{\calK} \left( \zeta^{\kappa_c}_{\calO}(\gog) \right) \to \mC\left[\Op_{\gogl}(\Spec \calK)\right].
	\]
	For any coordinate $t \in \calO$ the following diagram commutes:
	\[\begin{tikzcd}
		{\calU_{\calK}(\zeta^{\kappa_c}_{\calO}({\gog}))} && {\mC\left[ \Op_{\gogl}(\Spec \calK) \right]} \\
		\\
		{\calU_{\mC((z))}\left(\mC[\Op_{\gogl}(D)]\otimes\mC[[z]]\right)} && {\mC\left[\Op_{\gogl}(D^*)\right]}
		\arrow["\gamma_{\calK}", from=1-1, to=1-3]
		\arrow["{\calU_{\calK}(\tilde{\calY}_t)}", from=3-1, to=1-1]
		\arrow["{\phi_t^{\Op}}"', from=1-3, to=3-3]
		\arrow["{\tilde{\gamma}_{\mathrm{FF}}}", from=3-1, to=3-3]
	\end{tikzcd}\]
\end{proposition}

\begin{proof}
	All the maps $\calU_{\calK}(\tilde{\calY}_t),\tilde{\gamma}_{\mathrm{FF}}, \phi^{\Op}_t$ are isomorphisms, so, in order to prove both claims, it is enough to show that given any two coordinates $t,s$ the compositions \[(\phi_t^{\Op})^{-1}\circ \tilde{\gamma}_{\mathrm{FF}}\circ (\calU_{\calK}(\tilde{\calY}_t))^{-1} = (\phi_s^{\Op})^{-1}\circ \tilde{\gamma}_{\mathrm{FF}}\circ (\calU_{\calK}(\tilde{\calY}_s))^{-1} \] coincide. In order to do so write $s = \tau(t)$ for some $\tau \in \Autpiu{} O$ and rephrase the above equation as
	\[
		\tilde{\gamma}_{\mathrm{FF}} \circ \left( \calU_{\calK}(\tilde{\calY}_t^{-1}\tilde{\calY}_{\tau(t)} ) \right) = \left( \phi_t^{\Op}(\phi_{\tau(t)}^{\Op})^{-1} \right) \circ \tilde{\gamma}_{\mathrm{FF}}.	\]
	Finally notice that $\calU_{\calK}(\tilde{\calY}_t^{-1}\tilde{\calY}_{\tau(t)} )$ coincides with the action induced by $\calY_{\tau}$ of \cite[Lemma \ref{lem:groupactions}]{casmaffei1}, while $\phi_t^{\Op}(\phi_{\tau(t)}^{\Op})^{-1}$ coincides with the action of $\tau$ by coordinate changes on $\mC[\Op_{\gogl}(D^*)]$. After these remarks, the claim of the Proposition directly follows from Corollary \ref{coro:gammainvariantfeifrecasmaf}.
\end{proof}

We will need some more remarks about how the isomorphism $\gamma_{\calK}$ of Proposition \ref{prop:gammaisoFFnocoord} is related to the isomorphisms $\chi^*,\chi^*_U$ introduced at the beginning of Section \ref{sssec:moreongamma}. In order to do so, we need to introduce the analogues of $\chi^*,\chi^*_U$ in the $\calO,\calK$ setting. 
\begin{itemize}
	\item  We may rephrase the construction of $\chi^*_t$, obtaining an isomorphism
	\[
		\chi^*_t : \overline{\Sym}_{\mC} \left( (\check{V}^{\can})^*\otimes\calK\right) \to \mC[\Op_{\gogl}(\Spec \calK)].
	\]
	Here is a brief review on how to construct it: first consider the functor $L_\calK(\check{V}^{\can})^*(R) = \check{V}^{\can} \otimes (R\otimesr\calK)$. Then one can show that after the choice of a coordinate $t \in \calO$, the canonical form of opers (see Proposition \ref{prop:art2descrlocopervcan}) induces an isomorphism $\Op_{\gogl}(\Spec \calK) \simeq L_{\calK}(\check{V}^{\can})^*$, while  the residue pairing $\Res : \calK \times \calK dt \to \mC$ induces an isomorphism $\overline{\Sym}_{\mC}\left((\check{V}^{\can})^*\otimes \calK\right) \simeq \mC[L_{\calK}(\check{V}^{\can})^*]$ as in the fixed coordinate setting. Then $\chi^*_t$ is correct combination of these isomorphisms. Equivalently, $\chi^*_t$ may be constructed using the identification $\phi_t : \mC((z)) \simeq \calK$ induced by $t$, so that the following diagram commutes
	\[\begin{tikzcd}
	{\overline{\Sym}_{\mC}\left((\check{V}^{\can})^*\otimes\calK \right)} && {\mC\left[\Op_{\gogl}(\Spec \calK)\right]} \\
	\\
	{\overline{\Sym}_{\mC}\left((\check{V}^{\can})^*\otimes\mC((z)) \right)} && {\mC\left[\Op_{\gogl}(D^*)\right]}
	\arrow["{\chi^*_t}", from=1-1, to=1-3]
	\arrow["{\id\otimes\phi_t}", from=3-1, to=1-1]
	\arrow["{\phi^{\Op}_t}", from=1-3, to=3-3]
	\arrow["{\chi^*}"', from=3-1, to=3-3]
\end{tikzcd}\]
	Here $\phi_t^{\Op}$ is the isomorphism induced by the identification $\Spec \phi_t : \Spec \calK \to D^*$.
	\item To construct the analogue of $\chi^*_U$ notice that the map $\chi^*_V : (\check{V}^{\can})^* \to \mC[\Op_{\gogl}(D)]$ introduced in the above discussion induces a natural morphism \[(\check{V}^{\can})^*\otimes\calK \xrightarrow{\chi_V^*\otimes\id} \mC[\Op_{\gogl}(D)]\otimes\calK \to \Lie_\calK\left( \mC[\Op_{\gogl}(D)]\otimes \calO\right) = \frac{\mC[\Op_{\gogl}(D)]\otimes \calK}{\mathrm{Im} (T + \partial_t)}.\]
	Which, essentially as in the fixed coordinate case, upgrades to an isomorphism
	\[
		\chi^*_U : \overline{\Sym}_{\mC}\left( (\check{V}^{\can})^*\otimes\calK\right) \to \calU_{\calK}\left( \mC[\Op_{\gogl}(D)]\otimes \calO\right).
	\]
\end{itemize}

\begin{proposition}\label{rmk:descralphacenter}
    Let $\gamma_{\calK}$ be the isomorphism of Proposition \ref{prop:gammaisoFFnocoord}. Then the following diagram commutes
	\begin{equation}\label{eq:diagramgammacalychi}\begin{tikzcd}
	{\calU_{\calK}(\zeta^{\kappa_c}_{\calO}({\gog}))} && {\mC\left[ \Op_{\gogl}(\Spec \calK) \right]} \\
	\\
	{\calU_{\calK}\left(\mC[\Op_{\gogl}(D)]\otimes\calO\right)} && {\overline{\Sym}_{\mC}\left( (\check{V}^{\can})^* \otimes\calK \right)}
	\arrow["\gamma_{\calK}", from=1-1, to=1-3]
	\arrow["{\calU_{\calK}(\calY_t)}", from=3-1, to=1-1]
	\arrow["{\chi_t^*}"', from=3-3, to=1-3]
        \arrow["{\chi^*_{U}}", from=3-3, to=3-1]
	\end{tikzcd}\end{equation}
\end{proposition}

\begin{proof}
	Recall that by Proposition \ref{prop:gammaisoFFnocoord} the isomorphism $\gamma_{\calK}$ may be identified with the composition $(\phi^{\Op}_t)^{-1}\circ \tilde{\gamma}_{\mathrm{FF}} \circ (\calU_{\calK}(\tilde{\calY}_t))^{-1}$ and that the isomorphism $\tilde{\gamma}_{\mathrm{FF}}$ is given by the composition $\chi^* \circ (\chi^*_U)^{-1}$ as in Corollary \ref{coro:gammainvariantfeifrecasmaf}. Then to prove the Proposition it is enough to recall that the following diagrams commute:
	\[\begin{tikzcd}
	{\calU_{\calK}\left( \mC[\Op_{\gogl}(D)] \otimes \calO\right)} && {\overline{\Sym}_{\mC}\left((\check{V}^{\can})^* \otimes\calK \right)} \\
	\\
	{\calU_{\mC((z))}\left( \mC[\Op_{\gogl}(D)] \otimes \mC[[z]] \right)} && {\overline{\Sym}_{\mC}\left((\check{V}^{\can})^* \otimes\mC((z)) \right)}
	\arrow["{\chi^*_U}"', from=1-3, to=1-1]
	\arrow["{\calU(\id\otimes\phi_t)}", from=3-1, to=1-1]
	\arrow["{\Sym(\id \otimes\phi_t)}"', from=3-3, to=1-3]
	\arrow["{\chi^*_U}", from=3-3, to=3-1]
\end{tikzcd}\]
\[\begin{tikzcd}
	{\overline{\Sym}_{\mC}\left((\check{V}^{\can})^*\otimes\calK \right)} && {\mC\left[\Op_{\gogl}(\Spec \calK)\right]} \\
	\\
	{\overline{\Sym}_{\mC}\left((\check{V}^{\can})^*\otimes\mC((z)) \right)} && {\mC\left[\Op_{\gogl}(D^*)\right]}
	\arrow["{\chi^*_t}", from=1-1, to=1-3]
	\arrow["{\id\otimes\phi_t}", from=3-1, to=1-1]
	\arrow["{\phi^{\Op}_t}", from=1-3, to=3-3]
	\arrow["{\chi^*}"', from=3-1, to=3-3]
\end{tikzcd}\]
\end{proof}

\subsection{More on the topologies}

We will also need to make some remarks to state precisely how the topologies on $\mC[\Op_{\gogl}(\Spec \calK)]$ and on $Z_{\kappa_c}(\hat{\gog}_{\calK})$ (where $\hat{\gog}_{\calK}$ is the coordinate free version of the affine algebra built out of $\calK$) are related via the identification $\mC[\Op_{\gogl}(\Spec \calK)] = Z_{\kappa_c}(\hat{\gog}_{\calK})$ in a coordinate independent way. Even if we aim at a coordinate free description, we will compare the topologies via the isomorphism $\chi_t^* : \overline{\Sym}_{\mC}((\check{V}^{\can})^*\otimes \calK) \to \mC[\Op_{\gogl}(\Spec \calK)]$ which is induced by the choice of a coordinate and which should be thought of as giving a topological basis for $\mC[\Op_{\gogl}(\Spec \calK)]$. Recall that as discussed in Section \ref{ssec:canonicalrepresentativesopers} there exists a basis $x_l \in \check{V}^{\can}$ for which $[\frac{1}{2}\check{h}_0,x_l] = d_lx_l$, where $d_l +1$ are the exponents of $\check{\gog}$. We write $x_l^*$ for the corresponding dual basis.

\begin{definition}
	Before we relate the topologies we need to introduce some notation.
	\begin{enumerate}
		\item Given a non negative integer $N\geq 0$ we write $I_N = U_{\kappa_c}(\hat{\gog}_{\calK})(\gog\otimes\calO(-N))$ where $\calO(-N)$ is the $N$-th power of the maximal ideal of $\calO$; these ideals, by construction, define a basis for the topology of $U_{\kappa_c}(\hat{\gog}_{\calK})$. We write $I^Z_N = Z_{\kappa_c}(\hat{\gog}_{\calK})\cap I_N$;
		\item Given a non negative integer $N \geq 0$ we write $J_N \subset \overline{\Sym}_{\mC}((\check{V}^{\can})^*\otimes\calK)$ for the closed ideal generated by $\sum_l x_l^* \otimes \calO(-d_lN)$ so that the quotient $$\overline{\Sym}_{\mC}((\check{V}^{\can})^*\otimes\calK)/J_N = \Sym\left(\sum_l x_l^* \otimes (\calK/\calO(-d_lN))\right);$$
		\item Given a coordinate $t \in \calO$ we write $J_N^{\Op,t}$ for the image of the ideal $J_N$ along the isomorphism 
		$$\chi^*_t : \overline{\Sym}_{\mC}((\check{V}^{\can})^*\otimes\calK) \to \mC[\Op_{\gogl}(\Spec \calK)].$$
	\end{enumerate}
\end{definition}

\begin{proposition}\label{prop:comparisontopologiesonesingularity}
	With the above notation, we have that the Feigin-Frenkel isomorphism
	\[
		\mC\left[\Op_{\gogl}(\Spec \calK)\right] = Z_{\kappa_c}(\hat{\gog}_{\calK})
	\]
	identifies $J^{\Op,t}_N$ with $I^Z_N$.
\end{proposition}

\begin{proof}
	After identifying $\calK \simeq \mC((t))$, this can be found as a particular case of \cite[Corollaries 5.6.4, 5.6.5]{cas2023}. With respect to the notation of \emph{loc. cit.} we are in the case of $1$ singularity ($n=1$), specialized to $a_1 = 0$. In \emph{loc. cit.} the quotient $Z_{\kappa_c}(\hat{\gog}_{\calK})/I^Z_N$ is described as an algebra of polynomial algebra $\mC[P_{l,k}]_{k \in \mZ}/(P_{l,k_l})_{k_l \geq d_lN}$ (the "$j$" index of \emph{loc. cit.} disappears since we are in the $1$ singularity setting) and the $P_{i,k_i}$ are exactly the images of $x_l \otimes t^k$ along the map $\Psi^{V,t}_{\calK}$.
\end{proof}

\bibliography{biblio.bib}
\bibliographystyle{alpha}

\end{document}